\RequirePackage{luatex85}
\documentclass[10pt,a4paper,oneside,no-math,reqno]{amsart}
\usepackage[a4paper,margin=2.5cm]{geometry}
\usepackage[dvipsnames]{xcolor}
\usepackage[hypertexnames=false,colorlinks=true,linkcolor=blue,%
citecolor=purple,filecolor=magenta,urlcolor=cyan]{hyperref}                      
\usepackage{amssymb,anyfontsize,enumerate,layout,mathrsfs,mdwlist,stmaryrd,thmtools,tikz}
\usepackage{dsfont} 
\usepackage[textsize=footnotesize]{todonotes}
\presetkeys%
    {todonotes}%
    {color=Apricot}{}%
\usepackage{csquotes} 
\usepackage{fontspec}

\usepackage{polyglossia}
\setdefaultlanguage{english}
\usepackage{microtype}
\usepackage[style=alphabetic,maxnames=99,maxalphanames=5]{biblatex}
\renewbibmacro{in:}{}
\usepackage{mathtools}
\usepackage[english,noabbrev]{cleveref}
\setdefaultlanguage{english}

\allowdisplaybreaks 

\theoremstyle{plain}                          
\newtheorem{theorem}{Theorem}[section]
\newtheorem{proposition}[theorem]{Proposition}    
\newtheorem{lemma}[theorem]{Lemma}
\newtheorem{corollary}[theorem]{Corollary}
\theoremstyle{definition}
\newtheorem{definition}[theorem]{Definition}

\newtheorem{assumption}[theorem]{Assumption}
\newtheorem{prop-defin}[theorem]{Proposition-definition} 
\newtheorem{example}[theorem]{Example}
\theoremstyle{remark}
\newtheorem{remark}[theorem]{Remark}

\numberwithin{equation}{section}

\renewcommand{\theta}{\vartheta}
\renewcommand{\phi}{\varphi}
\renewcommand{\epsilon}{\varepsilon}

\newcommand{\mb}[1]{\mathbb{#1}} 
\newcommand{\mf}[1]{\mathfrak{#1}}
\newcommand{\mc}[1]{\mathcal{#1}}

\newcommand{\N}{\mb{N}} 
\newcommand{\C}{\mb{C}} 
\newcommand{\Z}{\mb{Z}} 

\renewcommand{\P}{\mb{P}}

\newcommand{\<}{\langle}
\renewcommand{\>}{\rangle}
\newcommand{\del}{\partial}

\DeclareMathOperator{\im}{Im}

\DeclareMathOperator{\Aut}{Aut}

\newcommand{\Id}{\mathord{\mathrm{Id}}}
\DeclareMathOperator*{\Res}{Res}

\DeclareMathOperator{\ad}{ad}

\newcommand{\lh}{\mathord{\mathrm{l\!h}}}

\begingroup\expandafter\expandafter\expandafter\endgroup
\expandafter\ifx\csname pdfsuppresswarningpagegroup\endcsname\relax
\else
\fi
{}

\begin{document}

\title[Failing to keep the balance]{Failing to keep the balance: explicit formulae and topological recursion for leaky Hurwitz numbers}
\author{Marvin Anas Hahn}
\address[M.~A.~Hahn]{School of Mathematics, 17 Westland Row, Trinity College Dublin, Dublin 2, Ireland}
\address[M.~A.~Hahn]{Hamilton Mathematics Institute, 16 Westland Row, Trinity College Dublin, Dublin 2, Ireland}
\email{hahnma@tcd.ie}

\author{Reinier Kramer}
\address[R.~Kramer]{Universit\`{a} di Milano-Bicocca, Dipartimento di Matematica e Applicazioni, Via Roberto Cozzi, 55, Milan, 20125, Italy}

\address[R.~Kramer]{INFN sezione di Milano-Bicocca, Milano, Italy}
\email{reinier.kramer@unimib.it}

\date{}


\begin{abstract}
    Recently a new family of enumerative invariants called \emph{leaky Hurwitz numbers} was introduced in the context of logarithmic intersection theory \cite{CMR25}. They admit an interpretation via tropical covers where the balancing condition fails, coining the term \emph{leaky}. 

     In the first part of this work, we study the polynomial behaviour of completed cycles Hurwitz numbers, generalising work of \cite{AKL25} and derive a different wall crossing that is cubic instead of quadratic. Moreover, we show how the recursive combinatorics of tropical covers are amendable to generatingfunctionology to find closed formulae for one--part and two--part completed cycles leaky Hurwitz numbers in genus $0$.

    Motivated by the notion of leaky Hurwitz numbers, we study the relation of $\tau$--functions of the KP hierarchy and topological recursion beyond the hypergeometric sector. More precisely, we study enumerative invariants arising from quite general, in particular non--diagonal, cut-and-join operators. 
    We find a new formalism, where these operators are obtained as Hamiltonian flows induced by a Hamiltonian function that arises from a specific dequantisation. In particular, under mild analytic constraints, the infinite general linear group acts on the space of (genus $0$) spectral curves via Hamiltonian flows.  Under these constraints, we find the appropriate spectral curves. Moreover, when the leakiness is fixed, i.e. when the matrix only contains non--zero entries along some off-diagonal, we show that this action is compatible with topological recursion, i.e. that the resulting enumerative invariants satisfy topological recursion. This provides a partial inverse to recent work of Alexandrov--Bychkov--Dunin-Barkowski--Kazarian--Shadrin \cite{ABDKS24a} producing differentials satisfying topological recursion for KP $\tau$-functions. In particular these results specialise to completed cycles leaky Hurwitz numbers.
    \end{abstract}

\maketitle

\tableofcontents

\section{Introduction}

\subsection{Background}
Hurwitz numbers were introduced \cite{Hur1879,Hur01} as counts of ramified covers of Riemann surfaces with given ramification profiles. They were soon found to be related to decomposition problems in symmetric groups and to the representation theory of these groups. In the last three decades, there has been a new surge of interest due to the realisation \cite{Pan00,Ok00} that certain kinds of Hurwitz numbers, called double Hurwitz numbers, can be assembled into generating functions which solve the 2D-Toda and Kadomtsev--Petviashvili (KP) integrable hierarchies. They are counts of ramified covers of $ \P^1$, with two specified, variable ramification profiles (hence \emph{double}). This realisation came from considerations in Gromov--Witten theory, and the connection between Hurwitz theory and Gromov--Witten theory has been a major pillar of research \cite{OP06,OP06a,GKLS25}. 

The KP hierarchy is an infinite series of partial differential equations in infinitely many variables, whose solution space is a semi-infinite Grassmannian \cite{SS83}, and as such has a natural action of an infinite general linear group, given by vertex operators. All equations in this hierarchy have terms with at least two derivatives, so linear functions are trivial solutions. Double Hurwitz numbers can be obtained from such trivial solutions by action of diagonal elements, called cut-and-join operators. Conversely, diagonal elements produce so-called \emph{hypergeometric solutions} \cite{orlov2001hypergeometric}. The formalism of \emph{weighted Hurwitz numbers} \cite{GPH17} extracts Hurwitz--type enumerations from solutions of hypergeometric type.

Hurwitz numbers are also an important class of examples of topological recursion. Topological recursion \cite{EO07} is a recursive way of construction a collection of symmetric multidifferentials $ \omega_{g,n}$ on powers of a (spectral) curve $ \Sigma$, i.e. $ \omega_{g,n}$ is defined on $ \Sigma^n$. The recursion is on $ 2g -2 + n$, which is the negative of the Euler characteristic of a genus $g $ curve with $n$ punctures, hence the name. Topological recursion has had strong connections to integrable hierarchies since its conception, because it encodes $\psi$-class intersection numbers (i.e. the Gromov--Witten theory of a point), which solve the Korteweg-de Vries hierarchy, and hence the KP hierarchy, by the Witten--Kontsevich theorem \cite{Wit91,Kon92}. Generating series of single simple Hurwitz numbers (with one of the two ramifications trivial, and the simplest cut-and-join operator), arranged by fixed genus $g$ of the covering surface and fixed length $n$ of the remaining ramification, give such multidifferentials $ \omega_{g,n}$. They were conjectured to satisfy topological recursion \cite{BM08} as a limit of a conjecture on the Gromov--Witten theory of toric Calabi--Yau threefolds \cite{BKMP09}. This was proved in \cite{EMS11}, and later generalised to many other kinds of Hurwitz numbers, notably the completed cycles case relevant for the Gromov--Witten/Hurwitz correspondence of \cite{OP06,OP06a} in \cite{DKPS23}. Conversely, all solutions to topological recursion fit into the framework of cohomological field theories \cite{Eyn14}.  In particular, all invariants solved by topological recursion may be expressed as Gromov--Witten invariants akin to the celebrated ELSV formula \cite{ELSV01}. A structural application is that therefore all such invariants are quasi-polynomial.

In a string of papers, Bychkov, Dunin-Barkowski, Kazarian, and Shadrin, later joined by Alexandrov, systematised the theory of topological recursion for Hurwitz numbers \cite{BDKS22,BDKS24}, and used their insights to prove many structural properties of topological recursion \cite{BDKS23,BDKS25,ABDKS24,ABDKS24a,ABDKS25universal}, extending the definition \cite{ABDKS25}, and showing that the collection of multidifferentials, expanded at any non-special point of $\Sigma$ in any local coordinate, yields a KP $\tau$-function \cite{ABDKS24KP0,ABDKS24KP+,ABDKS25blobbed}, depending on an extra formal parameter $ \hbar$ tracking the negative Euler characteristic $ 2g-2+n$.

Double Hurwitz numbers are closely related to the intersection theory of $\overline{\mathcal{M}}_{g,n}$ \cite{GV05,GJV05}. An instance of this relation is the celebrated ELSV formula  \cite{ELSV01} for \emph{single Hurwitz numbers}, i.e. the case of one ramification profile trivial. A generalisation for double Hurwitz numbers via topological recursion may be found in \cite{BDKLM20}. 

Another major point of view in Hurwitz theory is provided by the field of \emph{tropical geometry}, a relatively new field of mathematics at the crossroads of combinatorics and algebraic geometry. It provides the notion of \emph{tropicalisation} that provides a dictionary between combinatorics and algebraic geometry. In particular, maps between Riemann surfaces tropicalise to maps between graphs called \emph{tropical covers}. In \cite{CJM10} this perspective was employed to derive an expression of double Hurwitz numbers as a weighted enumeration of tropical covers. This interpretation arises naturally from tropicalising the moduli space of admissible covers \cite{cavalieri2016tropicalizing}. On the other hand, Wick's theorem allows a tropicalisation via the Fock space that produces the same result \cite{zbMATH07272609}. Based on previous calculations in \cite{HL22,HL20} the full weighted Hurwitz numbers problem was tropicalised in \cite{hahn2025tropicalising}. A major strength of the tropical approach is that the aforementioned polynomiality of double Hurwitz numbers is a consequence of the combinatorial structure of tropical covers \cite{CJM11}.

\subsection{Adding leakiness}
As outlined above, double Hurwitz numbers are intimately related to the intersection theory on $\overline{\mathcal{M}}_{g,n}$. In particular, there is a close conceptual relation to the double ramification cycle:  the locus of curves with map to $ \P^1$ with given ramification over $ 0$ and $ \infty$ \cite{GV05}. As it turns out, this relationship is quite subtle and -- relying on new advances in logarithmic enumerative geometry \cite{HPS19} -- was only recently made precise in \cite{CMR25} using a logarithmic enhancement of the double ramification cycle. More precisely, \cite{CMR25} proved that Hurwitz numbers may naturally be obtained as an intersection product of logarithmic double ramification cycle with a branch polynomial obtained from tropical combinatorics. This perspective naturally places Hurwitz numbers into a family of enumerative invariants indexed a non--negative integer parameter $k$, by replacing the double ramification cycle with its pluricanonical generalisation. These numbers were dubbed \emph{leaky Hurwitz numbers} in \cite{CMR25}. The terminology arises from their tropical interpretation. A crucial feature of tropical covers is a \emph{balancing condition} at each vertex. Leaky Hurwitz numbers are computed by leaky tropical covers, where the balancing condition fails exactly by $k$ at each vertex.

While this generalisation loses the geometric intuition of counting covers of $ \P^1$, as the degree is not even constant anymore, it turns out to be very natural from the integrability perspective. In contrast to the classical setting described above, leaky Hurwitz numbers  have \emph{non-diagonal} cut-and-join operators; and more specifically the `leakiness' is the distance from the diagonal of the cut-and-join operator. We are therefore firmly in the non--hypergeometric sector of Hurwitz theory. While \cite{CMR25} considered only a very particular subset of cut-and-join operators coming from their log double ramification analysis, from the point of view of combinatorics and integrability there is no reason for such a restriction. E.g. \cite{AKL25} introduced a leaky analogue of the important class of \emph{completed cycles Hurwitz numbers} (that appear in the stationary sector of Gromov--Witten theory \cite{OP06a}), and proved that they admit a piecewise polynomial structure akin to their non--leaky origins \cite{SSZ12}. Moreover, \cite{cavalieri2025k,CMS25} introduced a descendant version of leaky Hurwitz numbers with $\psi$--class insertions. Among others, they derived closed formulae for one--part leaky Hurwitz numbers, i.e. one ramification profile is the one--part partition, while the other is the trivial partition, in genus $0$ and genus $1$.

\subsection{Our results}
Our work may be separated into two parts. In the first part, we study the specific case of completed cycles leaky Hurwitz numbers through the lens of tropical geometry. To begin with, we derive a tropical interpretation of completed cycles leaky Hurwitz numbers in \cref{sec-tropcov}. Using this interpretation, we study the piecewise polynomial structure of completed cycles leaky Hurwitz numbers in \cref{PolyWC}. In a first step, we prove a piecewise polynomiality result generalising the one found before in \cite{AKL25}. Similar to \emph{loc.cit.}, we obtain a universal polynomiality result, i.e. the polynomiality extends to the leakiness parameters. Similar considerations were found in tropical enumerative geometry of toric surfaces before \cite{ardila2012universal,hahn2024universal}. 

Moreover, we study their wall--crossing behaviour, i.e. we consider two adjacent chambers of polynomiality and express the difference of polynomials in terms of completed cycles leaky Hurwitz numbers with smaller input data. A wall--crossing result for disconnected leaky Hurwitz numbers was previously derived in \cite{AKL25} using Fock space techniques. Our tropical methods allow the derivation of different formulae for connected invariants.

Further, motivated by the initial cases of topological recursion and the work in \cite{CMS25}, we explore closed formulae for an orbifold version of completed cycles leaky Hurwitz numbers in the cases $(0,1)$ and $(0,2)$, i.e. genus $0$ and ramification partitions of length $1$ and $2$. We refer to \cref{ClosedFormulaeSec} for precise definitions. We obtain closed formulae for the full enumerative problem for these cases. Our formula for the $(0,1)$ case recovers the formula in the same setting obtained in \cite{CMS25}. We consider the methodology the most interesting part of this section. In particular, we observe that the recursive cut--and--join analysis on tropical covers may be phrased as a functional equation. Generatingfunctionology, in particular Lagrange inversion and similar methods, then gives tools to solve this equation. A direct calculation reveals closed formulae for $(0,1)$, while residue calculations produce formulae for $(0,2)$.

From \cref{CJSection} forward, we explore the relation between topological recursion and $\tau$--functions of the KP hierarchy in large generality. More precisely, we study the question: \emph{How does the action of the infinite general linear group on $\tau$--functions of the KP hierarchy materialise in topological recursion?}

To answer this question, we consider enumerative problems arising from a very general cut-and-jon operator: we allow it to depend on $ \hbar$, and at first, we only require that it be lower triangular. The resulting invariants present a vast generalisation of (leaky) Hurwitz numbers. While previous work in topological recursion has focused on the hypergeometric sector of the KP hierarchy, this set-up moves beyond it by allowing the cut--and--join operator to be non--diagonal.

We decompose this operator as a linear combination of the leaky completed cycles operators of \cite{AKL25}, and pack the coefficients into a two-variable generating series. We show how a particular dequantisation, i.e. leading order in a certain expansion with respect to the quantum parameter $\hbar$, of this generating series is a Hamiltonian function which, under some analyticity assumption, generates the spectral curve for the Hurwitz numbers multidifferentials. Moreover, we show that the higher orders are governed by a Moyal star product deformation, and we find closed algebraic formulae for the multidifferentials under some conditions for the cut-and-join operator.

This gives a natural interpretation of the symplectic structure of topological recursion: we explicitly generate all data of topological recursion through Hamiltonian, hence symplectic, flow. In particular, we realise $ \omega_{0,1}$ directly as an integral of the standard symplectic form on $ \C^* \times \C$.

Our methods show how vertex operators, i.e. the group action generating the space of KP $ \tau $-functions, act on the space of spectral curves (of genus $0$) in a natural way. This is a partial converse to the results of \cite{ABDKS24KP0}: while they find $\tau$-functions from spectral curves through topological recursion, we generate spectral curves, and hence the genus-$0$ sector of topological recursion, from $\tau$-functions.

The procedure described above loses information, as it only keeps the leading order in $\hbar$. However, the Hurwitz numbers, and hence their associated multidifferentials, depend on the entire $ \hbar$-series. As topological recursion uniquely reconstructs multicorrelators from the spectral curve, there is at most one cut-and-join operator with given dequantisation whose Hurwitz numbers can be found trough topological recursion.

For a fixed positive leakiness, i.e. a cut-and-join operator with a single non-zero diagonal below the main diagonal, we write the spectral curve explicitly, and we show that topological recursion generically holds if the operator coincides with its leading order, in terms of quantisation. The most natural class of examples, with the cut-and-join operator a leaky completed cycles operator and the secondary point an orbifold ramification, all fall in this generic class. This is compatible with the $ k \to 0$ limit, even though in this case the usual generating series, the symplectic duality function $ \hat{\psi}$, does have non-trivial $\hbar$-corrections: our generating series is different, and we claim more natural.\par
Our proof for topological recursion hinges on the exact shape of the spectral curves we obtain, as we can construct these through a series of successive $x$--$y$ dualities and symplectic dualities.

\subsection{Future work}

In this work, we explore a new, but very natural, extension of Hurwitz numbers, particularly in connection with tropical geometry, integrability, and topological recursion. We have found that this extension yields new insights in these connections, that also shed light on the previously establishes relations. However, we do not prove these extended connections in full generality, and we also do not address certain other commonly related topics.

As such, we leave many open questions, including the following.

\begin{itemize}
    \item In case the leakiness is allowed to be negative, Hurwitz numbers become infinite counts. We would probably need to keep the flow parameter $ \beta$ as a (formal) parameter in the Hurwitz numbers and the spectral curve to manage this, but even then our approach runs into trouble, as such cut-and-join operators would not annihilate the covacuum of Fock space.
    \item Our proof of the topological recursion result requires iterated use of $ x$--$y$ swap and symplectic duality. We would like to see a more conceptual proof that really shows the `symplectic invariance' of topological recursion, not necessarily starting from a linear $ \tau$-function, and coming from the Hamiltonian flows. This should probably involve the Moyal star product.
    \item We only give a proof of topological recursion for a single non-zero diagonal, as our method is dependent on the specific shape the spectral curve has is this case. A more conceptual proof as above should hopefully also generalise to more non-zero diagonals.
    \item We expect the Hamiltonian flow from the cut-and-join operators to act similarly on all spectral curves (or at least those of genus zero), not just the trivial ones we start from for Hurwitz numbers.
    \item We would also like to see the $x$--$y$ swap itself as a particular implementation of this flow for an appropriate Hamiltonian. The first guess for such a Hamiltonian would be $ H = \frac{\pi}{2} ( \log(z)^2 + s^2 )$, but this is not a power series in $z$, so it does not readily fit our framework.
    \item As of now, the geometric interpretation of leaky Hurwitz numbers is not so clear. A more intuitive interpretation than the one through the log pluricanonical double ramification cycle would be interesting. We mention some possible interpretation in \cref{rem-twistedGw-H}.
    \item It would be interesting to know whether there is a leaky version of the weighted Hurwitz numbers formalism akin to \cite{GPH17}. In other words, is there a natural enumerative meaning underlying the base change of the cut--and--join operator from the Schur functions to symmetric power sums? For $k=0$ this arises from the eigenvalues of the cut--and--join operator. For $k>0$, there are no more eigenvalues and the calculations become more complicated.
    \item Topological recursion is strongly linked \cite{DOSS14} to intersection numbers of Cohomological Field Theories (CohFTs) with $\psi$-classes, and especially in the case of Hurwitz numbers, these CohFTs are very natural objects -- the Hodge class and its generalisation the Chiodo class are natural examples. Because of the general theory, CohFTs certainly exist also for leaky Hurwitz numbers, but we do not know if they are `natural' examples coming from geometric constructions of independent interest.
    \item In light of the recent work on non--orientable Hurwitz numbers via Jack functions \cite{chapuy2022non}, it would be interesting to explore a $b$--refinement of leaky Hurwitz numbers as well. The tropical interpretation in \cite{fesler2025refined} may be a natural starting point.
\end{itemize}

\subsection{Organisation of the paper}

In \cref{Preliminaries}, we give background information on Hurwitz theory, and in particular the aspects that we want to cover: the expression of Hurwitz numbers in Fock space through cut-and-join equations, as well as integrability, in \cref{FockSpace}, the interpretation as counts of tropical covers in \cref{sec-tropcov}, and we give an introduction to topological recursion including modern tools in \cref{TR}.

In \cref{PolyWC} we use tropical techniques to prove chamber polynomiality and wall crossing of leaky double Hurwitz numbers in, respectively, \cref{thm-piecepoly,thm-WC}.

From \cref{ClosedFormulaeSec} on, we study double Hurwitz numbers asymmetrically: we keep one ramification variable, and have the secondary ramification profile be governed by a generating series. \Cref{ClosedFormulaeSec} considers the case of monomial generating series, i.e. fixed leakiness, completed cycles, and orbifold numbers. In this setting, we find explicit formulae for genus zero one-part and two-part leaky Hurwitz numbers obtained through tropical reasoning.

In \cref{CJSection}, we use the Fock space formalism to interpret the cut-and-join operator to leading order as a Hamiltonian function. We show in \cref{GeneralSC} that this flow, along with the flow associated to the secondary ramification generating series, generates the spectral curve of the theory.

\Cref{ExplicitSection} consist of two parts: in \cref{ExplicitMultidifferentials}, we generalise the work of \cite{BDKS22} to give explicit closed algebraic formulas for multidifferentials of leaky Hurwitz numbers, in \cref{LeakyMultidifferentials}. This section hardly contains any proofs, as nearly all proofs generalise directly from \cite{BDKS22}. Then, in \cref{CurvesForFixedk} we restrict to a fixed leakiness to write down the spectral curve explicitly in terms of the given generating functions for the problem, \cref{SpectralCurveShape}.

Finally, in \cref{sec:TR} we show that, still for fixed $k >0$, the multidifferential generating series of Hurwitz numbers are calculated by topological recursion on the spectral curves found before if the cut-and-join generating function coincides with its dequantisation, and both this and the secondary ramification generating series are rational \cref{LeakyTRThm,LeakyTRLimit}, at least generically. This generic situations contains all `monomial' cases: a leaky cut-and-join operator with an orbifold secondary ramification.
The proof identifies the spectral curve as generated by a series of four $ x$--$y$ swaps and three symplectic dualities, and shows the multidifferentials obtained from this series of operations coincide with the Hurwitz generating series.

\subsection{Acknowledgments}

We would like to thank Renzo Cavalieri, Max Karev, Danilo Lewański, Hannah Markwig, Dhruv Ranganathan, Johannes Schmitt, and Sergey Shadrin for stimulating discussions on various aspects of this paper, and the anonymous referee of the MATRIX Annals article announcing our results for stimulating questions. R.K. would also like to thank Maxim Kazarian for his very enlightening series of talks at the MATRIX program `Enumerative Geometry, Integrability, and Topological Recursion', and the MATRIX itself and the program organisers for making this happen.

R.K. acknowledges support from the National Science and Engineering Research Council of Canada. R.K. is partially supported by funds of the Istituto Nazionale di Fisica Nucleare, by IS-CSN4 Mathematical Methods of Nonlinear Physics. R.K. is also thankful to GNFM (Gruppo Nazionale di Fisica Matematica) for supporting activities that contributed to the research reported in this paper.

The first version of this preprint was prepared without any use of AI. For the second version, the authors used Gemini and ChatGPT to check the results in \cref{QuasiPoly} as well as the results related to the Moyal star product in \cref{CJSection} and certain convergence results in \cref{ExplicitSection}. The authors take full responsibility for all results in the paper.

No data were generated for this project.

\section{Preliminaries}\label{Preliminaries}

\subsection{Notation}

We will use the notation $ \llbracket n \rrbracket \coloneqq \{ 1, \dotsc, n\}$, and let this thread as index, e.g. $ X_{\llbracket n \rrbracket} = \{ X_1, \dotsc, X_n \}$.

We also define the functions
\begin{equation}
    \varsigma(z)
    =
    2\sinh(z/2) , 
    \qquad 
    \mathcal{S}(z)
    =
    \varsigma(z)/z .
\end{equation} 
Moreover, we denote by $\mathbb{Z}' = \Z + \frac{1}{2}$ the set of half--integers.

For $\lambda$ a partition of $d$, we denote the character of the corresponding representation of $S_d$ by $\chi_\lambda$. Moreover, we denote the conjugacy class corresponding to a partition $\mu\vdash d$ by $C_\mu$ and define the central character

\begin{equation}
    f_\mu(\lambda)
    =
    \frac{|C_\mu|}{\mathrm{dim}(\lambda)} \chi_\lambda(\mu).
\end{equation}

We extend the notion of a central character to arbitrary partitions $\mu$ and $\nu$ via

\begin{equation}
    f_\mu(\lambda)
    =
    \binom{|\lambda|}{|\mu|} \frac{|C_\mu|}{\mathrm{dim}(\lambda)} \chi_\lambda(\mu).
\end{equation}

If $|\mu|<|\lambda|$, there is a natural inclusion $S_{|\mu|}\subset S_{|\lambda|}$ which allows us to define $\chi_\lambda(\mu)$. If $|\mu|>|\lambda|$ the binomial coefficient vanishes.
 We denote the shifted symmetric power sums 
 \begin{equation}
     \mathbf{p}_k(\lambda)
     =
     \sum_{i=1}^\infty \left( (\lambda_i-i+\frac{1}{2})^k-(-i+\frac{1}{2})^k\right) + (1-2^{-k}) \zeta(-k).   
 \end{equation}

 For a partition $\mu$, we define
 \begin{equation}
     \mathbf{p}_\mu
     =
     \prod \mathbf{p}_{\mu_i}.
 \end{equation}

\subsection{(Leaky) Hurwitz numbers}
We begin with the definition of the ring of shifted symmetric functions. Let $n$ be a positiver integer. We define an action of $S_n$ on $\mathbb{Q}[\lambda_1,\dots,\lambda_n]$ by permutation of the shifted variables $\lambda_i-i$. We denote the ring of invariants by

\begin{equation}
    \mathbb{Q}[\lambda_1,\dots,\lambda_n]^{(S_n)}.
\end{equation}

We have natural filtration by degree and natural maps along $n$ by setting the last variable to $0$. Thus, we may define the ring of shifted symmetric functions $\Lambda^\ast$ in infinitely many variables via the following projective limit

\begin{equation}
    \Lambda^\ast
    =
    \lim_{\leftarrow} \mathbb{Q}[\lambda_1,\dots,\lambda_n]^{(S_n)}.
\end{equation}

By work of Kerov and Olshanski \cite{Ker94}, we have $f_\mu(\lambda)\in\Lambda^\ast$. In fact, the central characters $f_\mu$ form a basis when ranging over all partitions $\mu$. Moreover, by construction $p_k\in\Lambda^\ast$ and $\Lambda^\ast$ is freely generated by them.

We now define double Hurwitz numbers.

\begin{definition}
Let $g\ge0$, $d>0$ integers and $\mu,\nu$ partitions of $d$. Denote $b=2g-2+\ell(\mu)+\ell(\nu)$ and fix pairwise distinct points $q_1,\dots,q_b\in\mathbb{P}^1$. We call a holomorphic map $f\colon S\to\mathbb{P}^1$ with
\begin{itemize}
    \item $S$ is a compact Riemann surface,
    \item $f$ has ramification profile $\mu$ over $0$, $\nu$ over $\infty$ and $(2,1,\dots,1)$ over $q_i$,
\end{itemize}
a branched covering of type $(g,\mu,\nu)$. We call two such covers $f_1\colon S_1\to\mathbb{P}^1$ and $f_2\colon S_2\to\mathbb{P}^1$ equivalent if there exists an isomorphism $g\colon S_1\to S_2$ with $f_1=f_2\circ g$.

Then, we define the (not-necessarily connected) double Hurwitz number

\begin{equation}
    h^\bullet_{g;\mu,\nu}
    =
    \sum_{[f]} \frac{1}{|\mathrm{Aut}(f)|},
\end{equation}
where the sum runs over all equivalence classes of branched covers of type $(g,\mu,\nu)$. When we additionally impose that $S$ is connected, we denote the connected double Hurwitz numbers by $h^\circ_{g;\mu,\nu}$.
\end{definition}

We note that when there is there no fear of confusion or when we do not want to specify it, we drop the notation of $\circ$ and $\bullet$ to distinguish between connected and not--necessarily connected numbers.

Hurwitz numbers are closely related to the representation theory of the symmetric group. In particular, we have

\begin{equation}
    h^\bullet_{g;\mu,\nu}
    =
    \frac{1}{\prod \mu_i\prod\nu_j} \sum_{\lambda\vdash d} \chi_\lambda(\mu) \chi_\lambda(\nu) f_{(2)}(\lambda)^r.
\end{equation}

This motivates the definition of completed cycles Hurwitz numbers of \cite{OP06a}.

\begin{definition}\label{CCHNdef}
    Let $g\ge0$, $d>0$ integers and $\mu,\nu$ partitions of $d$. Further, let $r_1,\dots,r_n$ be positive integers with $\sum r_i=2g-2+\ell(\mu)+\ell(\nu)$ and denote $\underline{r}=(r_1,\dots,r_n)$. Then, we define completed cycles double Hurwitz numbers

    \begin{equation}
        h^{\underline{r},\bullet}_{g;\mu,\nu}
        =
        \frac{1}{\prod \mu_i\prod\nu_j} \sum_{\lambda\vdash d} \chi_\lambda(\mu) \chi_\lambda(\nu) \prod_{i=1}^n \frac{\mathbf{p}_{r_i}(\lambda)}{(r_i-1)!}.
    \end{equation}
    Via the inclusion-exclusion principle, we obtain connected completed cycles double Hurwitz numbers $h_{g;\mu,\nu}^{\underline{r},\circ}$.
\end{definition}

We have $h_{g;\mu,\nu}^{(2^n)} = h_{g;\mu,\nu}$, since $f_{(2)}=\mathbf{p}_2$ by the Frobenius character formula, i.e. completed cycles Hurwitz numbers do recover the classical definition. For $r_i \ge 3 $ completed cycles Hurwitz numbers are still related to classical Hurwitz numbers where we require the ramification profile of $q_i$ to be $(r_i,1\dots,1)$, although there are additional contributions. The shifted symmetric functions $\mathbf{p}_i$ ``complete" the corresponding central characters $f_{(i)}$ in this sense. In seminal work \cite{OP06a}, Okounkov and Pandharipande gave a natural interpretation of completed cycles Hurwitz numbers in Gromov--Witten theory.

We describe the set-up briefly and refer to \cite{OP06a} for details. Fix two partitions $\mu$ and $\nu$ of a positive integer $d$. Let $\overline{\mathcal{M}}_{g,n}(\mathbb{P}^1,\mu,\nu)$ be the moduli space of relative stable maps to $\mathbb{P}^1$ with ramification $\mu$ over $0$ and $\nu$ over $\infty$. We define for $i\in[n]$ the corresponding $\psi$--class
\begin{equation}
    \psi_i
    =
    c_1(\mathbb{L}_i),
\end{equation}

where $\mathbb{L}_i$ is the $i$-th cotangent bundle on $\overline{\mathcal{M}}_{g,n}(\mathbb{P}^1,\mu,\nu)$. Moreover, let $\mathrm{ev}_i\colon\overline{\mathcal{M}}_{g,n}(\mathbb{P}^1,\mu,\nu)\to\mathbb{P}^1$ evaluating the $i$-th marked point. Let $\omega$ be the Poincaré dual of the point class. Then we define for $r_1,\dots,r_n\ge0$ as

\begin{equation}\label{GWdef}
    \left\langle \mu \mid \tau_{r_1}(\omega) \cdots \tau_{r_n}(\omega) \mid \nu \right\rangle_{g,n}^{\mathbb{P}^1}
    =
    \int_{[\overline{\mathcal{M}}_{g,n}(\mathbb{P}^1,\mu,\nu)]^{\textrm{vir}}}\prod_{i=1}^n\mathrm{ev}_i^\ast(\omega)\psi_i^{r_i}.
\end{equation}

When we restrict to connected stable maps, we add a superscript $\bullet$, where for possibly disconnected stable maps, we add $\circ$.

We have the following theorem, which is an incarnation for Gromov--Witten/Hurwitz correspondence of \cite{OP06a}.

\begin{theorem}\label{GW/H}
    We have
    \begin{equation}
        h^{\underline{r},\bullet}_{g;\mu,\nu}
        =
        \left\langle\mu\mid \tau_{r_1-1}(\omega)\cdots\tau_{r_n-1}(\omega)\mid\nu\right\rangle_{g,n}^{\mathbb{P}^1}.
    \end{equation}
\end{theorem}

Fix non--negative integers $g$, $k$, a positive integer $n$ and $\mathbf{x}\in\mathbb{Z}^n$ an integer vector with $\sum \mathbf{x}_i=k(2g-2+n)$. We consider the double ramification locus

\begin{equation}
    \mathrm{DRL}_g(\mathbf{x})=\left\{(C,q_1,\dots,q_n)\mid (\omega_C^{\mathrm{log}})^{\otimes k}\cong\mathcal{O}\left(\sum_{i=1}^n\mathbf{x}_iq_i\right)\right\}\subset {\mathcal{M}}_{g,n},
\end{equation}

where $\omega_C^{\textrm{log}}$ denotes the logarithmic canonical bundle of $C$. We drop the dependence on $k$ in the notation since $k$ can be uniquely reconstructed from $g$ and $\mathbf{x}$.

Note that for $k=0$ the condition becomes that $\mathcal{O}\left(\sum_{i=1}^n\mathbf{x}_iq_i\right)$ is trivial, i.e. there exists a branched covering $C\to\mathbb{P}^1$ with prescribed ramification over $0$ and $\infty$. The (virtual) fundamental class of $\mathrm{DRL}_g(\mathbf{x})$ may be extended to the \textit{double ramification cycle}
\begin{equation}
    \mathrm{DR}_g(\mathbf{x})\in\mathrm{CH}_{2g-3+n}(\overline{\mathcal{M}}_{g,n}).
\end{equation}

While it seems natural that double Hurwitz numbers should appear as an intersection product of $\mathrm{DR}_g(\mathbf{x})$ with a natural class on $\overline{\mathcal{M}}_{g,n}$, there is no obvious way to do so. In a series of works in recent years \cite{Hol21,MW20}, this failure to express double Hurwitz numbers has been identified to emerge from $\mathrm{DR}_g(\mathbf{x})$ ``living in the wrong space". This insight motivated the introduction of the \textit{logarithmic double ramification cycle} $\mathrm{logDR}_g(\mathbf{x})$ obtained via a limit of pull-backs of iterated smooth log-blowups of $\overline{\mathcal{M}}_{g,n}$.

It was proved in \cite{CMR25} that double Hurwitz numbers $H_g(\mu,\nu)$ may be indeed obtained as an intersection product with $\mathrm{logDR}_g(\mu,-\nu)$ for $k=0$. Moreover, the same construction provides a whole new family of enumerative invariants $H_g^k(\mu,\nu)$ for arbitrary integers $k$ called \textit{leaky Hurwitz numbers}. We give equivalent definitions first via the Fock space formalism in \cref{FockSpace} and in terms of tropical geometry in \cref{sec-tropcov}.

\subsection{Fock space formalism}\label{FockSpace}

Fock space is an important representation of several algebras.
Most importantly for us, the Lie algebra $ \widehat{\mf{a}}_\infty$ is a central extension of the infinite general linear algebra\footnote{Defining this properly requires some care, to avoid divergences. We will require that our infinite matrices $ \{ a_{ij} \}_{i,j \in \Z'} $ have only finitely many non-zero diagonals above the main diagonal, i.e. $ a_{ij} = 0 $ for $ i \ll j$.} with generators $ \{ E_{ij} \, \mid \, i, j \in \Z ' \} \cup \{ c\}$, commutation relations
\begin{equation}
    [ E_{ij}, E_{kl} ]
    =
    \delta_{j,k} E_{il} - \delta_{i,l} E_{kj} + \delta_{i,l} \delta_{k,j} (\delta_{j > 0} - \delta_{l > 0} ) c
\end{equation}
and $c$ central. The algebra $ \widehat{\mf{a}}_\infty$ includes the following operators, for $ n \in \Z $ and a formal variable $u$,
\begin{equation}
    \mathcal{E}_n(u)
    =
    \sum_{l\in\mathbb{Z}'}e^{u(l-n/2)}E_{l-n,l}+\frac{\delta_{n,0}}{\varsigma(u)} c 
    \in
    \widehat{\mf{a}}_\infty \llbracket u \rrbracket .
\end{equation}
We recall their well-known commutators
\begin{equation}
    [\mathcal{E}_m(u),\mathcal{E}_n(v)]
    =
    \varsigma\left(\det \begin{pmatrix}
        m & u\\n & v
    \end{pmatrix}\right) 
    \mc{E}_{m+n} (u+v).
\end{equation}
In particular, the specialisations $ J_n = \mathcal{E}_n(0)$ (for $ n \neq 0$) with commutation relations\footnote{A common notation is $ \alpha_n $ instead of $J_n$.}
\begin{equation}
    [J_m , J_n]
    =
    n \delta_{m+n,0} c
\end{equation}
form the \emph{Heisenberg algebra}.

We will also use a two-parameter generating series
\begin{equation}
    \mc{E}(u,z)
    \coloneqq
    \sum_{n \in \Z } z^n \mc{E}_n (u)\frac{dz}{z}.
\end{equation}

\emph{Fock space} has two main models, fermionic and bosonic. We will only use bosonic Fock space, but we will say a few words on Fermionic Fock space for context.

\emph{Fermionic Fock space} (of charge zero) $ \mc{V}_0$ is based on a vector space with basis indexed by $\mathbb{Z}'$, $V=\bigoplus_{i\in\mathbb{Z}'}\mathbb{C}v_i$.  Fermionic Fock space is the specific semi--infinite wedge space, spanned by wedge products infinite in the negative direction such that nearly all negative base vectors are in the product and nearly none of the positive ones.



As fermionic Fock space is based on $ V$, it has a natural representation of $ \widehat{\mf{a}}_\infty$.


The (charge-zero) \emph{bosonic Fock space} is the space of symmetric functions $ \Lambda = \lim_n \Lambda_n$, where  $ \Lambda_n = \C [ X_1, \dotsc, X_n]^{\mf{S}_n}$, and the limit is in the category of graded algebras. It is freely generated by the power sum functions $ \{ p_k = \sum_n X_n^k \, \mid \, k > 0\} $, and has a Heisenberg action by assigning to the generators, for $ k > 0$
\begin{equation}
    J_k \mapsto k \del_k \coloneqq k \frac{\del}{\del p_k} \,,
    \qquad
    J_{-k} \mapsto p_k .
\end{equation}

It has a unit, $ 1$, often called the \emph{vacuum} and denoted $ |0 \>$, and a counit, often called the \emph{covacuum}, and denoted $ \< 0| \colon \Lambda \to \C \colon f(\underline{p}) = f(\underline{0})$.

\begin{proposition}[Boson--Fermion correspondence]
    There is an isomorphism of Heisenberg modules $ \mc{V}_0 \to \Lambda$, called the Boson--Fermion correspondence. Under this isomorphism,
    \begin{equation}
        \sum_{k,l \in\Z'} x^l y^{-k} E_{kl}
        \mapsto
        x^{\frac{1}{2}}y^{\frac{1}{2}}\frac{e^{\sum_{j=1}^\infty (y^{-j} - x^{-j})p_j/j } e^{\sum_{j=1}^\infty (x^j - y^j)\del_j} -1}{x-y} 
    \end{equation}
    and $ c \mapsto \Id $. In particular
    \begin{equation}
        \mc{E}(u,z) \mapsto \frac{e^{\sum_{j=1}^\infty \varsigma (uj) \frac{J_{-j}}{j} z^{-j}} e^{ \sum_{k=1}^\infty \varsigma (uk) \frac{J_k z^k }{k}} }{\varsigma (u)} \frac{dz}{z}.
    \end{equation}
\end{proposition}

We can formally exponentiate the Lie algebra $ \widehat{\mf{a}}_\infty$ to a Lie group\footnote{Again, we should be a bit careful with convergence here. We will comment on this when relevant.}
\begin{equation}
    G
    \coloneqq
    \big\{ e^{X_1} \dotsb e^{X_n} \, \big| \, X_i \in \widehat{\mf{a}}_\infty \big\}
\end{equation}

and this group naturally acts on Fock space. The orbit of the vacuum under this action turns out to be of extreme importance in the theory of integrable hierarchies (see \cite{MJD00} for more details). It is often written in terms of the vacuum expectation value $ \< g\> \coloneqq \< 0 | g | 0 \>$.

\begin{theorem}[{\cite{SS83}}]
    The orbit of the vacuum under the action of $ G $ on $ \Lambda$ is the space of $\tau $-functions of the Kadomtsev--Petviashvili integrable hierarchy in the variables $ t_k \coloneqq \frac{p_k}{k}$. More explicitly, for $ g \in G$, the $ \tau$-function is
    \begin{equation}
        \tau (t, g) 
        \coloneqq
        g | 0 \>
        =
        \< e^{\sum_{k=1}^\infty t_k J_k} g \>.
    \end{equation}
    Here the final expression should be seen as a vacuum expectation value on Fock space in \emph{another} set of variables, or in e.g. fermionic Fock space.
\end{theorem}

On the fermionic side, the orbit $ G | 0 \> $ is (the cone over) the image of the Plücker embedding of the Sato Grassmannian: the Grassmannian of subspaces $ W$ of $ V $ such that the composition $ W \hookrightarrow V \twoheadrightarrow V_- = \mathop{\mathrm{Span}} \{ v_i \, | \, i < 0\} $ is Fredholm of index zero; the projection is along the complementary subspace $V_+ = \mathop{\mathrm{Span}} \{ v_i \, | \, i > 0\} $.

We also consider the \emph{bosonic current}
\begin{equation}
    J(X)
    \coloneqq
    \sum_{m \in \Z } J_m X^m \frac{dX}{X} .
\end{equation}

\begin{definition}
    Following \cite{AKL25}, we define the $k$-\emph{leaky} $r$-completed cycles cut-and-join operators $\mathcal{F}_r^k$ via
    \begin{equation}\label{LeakyCompletedOperator}
        \mathcal{E}_{-k} (u)
        =
        \sum_{r=0}^\infty\mathcal{F}^k_r\frac{u^r}{r!}.
    \end{equation}
    For $k = 0$, we write $ \mc{F}_r = \mc{F}_r^0$; these are the usual (non-leaky) $r$-completed cycles operators.\footnote{These are sometimes defined without the central term from $ \mc{E}_0$, but this only slightly changes the disconnected Hurwitz numbers, while the connected ones are unaffected.}
\end{definition}

The minus sign for $k $ on the left side of \cref{LeakyCompletedOperator} comes from the fact that \emph{positive} leakiness means that that something \emph{decreases}. 
Note that $ \mc{F}_0^k = J_{-k}$.

The starting point of the relation between Fock space and integrability on the one hand and Hurwitz theory on the other hand, is the following result of Okounkov, confirming and extending a conjecture by Pandharipande \cite{Pan00}.

\begin{theorem}[{\cite{Ok00}}]
    The generating series of double simple Hurwitz numbers
    \begin{equation}
        \tau ( \beta, \underline{p}, \underline{p}')
        =
        \sum_{d, g = 0}^\infty \sum_{\mu, \nu \vdash d} \frac{\beta^{2g-2+ l(\mu) + l (\nu) }}{(2g-2+l(\mu) + l(\nu))!} h^\bullet_{g;\mu, \nu} p_{\mu} p'_\nu
    \end{equation}
    in two infinite sets of power sum variables can be expressed in Fock space as
    \begin{equation}\label{SimpeHurwitzFS}
        \tau ( \beta, \underline{p}, \underline{p}')
        =
        \Big\< e^{\sum_{k=1}^\infty \frac{p_k J_k}{k}} e^{\beta \mc{F}_2} e^{\sum_{k=1}^\infty \frac{p_k' J_{-k}}{k}}\Big\> .
    \end{equation}
    It is a $\tau$-function of the 2D-Toda hierarchy, and hence in particular a KP $\tau$-function in each of its variables.
\end{theorem}

By now, this is one of the major points of view on Hurwitz numbers, so much so that many types of Hurwitz numbers are defined essentially via the Fock space formalism, again starting with Okounkov and Pandharipande \cite{OP06a}. In  \cite{OP06}, they proved the Gromov-Witten/Hurwitz correspondence, \cref{GW/H}, where in the Fock space formalism the operator $ \mc{F}_2$ in \cref{SimpeHurwitzFS} is replaced with different $ \mc{F}_r$ depending on the descendant insertions.

The origin of this paper is \cite{CMR25}, which defines \emph{leaky} Hurwitz numbers loosely by interpreting usual Hurwitz numbers as counting meromorphic sections $ f $ of $ \mc{O}_{\Sigma} $ with given zeroes and poles, instead of maps $ f \colon \Sigma \to \P^1$, and then replacing the line bundle by $ \omega_{\Sigma}^k $. More precisely, this involves intersection with the log pluricanonical double ramification cycle, but this is irrelevant for our discussion. One result of that paper, which we take as one of our starting points, is an expression of such numbers in terms of Fock space.

\begin{theorem}[{\cite[Theorem 5.2.2]{CMR25}}]\label{CMRFockSpace}
    The count of genus $g$, $k$-leaky covers of $ \P^1$ with left degree $ \mu$ and right degree $ \nu$ is given in Fock space by
    \begin{equation}
        \frac{1}{| \Aut \mu| | \Aut \nu| } \Big\< \prod_{i = 1}^{l (\mu)} \frac{J_{\mu_i}}{\mu_i} (M_k)^{2g-2+l(\mu) + l(\nu)} \prod_{j = 1}^{l(\nu)} \frac{J_{-\nu_j}}{\nu_j} \Big\>,
    \end{equation}
    where $ M_k = \mc{F}_2^k - \frac{k^2}{24} J_{-k}$, cf. \cite[Section 4]{AKL25}.
\end{theorem}

\begin{remark}
\label{rem-twistedGw-H}
	We first note that  \cite[Section 4]{AKL25} refers to an earlier version of \cite{CMR25}, where the operator $M_k$ was missing a summand of $\frac{-1}{24}J_{-k}$. In forthcoming work of Cavalieri, Markwig, and Ranganathan, a natural geometric interpretation of the operator $\mc{F}_2^k$ will be derived as a twisted version of the Gromov--Witten/Hurwitz correspondence.
\end{remark}

The difference between $ M_k $ and $ \mc{F}_2^k$, can be seen as a \emph{completion}, which for $ k =0$ only appears in case $ r > 2$.

This leads us to our very general definition of (leaky) Hurwitz numbers.

\begin{definition}\label{DefHNfromCJ}
    Consider an operator $ W \in \widehat{\mf{a}}_\infty \llbracket \hbar \rrbracket $, which we will call the \emph{cut-and-join operator}. The disconnected \emph{double $W$-Hurwitz number} with left degree $ \mu$ and right degree $ \nu$ is
    \begin{equation}\label{DefLH}
        \lh^{W,\bullet}_{g,\mu,\nu}
        \coloneqq
        [\hbar^{2g-2} ]\frac{1}{| \Aut \mu| | \Aut \nu| } \Big\< \prod_{i = 1}^{l (\mu)} \frac{J_{\mu_i}}{\hbar \mu_i} e^W \prod_{j = 1}^{l(\nu)} \frac{J_{-\nu_j}}{\hbar \nu_j} \Big\> .
    \end{equation}
    Now pick also a sequence of numbers (or formal variables) $ s_1, s_2, \dotsc$. The \emph{disconnected $n$-point functions} are
    \begin{equation}
        H_{n}^\bullet (X_{\llbracket n\rrbracket})
        \coloneqq
        \sum_{m_1, \dotsc, m_n = 1}^\infty \frac{X_1^{m_1} \dotsc X_n^{m_n}}{m_1 \dotsb m_n} \Big\< J_{m_1} \dotsb J_{m_n} e^W e^{\sum_{q=1}^\infty s_q \frac{J_{-q}}{q\hbar}} \Big\>.
    \end{equation}
    The \emph{connected $n$-point functions} are defined via the inclusion-exclusion relations
    \begin{equation}
        H_n (X_{\llbracket n \rrbracket})
        \coloneqq
        \sum_{I \vdash \llbracket n \rrbracket} (-1)^{|I|-1} (|I|-1)! \prod_{i=1}^{|I|} H_{|I_i|}^\bullet (X_{I_i})
    \end{equation}
    with inverse
    \begin{equation}
        H_n^\bullet (X_{\llbracket n \rrbracket})
        =
        \sum_{I \vdash \llbracket n \rrbracket} \prod_{i =1}^{|I|} H_{|I_i|} (X_{I_i}) .
    \end{equation}
\end{definition}

We will use the inclusion-exclusion relations more often to relate disconnected and connected invariants of similar types. We may also add a superscript $ \circ$ to emphasize connected invariants.

\begin{example}
    The case covered in \cref{CMRFockSpace} has cut-and-join operator $ W = \hbar M_k$.
\end{example}

We consider the $n$-point functions as generating series of double $W$-Hurwitz numbers. Note that they behave differently with respect to the two partitions $ \mu$ and $\nu$: the parts of the left partition $ \mu$ are encoded in the exponents of the $ X$s, and therefore the length of $ \mu $ is fixed to be $n$. On the other hand, the right partition $ \nu$ is encoded through the variables $ s_q$, attached to the parts of $\nu$ of size $ q$.

There exists a large literature on Hurwitz numbers that fits this definition in some way; mostly on the case of diagonal $W$, see e.g. \cite{ALS16,BDKS22,BDKS24}. The work \cite{CMR25}, and in particular \cref{CMRFockSpace} is, to our knowledge, the first interpretation of an off-diagonal cut-and-join operator in the Hurwitz framework, and it led us to consider this larger class of problems.

Cut-and-join operators in the literature are usually expressed in the following shape.
\begin{equation}
	W
	=
	\sum_{l \in \Z'} \sum_{k \in \Z} w_k (l + \frac{k}{2}, \hbar)  E_{l + k,l} 
    =
    \sum_{k \in Z} \sum_{r = 0}^\infty w_{k,r} (\hbar ) \mc{F}_r^k ,
\end{equation}
where $ w_k (s, \hbar )  = \sum_r w_{k,r} (\hbar) s^r \in \hbar^{-1} \C \llbracket s, \hbar \rrbracket$.

The exponent of $ \hbar $ is the negative of the Euler characteristic of the Riemann surfaces counted. Since the Euler characteristic is additive, this interpretation holds for any individual operator: the $ J_{\mu_i}$ and $ J_{-\nu_j}$ have an interpretation as caps and cups, i.e. both are topologically disks, and therefore require $ \hbar^{-1}$. The operators $ \mc{F}_2^k$ represent three-holed disks, and hence need $ \hbar^1$.

\begin{example}\label{CCexample}
    For fixed $k = 0$ and $r$, i.e. $W = \frac{\mc{F}_r^0}{r}$, \cref{DefLH} is related to \cref{CCHNdef} by
    \begin{equation}
        \lh^W_{g,\mu,\nu}
        =
        \frac{h^{(r, \dotsc, r)}_{g;\mu, \nu}}{(2g-2+ l(\mu) + l(\nu))!}  .
    \end{equation}
\end{example}

In order to have such a good geometric interpretation of the parameter $ \hbar$ in general, we will impose the following assumption.

\begin{assumption}\label{TopologicalAssumption}
    From now on, we assume that for all $ k$,
	\begin{equation}
		w_k (\frac{s}{\hbar}, \hbar) \in \hbar^{-1} \C \llbracket s , \hbar^2 \rrbracket .
	\end{equation}
\end{assumption}

We interpret $ \mc{F}_r^k$ as an $(r+1)$-holed Riemann sphere to glue, with correction terms representing $ (r+1-2g)$-holed, genus $g$ Riemann surfaces.
Or, in the tropical picture, a linear combination of vertices of genus $g$ with $ r+1-2g$ incident edges. We want the exponent of $ \hbar$ to be minus the Euler characteristic, i.e. $ 2g -2 + (r+1) = 2g-1+r$. The assumption encodes that $ g$ is a non-negative integer.\par
In principle, we could work with the weaker assumption that $ w_k (\frac{s}{\hbar}, \hbar ) \in \hbar^{-1} \C \llbracket s, \hbar \rrbracket$, allowing for half-integer genus. This can be useful to allow for counting non-orientable surfaces. We will not allow for half-integer genus, partially because our motivation comes from algebraic geometry, and partially because some of the techniques we use, especially in \cref{sec:TR}, have not been established in this generality.

\begin{lemma}
    With \cref{TopologicalAssumption}, the connected $n$-point functions have an expansion in powers of $ \hbar$ of the following form.
    \begin{equation}
        H_n (X_{\llbracket n \rrbracket})
        =
        \sum_{g = 0}^\infty \hbar^{2g-2+n} H_{g,n} ( X_{\llbracket n \rrbracket} ).
    \end{equation}
\end{lemma}

\begin{definition}\label{def:Multidifferentials}
    The \emph{multidifferentials} are
    \begin{align}
        \omega_{n} (X_{\llbracket n \rrbracket} )
        &\coloneqq
        d_1 \dotsb d_n H_{n} (X_{\llbracket n \rrbracket})
        =
        \sum_{m_1, \dotsc, m_n = 1}^\infty \Big( \prod_{j=1}^n X_j^{m_j} \frac{dX_j}{X_j} \Big) \Big\< J_{m_1} \dotsb J_{m_n} e^W e^{\sum_{q=1}^\infty s_q \frac{J_{-q}}{q\hbar}} \Big\>^\circ
        \\
        \omega_{g,n} (X_{\llbracket n \rrbracket} )
        &\coloneqq
        d_1 \dotsb d_n H_{g,n} (X_{\llbracket n \rrbracket}).
    \end{align}
    The \emph{completed multidifferentials} are
    \begin{equation}\label{Correlators}
        \begin{split}
            \widehat{\omega}_{n} (X_{\llbracket n \rrbracket} )
            &\coloneqq
            \sum_{m_1, \dotsc, m_n \in \Z} \Big( \prod_{j=1}^n X_j^{m_j} \frac{dX_j}{X_j} \Big) \Big\< J_{m_1} \dotsb J_{m_n} e^W e^{\sum_{q=1}^\infty s_q \frac{J_{-q}}{q\hbar}} \Big\>^\circ
            \\
            &=
            \Big\< J(X_1) \dotsb J(X_n) e^W e^{\sum_{q=1}^\infty s_q \frac{J_{-q}}{q\hbar}} \Big\>^\circ.
        \end{split}
    \end{equation}
\end{definition}

The (non-completed) multidifferentials are power series in the variables $ X_{\llbracket n \rrbracket}$ whose coefficients are Hurwitz numbers. On the other hand, the completed multidifferentials have better formal properties due to the entire current being present in the Fock space expression. They are very closely related, as the following proposition shows.

\begin{proposition}[{\cite[Proposition 4.1]{BDKS22}}]
    The non-completed and completed multidifferentials are related by
    \begin{equation}
        \widehat{\omega}_n
        =
        \omega_n + \delta_{n,2} \frac{dX_1 dX_2}{(X_1 - X_2)^2} ,
        \quad
        \text{i.e.}
        \quad
        \widehat{\omega}_{g,n}
        =
        \omega_{g,n} + \delta_{g,0} \delta_{n,2} \frac{dX_1 dX_2}{(X_1 - X_2)^2} ,
    \end{equation}
\end{proposition}

The proof of \cite[Proposition 4.1]{BDKS22} is given in a more restricted context, requiring that $W$ be diagonal of a certain shape, but this restriction plays no role in the proof. So it holds in our larger generality.

\subsection{Tropical covers}
\label{sec-tropcov}

In this section, we introduce some necessary basics of tropical geometry.
We note that we allow our graphs to have multiple edges between vertices.

\begin{definition}
    An abstract tropical curve is a graph $\Gamma$ endowed with the following:
    \begin{enumerate}
        \item The vertex set of $\Gamma$ is denoted by $V(\Gamma)$, while its edge set is denoted by $E(\Gamma)$.
        \item We call the one-valent vertices of $\Gamma$ \emph{leaves} and their adjacent edges \emph{ends}.
        \item We denote the set of leaves by $V^\infty(\Gamma)$ and its complement the set of \emph{inner vertices} by $V^0(\Gamma)$.
        \item The set $E(\Gamma)$ is decomposed into set of ends $E^\infty(\Gamma)$ and its complement $E^0(\Gamma)$ called the set of \emph{internal edges}.
        \item We have a lengths function
        \begin{equation}
            \ell \colon E(\Gamma)\to\mathbb{R}\cup\{\infty\}
        \end{equation}
        with $\ell^{-1}(\infty)=E^\infty(\Gamma)$, turning $\Gamma$ into a metric graph.
        \item We have a function
        \begin{equation}
            g \colon V^0(\Gamma)\to\mathbb{Z}_{\ge0}.
        \end{equation}
        For a vertex $v\in V^0(\Gamma)$, we call $g(v)$ its \emph{genus}.
    \end{enumerate}
    We define the genus of an abstract tropical curve as
    \begin{equation}
        g(\Gamma)
        =
        b^1(\Gamma)+\sum_{v\in V^0(\Gamma)}g(v),
    \end{equation}
    where $b^1(\Gamma)$ denotes the first Betti number of $\Gamma$.

\end{definition}

\begin{definition}
\label{def:tropcov}
    A tropical cover is a surjective harmonic map between the geometric realisations of two abstract tropical curves $\pi\colon\Gamma_1\to\Gamma_2$, i.e.:
    \begin{enumerate}
        \item $\pi(V(\Gamma_1))=V(\Gamma_2)$, $\pi(E(\Gamma_1))=E(\Gamma_2)$
        \item We interpret each edge $e$ as an interval $[0,\ell(e)]$. For $e\in E(\Gamma_1)$ with $\pi(e)\in E(\Gamma_2)$, we require that there exists a positive integer $\omega(e)$, such that $\pi_{|e}\colon [0,\ell(e)]\to[0,\ell(\pi(e))]$ is given by $t\mapsto \omega(e)\cdot t$. We call $\omega(e)$ the \emph{weight} of $e$.
        \item Let $v\in V(\Gamma_1)$ and $v'=\pi(v)$. Let $e'$ adjacent to $v'$. Then, we define the \emph{local degree} at $v$ as
        \begin{equation}
            d_v=\sum_{\substack{e\in E(\Gamma_1)\colon,\, v\in\partial e\\\pi^{-1}((e')^\circ)\cap e\neq\emptyset}}\omega(e),
        \end{equation}
        where $(e')^\circ$ denotes the interior of the edge $e'$. We require $d_v$ to be independent of the choice of $e'$ adjacent to $v'$. This is the \emph{balancing} or \emph{harmonicity} condition.
    \end{enumerate}

  Let $v'\in V(\Gamma_2)$, we define the \emph{degree} $\mathrm{deg}(\pi)$ of $\pi$ as $\sum_{\pi(v)=v'}d_v$. Due to harmonicity, the degree is independent of the choice of $v'$.

  For any end $e$ in $\Gamma_2$, we define a partition $\mu_e$ of $\mathrm{deg}(\pi)$ as the tuple of weights of ends mapping to $e$. We call $\mu_e$ the \emph{ramification profile} of $e$.

  Finally, we call two covers $\pi_1\colon \Gamma_1\to\Gamma$ and $\pi_2\colon\Gamma_2\to\Gamma$ \emph{equivalent} if there is an isomorphism $g\colon\Gamma_1\to\Gamma_2$ of metric graphs with $\pi_1=\pi_2\circ g$.
\end{definition}

In this paper, we focus on the case of tropical covers $\pi\colon\Gamma_1 \to \Gamma_2$ with target $\Gamma_2$ the \emph{tropical projective line} $\mathbb{P}^1_{\textrm{trop}}$, i.e. $\mathbb{P}^1_{\textrm{trop}}=\mathbb{R}\cup\{\pm\infty\}$ possibly subdivided by two-valent vertices $q_1,\dots,q_r\in\mathbb{R}$, e.g.:\vspace{\baselineskip}

\begin{center}
\tikzset{every picture/.style={line width=0.75pt}} 

\begin{tikzpicture}[x=0.75pt,y=0.75pt,yscale=-1,xscale=1]

\draw    (30.33,40) -- (260.08,40.08) ;
\draw  [fill={rgb, 255:red, 0; green, 0; blue, 0 }  ,fill opacity=1 ] (29.13,40) .. controls (29.13,39.33) and (29.67,38.79) .. (30.33,38.79) .. controls (31,38.79) and (31.54,39.33) .. (31.54,40) .. controls (31.54,40.67) and (31,41.21) .. (30.33,41.21) .. controls (29.67,41.21) and (29.13,40.67) .. (29.13,40) -- cycle ;
\draw  [fill={rgb, 255:red, 0; green, 0; blue, 0 }  ,fill opacity=1 ] (48.63,39.75) .. controls (48.63,39.08) and (49.17,38.54) .. (49.83,38.54) .. controls (50.5,38.54) and (51.04,39.08) .. (51.04,39.75) .. controls (51.04,40.42) and (50.5,40.96) .. (49.83,40.96) .. controls (49.17,40.96) and (48.63,40.42) .. (48.63,39.75) -- cycle ;
\draw  [fill={rgb, 255:red, 0; green, 0; blue, 0 }  ,fill opacity=1 ] (69.13,40) .. controls (69.13,39.33) and (69.67,38.79) .. (70.33,38.79) .. controls (71,38.79) and (71.54,39.33) .. (71.54,40) .. controls (71.54,40.67) and (71,41.21) .. (70.33,41.21) .. controls (69.67,41.21) and (69.13,40.67) .. (69.13,40) -- cycle ;
\draw  [fill={rgb, 255:red, 0; green, 0; blue, 0 }  ,fill opacity=1 ] (99.13,40) .. controls (99.13,39.33) and (99.67,38.79) .. (100.33,38.79) .. controls (101,38.79) and (101.54,39.33) .. (101.54,40) .. controls (101.54,40.67) and (101,41.21) .. (100.33,41.21) .. controls (99.67,41.21) and (99.13,40.67) .. (99.13,40) -- cycle ;
\draw  [fill={rgb, 255:red, 0; green, 0; blue, 0 }  ,fill opacity=1 ] (128.88,40) .. controls (128.88,39.33) and (129.42,38.79) .. (130.08,38.79) .. controls (130.75,38.79) and (131.29,39.33) .. (131.29,40) .. controls (131.29,40.67) and (130.75,41.21) .. (130.08,41.21) .. controls (129.42,41.21) and (128.88,40.67) .. (128.88,40) -- cycle ;
\draw  [fill={rgb, 255:red, 0; green, 0; blue, 0 }  ,fill opacity=1 ] (159.13,40) .. controls (159.13,39.33) and (159.67,38.79) .. (160.33,38.79) .. controls (161,38.79) and (161.54,39.33) .. (161.54,40) .. controls (161.54,40.67) and (161,41.21) .. (160.33,41.21) .. controls (159.67,41.21) and (159.13,40.67) .. (159.13,40) -- cycle ;
\draw  [fill={rgb, 255:red, 0; green, 0; blue, 0 }  ,fill opacity=1 ] (189.13,40) .. controls (189.13,39.33) and (189.67,38.79) .. (190.33,38.79) .. controls (191,38.79) and (191.54,39.33) .. (191.54,40) .. controls (191.54,40.67) and (191,41.21) .. (190.33,41.21) .. controls (189.67,41.21) and (189.13,40.67) .. (189.13,40) -- cycle ;
\draw  [fill={rgb, 255:red, 0; green, 0; blue, 0 }  ,fill opacity=1 ] (258.87,40.08) .. controls (258.87,39.42) and (259.42,38.88) .. (260.08,38.88) .. controls (260.75,38.88) and (261.29,39.42) .. (261.29,40.08) .. controls (261.29,40.75) and (260.75,41.29) .. (260.08,41.29) .. controls (259.42,41.29) and (258.87,40.75) .. (258.87,40.08) -- cycle ;
\draw  [fill={rgb, 255:red, 0; green, 0; blue, 0 }  ,fill opacity=1 ] (239.08,40.08) .. controls (239.08,39.42) and (239.62,38.88) .. (240.29,38.88) .. controls (240.96,38.88) and (241.5,39.42) .. (241.5,40.08) .. controls (241.5,40.75) and (240.96,41.29) .. (240.29,41.29) .. controls (239.62,41.29) and (239.08,40.75) .. (239.08,40.08) -- cycle ;
\draw  [fill={rgb, 255:red, 0; green, 0; blue, 0 }  ,fill opacity=1 ] (219.08,40.08) .. controls (219.08,39.42) and (219.62,38.88) .. (220.29,38.88) .. controls (220.96,38.88) and (221.5,39.42) .. (221.5,40.08) .. controls (221.5,40.75) and (220.96,41.29) .. (220.29,41.29) .. controls (219.62,41.29) and (219.08,40.75) .. (219.08,40.08) -- cycle ;
\end{tikzpicture}    
\end{center}
\vspace{\baselineskip}

For target the tropical projective line, we now define the notion of leaky tropical covers.

\begin{definition}
    Let $\pi\colon\Gamma\to\mathbb{P}^1_{\textrm{trop}}$ a map between abstract tropical curves satisfying conditions (1) and (2) of \cref{def:tropcov} with $-\infty<q_1<\cdots<q_n<\infty$ the vertices of $\mathbb{P}^1_{\textrm{trop}}$. Let $w$ be an inner vertex of $\Gamma$ with $\pi(w)=q_i$. Let $e_1$  be the edge adjacent to $q_i$ pointing in negative direction and $e_2$ the one pointing in positive direction. Consider
    \begin{equation}
        d_w^j=\sum_{\substack{e\in E(\Gamma)\colon,\, w\in\partial e\\\pi^{-1}((e_j)^\circ)\cap e\neq\emptyset}}\omega(e)
    \end{equation}
    and define $k_w=d_w^1-d_w^2$. We call $k_w$ the \emph{leakiness} of $w$ and $\pi$ a \emph{leaky tropical cover} with leakiness $\underline{k}=(k_w)_{w\in V^0(\Gamma)}$.

    The notions of ramification profiles and isomorphisms of \cref{def:tropcov} are adopted verbatim.
\end{definition}

A main feature of tropical covers it their use in computing vacuum expectations in the bosonic Fock space via Wick's theorem. For this, we consider the following set-up: Let $\mathbf{x}_1,\dots,\mathbf{x}_n$ be vectors of non-zero integers of any finite length with $\sum_{j=1}^{\ell(x_i)} (x_i)_j=k_i\in\mathbb{Z}$. We denote for $\mathbf{x}_i$ the vector only containing its positive (negative) entries by $\mathbf{x}_i^+$ (resp. $\mathbf{x}_i^-$). Moreover, consider vectors $\mathbf{x}_{-\infty} \in\mathbb{Z}_{>0}^m,\mathbf{x}_\infty \in \mathbb{Z}_{<0}^n$, then we aim to compute the following vacuum expectation:

\begin{equation}
\label{equ:vacc}
    \langle J_{\mathbf{x}_{-\infty}} \mid \prod_{i\colon (x_1)_i<0} J_{(x_1)_i} \prod_{i\colon (x_1)_i>0} J_{(x_1)_i} \cdots \prod_{i\colon (x_n)_i<0} J_{(x_n)_i} \prod_{i\colon (x_n)_i>0} J_{(x_n)_i} \mid J_{\mathbf{x}_\infty}\rangle.
\end{equation}

\begin{definition}
   Let $\mathbf{x}_{-\infty},\mathbf{x}_\infty,\mathbf{x}_1,\dots,\mathbf{x}_n$ vectors as above. Fix $n$ points $-\infty<q_1<\cdots<q_n<\infty$ in $\mathbb{P}^1_{\textrm{trop}}$ and $\pi\colon\Gamma\to\mathbb{P}^1_{\textrm{trop}}$ a leaky tropical cover, such that
   \begin{enumerate}
       \item The ramification profile of $-\infty$ is $\mathbf{x}_{-\infty}$ and the ramification profile of $\infty$ is $\mathbf{x}_\infty$.
       \item The vertex $q_i$ of $\mathbb{P}^1_{\textrm{trop}}$ has exactly one vertex $v_i$ in its pre-image,
       \item The vertex $v_i$ has valency $\ell(\mathbf{x}_i)$. Its adjacent edges are in bijection with the components of $\mathbf{x}_i$. The edge corresponding to the $j$--th component has weight $|(\mathbf{x}_i)_j|$. It maps into $(-\infty,q_i)$ if $(\mathbf{x}_i)_j>0$ and into $(q_i,\infty)$ if $(\mathbf{x}_i)_j<0$.
       \item The genus of each vertex is $0$.
   \end{enumerate}

   We call such a leaky tropical cover of type $(\mathbf{x}_{-\infty},\mathbf{x}_1,\dots,\mathbf{x}_n,\mathbf{x}_\infty)$. Moreover, we define its multiplicity as
   \begin{equation}
       \mathrm{mult}(\pi)
       =
       \frac{1}{|\mathrm{Aut}(\pi)|} \prod_{i=1}^n |\mathrm{Aut}(\mathbf{x}_i^+)| |\mathrm{Aut}(\mathbf{x}_i^-)| \prod_{e\in E(\Gamma)} \omega(e).
   \end{equation}
\end{definition}

\begin{remark}
    We note that a leaf adjacent to an end of weight $\omega(e)$ may be viewed as vertex with leakiness $\pm \omega(e)$.
\end{remark}

The following result is a version of Wick's theorem, see e.g. \cite{Wic50,BG16}.

\begin{theorem}[{Wick's theorem}]
    The vacuum expectation in \cref{equ:vacc} is equal to
    \begin{equation}
        \sum_\pi\mathrm{mult}(\pi),
    \end{equation}
    where $\pi$ is a leaky tropical cover of type $(\mathbf{x}_{-\infty},\mathbf{x}_1,\dots,\mathbf{x}_n,\mathbf{x}_\infty)$.
\end{theorem}

We may use Wick's theorem to obtain a tropical interpretation of completed cycles. To begin with, we have the following expansion (see e.g. \cite{AKL25})

\begin{equation}
    \mathcal{E}_{k}(z)
    =
    \frac{1}{\varsigma(z)} \sum_{n=0}^\infty\frac{1}{n!} \sum_{\substack{j_1,\dots,j_n\\j_1+\dots+j_n=k}} \prod_{i=1}^n \varsigma (j_iz)\colon\prod_{i=1}^n\frac{J_{j_i}}{j_i}\colon+\frac{\delta_{k,0}}{\varsigma(z)} c.
\end{equation}

From this, we can obtain an explicit bosonic expression of $\mathcal{F}_r^k$. For this, we define the following set

\begin{equation}
    S\mathbb{Z}^{r+1-2g;k}
    =
    \{ \textbf{x}\in\mathbb{Z}^{r+1-2g}\mid x_1\le \dots\le x_l<0<x_{l+1}\le\dots\le x_{r+1-2g},\, l>0,\,\sum x_i=-k \}.
\end{equation}

Moreover, for two partitions $\mu$ and $\nu$ not necessarily of the same size and a non--negative integer $g$, we define quantities

\begin{equation}
    m_g(\mu,\nu)
    =
    [z^{2g}] \frac{\prod_{i=1}^{\ell(\mu)} \mathcal{S}(\mu_iz) \prod_{j=1}^{\ell(\nu)} \mathcal{S}(\nu_jz)}{\mathcal{S}(z)}.
\end{equation}

Analogous to \cite[Lemma 3.3]{HL22}, we obtain the following result.

\begin{lemma}\label{prop-expansion}
    We have
    \begin{equation}
        \mathcal{F}_r^k
        =
        r!\sum_{g=0}^\infty \sum_{\textbf{x}\in S\mathbb{Z}^{r+1-2g;k}} m_g(\textbf{x}^+,\textbf{x}^-) \prod_{x_i\in\textbf{x}^-} J_{-x_i} \prod_{x_i\in\textbf{x}^+} J_{x_i}.
    \end{equation}
\end{lemma}

    Let $g$ a non-negative integer, $\mu$ and $\nu$ partitions of not necessarily the same size. Moreover, let $\underline{k}=(k_1,\dots,k_s)$ a tuple of integers and $\underline{r}=(r_1,\dots,r_s)$ a tuple of positive integers for $s>0$ satisfying
    \begin{equation}
        \sum (r_i-1)
        =
        2g-2+\ell(\mu)+\ell(\nu)
        \quad
        \text{and}
        \quad
        \sum \mu_i
        =
        \sum\nu_j+\sum k_l.
    \end{equation}

\begin{definition}
    We denote completed cycles $(\underline{k},\underline{r})$\emph{-leaky Hurwitz numbers} as
    \begin{equation}
    \label{def-leaky}
        h^{\underline{k},\underline{r}}_{g;\mu,\nu}
        =
        \frac{1}{\prod\mu_i\prod\nu_j} \left\langle \prod J_{\mu_i} \prod_{i=1}^s \frac{\mathcal{F}_{r_i}^{k_i}}{r_i} \prod J_{-\nu_j} \right\rangle.
    \end{equation}
    The case of $k_i=0$ recovers completed cycles Hurwitz numbers of \cite{OP06a}.
  
    When $\underline{k}=(k^s)$ and $\underline{r}=(r^s)$, we denote the resulting numbers by  $h_g^{k,r}(\mu,\nu)$. Moreover, we will also be interested in studying the special case of $\nu=(q,\dots,q)$, in which case we denote the resulting number by $h^{k,r,q}_{g,\mu}$. We call these \emph{orbifold} leaky Hurwitz numbers.
\end{definition}

    We may now use Wick's theorem to obtain an interpretation of $h_g^{\underline{k},\underline{r}}(\mu,\nu)$ via leaky tropical covers.

\begin{theorem}
\label{thm-corr}
    Let $g$ a non-negative integer, $\mu$ and $\nu$ partitions of not necessarily the same size. Moreover, let $\underline{k}=(k_1,\dots,k_s)$ a tuple of integers and $\underline{r}=(r_1,\dots,r_s)$ a tuple of positive integers for $s>0$ satisfying
    \begin{equation}
        \sum_{i=1}^s (r_i - 1) 
        =
        2g-2 + \ell(\mu) + \ell(\nu)
        \quad
        \text{and}
        \quad\sum \mu_i
        =
        \sum\nu_j+\sum k_l.
    \end{equation}
    We define the $C(\mu,\nu,\underline{k},\underline{r})$ as the set of $\underline{k}$-leaky tropical covers $\pi\colon\Gamma_1\to \P^1_{\text{trop}}$ satisfying:
    \begin{enumerate}
        \item The ramification profile of $(-\infty,q_1)$ is $\mu$ and of $(q_s,\infty)$ is $\nu$.
        \item For each $q_i\in \Gamma_2$, there is a unique vertex $v_i$ with $\pi(v)=q_i$. We denote $w_i=\mathrm{val}(v_i)$.
        \item We have $r_i-1 = 2g(v_i)-2+w_i$ for all $i=1,\dots,s$ and $g=g(\Gamma_1)$.
        \item For each vertex $v_i$, we denote by $\textbf{x}^-_{i}$ ($\textbf{x}^+_{i}$) the set of weights of incoming (resp. outgoing) edges adjacent to $v_i$. We define the vertex multiplicity of $v_i$ to be
        \begin{equation}
            m_{v_i}
            =
            (r_i-1)! m_{g(v_i)}(\textbf{x}^+_i,\textbf{x}^-_i).
        \end{equation}
    \end{enumerate}
    Then, we have
    \begin{equation}
        h^{\underline{k},\underline{r}}_{g;\mu,\nu}
        =
        \sum_{\substack{\pi\colon\Gamma_1\to\mathbb{P}^1_{\textrm{trop}}\\\pi\in C(\mu,\nu,\underline{k},\underline{r})}} \frac{1}{|\mathrm{Aut}(\pi)|} \prod_{i=1}^s m_{v_i} \prod_{e\in E^0(\Gamma_1)} \omega(e).
    \end{equation}
    \end{theorem}

\begin{proof}
    We consider the equality in \cref{def-leaky}. By employing \cref{prop-expansion}, we see that a summand of the right hand side has the form
    \begin{equation}
        \frac{1}{\prod\mu_i\prod\nu_j}\prod_{i=1}^s (r_i-1)! m_{g_i}(\textbf{x}^+_i,\textbf{x}^-_i)\left\langle \prod_{i=1}^{\ell(\mu)} J_{\mu_i}\prod_{i=1}^s\prod_j  J_{(\textbf{x}_i)_j}\prod_{l=1}^{\ell(\nu)} J_{-\nu_j}\right\rangle
    \end{equation}
    for $\textbf{x}\in S\mathbb{Z}^{r_i+1-2g_i}$. We apply Wick's theorem, we obtain a weighted enumeration of $\underline{k}$-leaky tropical covers $\pi\colon\Gamma_1\to\Gamma_2$, by moving the factors $(r_i-1)! m_{g_i}(\textbf{x}^+_i,\textbf{x}^-_i)$ into vertex multiplicities. The only thing that remains is that Wick's theorem produces tropical covers with all zero vertex genera. We define the genus $g(v_i)$ of $i$-th vertex of the obtained tropical cover via $r_i=2g_(v_i)-1+w_i$. The only thing that remains to show to see that the thus obtained set is $C(\mu,\nu,\underline{k},\underline{r})$, is to see that $g=g(\Gamma_1)$.

    This is done by an Euler characteristic computation. First, we observe that
    \begin{align}
    \begin{split}
        |E(\Gamma_1)|&=\frac{1}{2}\sum_{v\in V(\Gamma_1)}\mathrm{val}(v)=\frac{1}{2}(\ell(\mu)+\ell(\nu)+\sum_{i=1}^s\ell(\textbf{x}_i))\\
        &=\frac{1}{2}(\ell(\mu)+\ell(\nu)+\sum_{i=1}^s(r_i+1-2g(v_i)))
    \end{split}
    \end{align}
    We have imposed that
    \begin{equation}
        \sum (r_i - 1)
        =
        2g-2+\ell(\mu)+\ell(\nu)
    \end{equation}
    and thus we obtain

    \begin{equation}
    \label{euler-edges}
        |E(\Gamma_1)|=\frac{1}{2}(2\ell(\mu)+2\ell(\nu))+2g-2+2s-\sum_{i=1}^s 2g(v_i)=\ell(\mu)+\ell(\nu)+g-1+s-\sum_i g(v_i)
    \end{equation}
    Moreover, we have that
    \begin{equation}
        \begin{split}
        |V(\Gamma_1)|=s+\ell(\mu)+\ell(\nu).
        \end{split}
    \end{equation}
    Putting things together, we derive
    \begin{equation}
        1-h^1(\Gamma_1)=|V(\Gamma_1)|-|E(\Gamma_1)|=s+\ell(\mu)+\ell(\nu)-(\ell(\mu)+\ell(\nu)+g-1+s-\sum g_i)=1-(g-\sum g_i)
    \end{equation}
    and therefore
    \begin{equation}
        g=h^1(\Gamma_1)+\sum g_i=g(\Gamma_1).
    \end{equation}
    This completes the proof.
\end{proof}

\subsection{Topological recursion} \label{TR}

In this section, we introduce the basic background on topological recursion. We start with the original formulation of \cite{CEO06,EO07} and then discuss the relevant aspects of the more recent Alexandrov--Bychkov--Dunin-Barkowski--Kazarian--Shadrin (ABDKS)-programme \cite{BDKS22,BDKS23,BDKS25,BDKS24,ABDKS24TR-SD-FSmaps,ABDKS25universal,ABDKS24KP0,ABDKS25,ABDKS24KP+,ABDKS25blobbed}.

Eynard--Orantin topological recursion is a recursion of multi--differentials depending on a \emph{spectral curve} $( \Sigma, x, y, B)$, where
\begin{itemize}
    \item $\Sigma$ a Riemann surface;
    \item $x$ and $y$ are meromorphic functions on $\Sigma$, where    
    \begin{itemize}
        \item all ramification points of $x$ are simple, and disjoint from those of $ y$. We write $ R $ for the set of these ramification points of $x$;
        \item $y$ is regular at all critical points.
    \end{itemize}  
    \item $B$ a symmetric bidifferential on $\Sigma$ with a double pole on the diagonal, biresidue $1$ and no other singularities.
\end{itemize}
In case $ \Sigma$ is compact, $B$ can be determined by a choice of Torelli marking, i.e. a symplectic basis of $ H_1(\Sigma; \Z)$, by requiring that the $ \mf{A}$-periods of $B$ vanish.

From this, we define the unstable cases
\begin{equation}
    \omega_{0,1}(z_1)=y(z_1)dx(z_1)\quad\textrm{and}\quad \omega_{0,2}(z_1,z_2)=B(z_1,z_2).
\end{equation}

We then recursively define for $2g-2+n+1>0$  multidifferentials

\begin{equation}
    \omega_{g,n+1}(z_0, z_{\llbracket n \rrbracket})
    =
    \sum_{a \in R} \Res_{z=a} K^a(z_0,z) \bigg( \omega_{g-1,n+2}(z,\sigma_{a}(z),z_{\llbracket n \rrbracket}) 
    + 
    \!\!\!\! \sum_{\substack{g_1+g_2 = g \\ I \sqcup J = \llbracket n \rrbracket}}' \!\!\!\! \omega_{g_1,I}(z,z_{I})\omega_{g_2,J}(\sigma_{a}(z),z_{J})\bigg),
\end{equation}

where $\sigma_a$ is the local deck transformation for $ x$ near $ a$, the prime on the summation means excluding the terms with $ (g,n) = (0,1)$, and 

\begin{equation}
    K^a(z_1,z)
    =
    \frac{\int_{z}^{\sigma_a(z)}\omega_{0,2}(z_1,\cdot)}{(\omega_{0,1}(\sigma_a(z)-\omega_{0,1}(z))dx(z)}.
\end{equation}

We say that the differentials $\omega_{g,n}$ \emph{satisfy topological recursion}.

Topological recursion was developed to solve loop equations, and these loop equations are still a major tool for understanding the formalism. 
\begin{proposition}\label{LoopEquations}
    On a compact spectral curve, the correlation differentials $\omega_{g,n}$ satisfy topological recursion if and only if they satisfy the following four properties:
    \begin{enumerate}
        \item \emph{Initial cases}:
        \begin{equation}
            \omega_{0,1}=ydx\quad\textrm{and}\quad \omega_{0,2}=B.
        \end{equation}
        \item The differentials satisfy the \emph{linear loop equations}: for each ramification point $a$,
        \begin{equation}
            \omega_{g,n+1}(z,z_{\llbracket n \rrbracket})+\omega_{g,n+1}(\sigma_a(z),z_{\llbracket n \rrbracket})
        \end{equation}
        is holomorphic at $z\to a$ and has at least a simple zero at $z=a$.
        \item The differentials satisfy the \emph{quadratic loop equations}: for each ramification point $a$,
        \begin{equation}
            \omega_{g-1,n+2}(z,\sigma_a(z),z_{\llbracket n \rrbracket}) + \sum_{\substack{g_1+g_2=g\\I\sqcup J = \llbracket n \rrbracket}} \omega_{g_1,|I|+1}(z,z_{I})\omega_{g_2,|J|+1}(\sigma_a(z),z_{J})
        \end{equation}
        has a double zero at $ z = a$.
        \item The \emph{projection property}: the differentials $\omega_{g,n}$ for $2g-2+n>0$ have no poles other than the critical points of $x(z)$ and all their $\mathfrak{A}$-periods vanish.\label{ProjectionProperty}
    \end{enumerate}
\end{proposition}

Slightly more abstractly, the loop equations exist at any point $ a$ in $ \Sigma$, and there are $ r$ independent ones, where $r$ is the ramification order of $ a$. If $ r = 1$, i.e. the unramified case, the single loop equation requires that the $ \omega_{g,n}$ for $ 2g-2+n > 0$ are holomorphic at $ a$. At simple ramifications, $r = 2$, and there are two loop equations: linear and quadratic. From this point of view, the projection property only adds the vanishing of $ \mf{A}$-periods. If the spectral curve is not compact, though, the projection property contains a lot more information.

The explicit formula for topological recursion can be interpreted from this perspective as follows: at every step of the recursion, the loop equations determine uniquely the principal parts of the differential, and the residue computation with the recursion kernel $K$ constructs the unique global differential with these principal parts and vanishing $ \mf{A}$-periods, using that $ \int_\mf{A} B = 0$.

A  strong point of the topological recursion formula is that it is an explicit recursive construction of the unique solution to these loop equations. This is one of the main motivations of the original development in  \cite{CEO06,EO07} -- the loop equations in these works coming from matrix models.

But topological recursion has more uses than just the explicit formula. In the context of enumerative invariants, including Hurwitz numbers, the multidifferentials are usually defined formally as generating series, cf. \cref{def:Multidifferentials}. Part of the statement of topological recursion for such multidifferentials is that they are all convergent power series on (cartesian powers of) one spectral curve, and that their poles are of bounded order and may only lie at ramification points. This already restricts each $ \omega_{g,n}$ to a finite-dimensional space, and implies some quasi-polynomiality property of the generating series. More explicitly, we have the following result.

\begin{proposition}
    The multidifferentials $ \omega_{g,n}$ are symmetric meromorphic multidifferentials on $ \Sigma^n$, with poles only at $R$, of order $ 6g-6+2n$.
\end{proposition}

Combining this with the linear loop equations, we find the following.

\begin{corollary}
\label{cor-polyfromTR}
    For a genus zero compact spectral curve, the $ \omega_{g,n}$ can be expressed in terms of a  basis $ \{ \xi_{a,k} (z) = d (\frac{d}{dx})^k \frac{1}{z -a} \, | \, k \geq 0, a \in R\} $, where for a particular $ \omega_{g,n}$, $ k \leq 3g-3+n$.\par
    Expanding at $ x= 0$, they are given by the (non-polynomial) expansion of the $ \xi_{a,0}$, times a polynomial of degree $ 3g-3+n$.
\end{corollary}

\subsubsection{Generalisations of topological recursion}

The setup of topological recursion has been generalised in many steps, mostly weakening the restrictions on the functions $ x$ and $ y$. Historically, the first step was to allow higher-order ramification \cite{BE13}. For the simply-ramified case, regularity conditions on $y$ were relaxed in \cite{DN18,CN19}. In a third direction, we can generalise to meromorphic one-forms $ dx$ and $ dy$, which may not integrate to actual functions $ x$ and $y$; in case they have residues we need a generalisation called \emph{log topological recursion} or \emph{logTR} for short \cite{ABDKS24a}.

Topological recursion with higher ramification of $ x$ and irregular $y$ -- with $ \omega_{0,1} = y \, dx$ still holomorphic -- was first constructed by \cite{BBCCN24} through Airy structures of \cite{KS18}. This generalisation surprisingly required an extra admissibility condition: unless a certain relation between the leading orders of $ x$ and $ y$ is satisfied at every point, the formula does not yield symmetric correlators. In \cite{BKS24}, this method was pushed to singular spectral curves. Moreover, \cite{BBCKS25} studied how this generalised topological recursion behaves in families when the ramification behaviour changes. 

Very recently, \cite{ABDKS25} constructed another generalisation of topological recursion for general $x$ and $ y$, called \emph{degenerate topological recursion}. This only requires two arbitrary non-zero differentials $ dx$ and $ dy$, and a choice of the subset of `special points', i.e. those points that could possibly contribute. This is compatible with logTR, also behaves well in families, and has no admissibility conditions. Interestingly, in case the generalised topological recursion of \cite{BBCCN24} is well-defined, it does not always coincide with this degenerate topological recursion, and the two can be obtained as limits of different families.

The reason that these two generalisations differ, is that they construct different versions of higher loop equations. In this paper, we will use the version of Alexandrov--Bychkov--Dunin-Barkowski--Kazarian--Shadrin, and we will recall it here.
 
This string of generalisations was driven by several factors: for one, the theory was needed to understand particular examples, in particular in the context of the relation with cohomological field theories \cite{DOSS14,BCGS25}. As mentioned, topological recursion should also behave well with respect to deformations, or dually collapsing several special points on the spectral curve. Finally, it should behave well with respect to KP integrability (explained below) and symplectic transformations: transformations of the spectral curve that leave the symplectic form $ dx \wedge dy$ (in some ambient space) invariant.

\subsubsection{Degenerate topological recursion}

Here, we will introduce the (degenerate, generalised) topological recursion of \cite{ABDKS25}, emphasising the parts we will need.

Let $ \Sigma$ be a smooth curve and $ dx$ and $ dy $ two meromorphic differentials on $ \Sigma $ -- they need not integrate to functions $ x$ and $ y$.

\begin{definition}\label{DefSpecialPoints}
    Let $ q \in \Sigma$ and $ z $ a local coordinate on $ \Sigma $ at $q$. Expand $ dx$ and $ dy $ in $z$ as
    \begin{equation}
        dx 
        =
        a z^{r-1} ( 1 + \mc{O}(z)) dz ,
        \qquad
        dy
        =
        b z^{s-1} (1 + \mc{O}(z)) dz,
        \qquad
        \text{with $ a, b \in \C^* $ and $ r, s \in \Z$.}
    \end{equation}
    Then $ q$ is called \emph{non-special} if $ r = s = 1$ or $ r + s \leq 0$, and \emph{special} otherwise. We write $ \mathsf{S} $ for the set of special points.
\end{definition}

\begin{definition}
    The \emph{initial data} of generalised topological recursion, or \emph{spectral curve} is a tuple $(	\Sigma, dx, dy, B, \mc{P})$, where
    \begin{itemize}
        \item $ \Sigma $ is a smooth complex curve.
        \item $B$ is a symmetric bi-differential on $\Sigma^2$ with a second order pole on the diagonal with biresidue $1$ and no other poles.
        \item $dx$, $dy$ are two arbitrary non-zero meromorphic differentials on $\Sigma$.
        \item $\mc{P} = \{q_1,\dotsc , q_N \}$ is a finite subset of $\mathsf{S}$.
    \end{itemize}
\end{definition}

For the rest of this section, we shall consider these initial data as a given.

\begin{definition}\label{DefExtendedDifferentials}
    Let $ \{ \omega_{g,n} \}_{g \geq 0, n \geq 1} $ be a collection of symmetric differentials on $ \Sigma$, i.e. $ \omega_{g,n} $ is an $n$-differential on $ \Sigma^n$. Assume that $ \omega_{g,n} $ is meromorphic for $ (g,n) \neq (0,1)$ and $ \omega_{0,1} = y \,dx$.\footnote{This need not be well-defined globally; all definitions from here only involve derivatives of $ y$.}.\par
    We first define
    \begin{equation}
        \omega_n
        =
        \sum_{g \geq 0} \hbar^{2g-2+n} \omega_{g,n}.
    \end{equation}
    The \emph{extended differentials} are defined by
    \begin{align}
        \mb{W}_n (z_{\llbracket n \rrbracket}, u_{\llbracket n \rrbracket})
        =
        \prod_{i=1}^n \frac{dx_i}{\hbar u_i} e^{u_i ( \mc{S} (u_i \hbar \del_{x_i}) - 1) y_i} 
        \sum_\Gamma \frac{1}{| \Aut \Gamma |} \prod_{e \in E (\Gamma )} \prod_{j=1}^{|e|} \Big\lfloor_{\substack{\tilde{u}_j \to u_{e(j)} \\ \tilde{x_j} \to x_{e(j)}}} \tilde{u}_j \hbar \mc{S} ( \tilde{u}_j \hbar \del_{\tilde{x}_j} ) \frac{ \tilde{\omega}_{|e|}(\tilde{z}_{\llbracket|e|\rrbracket}) }{\prod_{j=1}^{|e|} d\tilde{x}_j},
    \end{align}
    where $ \tilde{\omega}_n = \omega_n - \delta_{n,1} \hbar^{-1} \omega_{0,1}$ unless $ n =2 $ and $ e(1) = e(2)$, in which case $ \tilde{\omega}_2 = \omega_2 - \frac{dx_1 dx_2}{(x_1 - x_2)^2}$, and the sum is taken over all connected multigraphs (i.e. graphs with `multi-edges' of $\ge 1$ half-edges) with $ n $ labelled vertices.
\end{definition}

\begin{lemma}\label{DegenerateLoopEquations}
    For any $ (g,n) \neq (0,0)$,
    \begin{equation}
        \overline{\mc{W}}_{g,n} (z, u; z_{\llbracket n \rrbracket} )
        \coloneqq
        [\hbar^{2g-1+n}] \bigg\lfloor_{u_{\llbracket n \rrbracket} \to 0} \mb{W}_{n+1} ( z, z_{\llbracket n \rrbracket}, u, u_{\llbracket n \rrbracket}) - \omega_{g,n+1}(z, z_{\llbracket n \rrbracket})
    \end{equation}
    involves only positive powers of $ [u]$, and only $ \omega_{g',n'}$ with $ 2g' - 2 + n' < 2g-1+n$.
\end{lemma}

\begin{definition}\label{DegenerateTR}
    The \emph{differentials of (degenerate, generalised) topological recursion} are defined by the base steps $ \omega_{0,1} = y \, dx$, $ \omega_{0,2} = B$ and the induction step
    \begin{equation}
        \omega_{g,n+1}(z_0, z_{\llbracket n \rrbracket})
        =
        - \sum_{q \in \mc{P}} \Res_{z = q} 
        \Big( \int^z B( \cdot, z_0) \Big)
        \sum_{r = 1}^\infty ( - d \frac{1}{dy} )^r [u^r] \overline{\mc{W}}_{g,n} (z, u; z_{\llbracket n \rrbracket}).
    \end{equation}
\end{definition}

\begin{theorem}[{\cite{ABDKS25}}]
    \Cref{DegenerateTR} recursively produces symmetric meromorphic multidifferentials satisfying \emph{loop equations}:
    \begin{equation}
       \sum_{r = 1}^\infty ( - d \frac{1}{dy} )^r [u^r] \bigg\lfloor_{u_{\llbracket n \rrbracket} \to 0} \mb{W}_{n+1} ( z, z_{\llbracket n \rrbracket}, u, u_{\llbracket n \rrbracket})
       \qquad
       \text{is holomorphic at any $ z \in \mc{P}$.}
    \end{equation}
    Moreover, they satisfy the projection property, \cref{LoopEquations} \cref{ProjectionProperty}.\par
    If all special points are in $ \mc{P}$ and have $ (r,s) = (2,1)$, this definition reduces to the Eynard--Orantin definition.
\end{theorem}

We will use the following example.

\begin{example}[{\cite[(15)]{ABDKS24}}]\label{TrivialTR}
    For the spectral curve $\mc{S}_0 = (\P^1, x_0 = \log z, y_0 = S(z), B = \frac{dz_1 dz_2}{(z_1 - z_2)^2}, \mc{P} = \emptyset)$, topological recursion is trivial, in the sense that $ \omega_{g,n} = 0$ for all $ 2g - 2 + n > 0$. The extended differentials, in the sense of \cref{DefExtendedDifferentials}, are given by
    \begin{equation}
	   \mb{W}_n^{\mc{S}_0} (z_{\llbracket n\rrbracket}, u_{\llbracket n\rrbracket}) 
	   =
	   \prod_{i=1}^n \Big( dz_i \, e^{u_i (\mc{S}(u_i \hbar z_i \del_{z_i}) - 1) y_i} \Big) (-1)^{n-1} \!\!\! \sum_{\sigma \in n\text{-cycles}} \prod_{i=1}^n \frac{1}{e^{\frac{u_i \hbar}{2}} z_i - e^{-\frac{u_{\sigma (i)} \hbar }{2}} z_{\sigma (i)}} .
    \end{equation}
\end{example}

Moreover, by construction this topological recursion behaves well with respect to a large class of transformation and deformations.

\begin{theorem}[{\cite[Theorem 1.14]{ABDKS25universal}, \cite[Theorem 3.6]{ABDKS25}}]\label{Thmx-y}
    Assume that $ \Sigma $ is compact and $ B $ is the canonical Bergman kernel with vanishing $ \mf{A}$-periods. Then the \emph{$ x$-$y$ dual differentials}, i.e. the differentials computed by topological recursion on the \emph{dual} spectral curve $ (\Sigma, dy, dx, B, \mathsf{S} \setminus \mc{P})$, are given by
    \begin{equation}\label{NonExtendedx-y}
        \omega^\vee_{n} (z_{\llbracket n \rrbracket}
        =
        (-1)^n \Bigg( \prod_{i=1}^n \sum_{r = 0}^\infty ( - d\frac{1}{dy_i})^r [u_i^r] \Bigg) \mb{W}_n (z_{\llbracket n \rrbracket, u_{\llbracket n \rrbracket }}).
    \end{equation}
    A useful reformulation, cf. \cite[Theorem 3.10.3]{ABDKS24} is the following: the dual extended differentials can be expressed in terms of the original extended differentials as
    \begin{equation}\label{Extendedx-y}
        \mb{W}^\vee_n (z_{\llbracket n \rrbracket}, v_{\llbracket n \rrbracket})
        =
        (-1)^n \prod_{i=1}^n \Big( e^{-v_i x_i} \sum_{r=0}^\infty (- d_i \frac{1}{dy_i} )^r e^{v_i x_i} [u_i^r] \Big) \mb{W}_n (z_{\llbracket n \rrbracket}, u_{\llbracket n \rrbracket}) + \delta_{n,1} \frac{dy_1}{v_1 \hbar}.
    \end{equation}
\end{theorem}

There is another, large, group of transformations which preserve topological recursion. Namely, one can add to $ y $ a function of $x$, or dually add to $ x$ a function of $y$. In principle, this first version is straightforward, as it only changes $ \omega_{0,1}$, and the second can be obtained from the first by $ x$-$y$ duality. However, in the general situation where $ dx$ and $ dy$ do not integrate to univalued functions, but have residues, one has to be very careful with these modifications, and they affect all $ \omega_{g,1}$ in a specific way. This analysis was performed in a large class of cases in \cite{ABDKS24}.

\begin{theorem}[{\cite[Section 6]{ABDKS24}}]\label{ThmSD}
    Let $ \mc{S} = ( \Sigma, dx, dy, B, \mc{P}) $ be a spectral curve, and $ \tilde{\psi} (\theta ;\hbar ) $ a function depending on $ \theta \in \P^1$ and formal $ \hbar$ Set $ \psi (\theta ) = \tilde{\psi}(\theta; 0)$, and let $ \mc{S}^\dagger = ( \Sigma^\dagger = \Sigma, dx^\dagger = dx , dy^\dagger = d( y + \psi \circ x), B^\dagger = B, \mc{P}^\dagger = \mc{P}) $. If $ \{ \omega _{g,n} \} $ are the multidifferentials obtained by topological recursion on $ \mc{S}$ then
    \begin{align}
        \omega_{g,n}^\dagger
        &=
        \omega_{g,n} + \delta_{n,1} [\hbar^{2g} ] \tilde{\psi} (x (z); \hbar ) dx
        \\
        \mb{W}^\dagger_n (z_{\llbracket n \rrbracket}, u_{\llbracket n \rrbracket })
        &=
        \bigg( \prod_{i=1}^n e^{u_i ( \mc{S}(u_i \hbar \del_{x_i}) \tilde{\psi}(x_i ) - \psi(x_i))}\bigg) \mb{W}_n (z_{\llbracket n \rrbracket}, u_{\llbracket n \rrbracket}) \label{ExtendedSD}
    \end{align}
    satisfies topological recursion on $ \mc{S}^\dagger$ in the following cases:
    \begin{itemize}
        \item $x$ and $y$ are meromorphic, $ \psi = \tilde{\psi} $ is rational and regular at critical values of $ x$, all zeroes of $ dx$, $dy$, and $ dy^\dagger$ are simple, and the zeroes of $ dx $ are disjoint from the zeroes of $ dy$.
        \item $ \Sigma = \P^1$, 
        \begin{equation}
            x 
            = R(z), 
            \qquad
            \tilde{\psi} 
            = 
            T (\theta ) + \sum_i \frac{1}{\beta_i \mc{S} (\hbar \beta_i \del_\theta)} \log (\theta - b_i)
        \end{equation}
        are not constant with $ R $ and $ T $ rational, all $ b_i $ distinct, such that $ R(z) - b_i$ have only simple zeroes, $ \psi $ is regular at critical values of $ x$, $T$ is regular at the $b_i$, and the poles of $ d( \psi \circ x) $ which are not poles of $ dx $ are also not poles of $ dy$.
        \item $ \Sigma = \P^1$,
        \begin{equation}
            x 
            =
            R(z) + \sum_i \alpha_i^{-1} \log (z - a_i),
            \qquad
            \tilde{\psi} 
            = 
            \kappa_1 \theta + \kappa_0
        \end{equation} 
        are not constant with all $ a_i$ distinct and $ R $ a rational function which is regular at the $ a_i$. 
        \item $ \Sigma = \P^1$,
        \begin{equation}
            x
            = 
            \sum_i \alpha_i^{-1} \log (z - a_i),
            \qquad
            \tilde{\psi}  
            = 
            T ( e^{\gamma_1 \theta}, \dotsc, e^{\gamma_K \theta}) + \kappa \theta 
        \end{equation} 
        are not constant with $T $ a rational function, all $ a_i $ distinct, $ \psi $ regular at critical values of $ x$, and $ \frac{\gamma_j}{\alpha_i} \in \Z $ for all $ i,j$.
    \end{itemize}
\end{theorem}

For the main deformation result, we will not quite give all the intricate details.

\begin{theorem}\label{TRLimits}
    Suppose that $ \Sigma \to T $ is an analytic family of smooth curves with meromorphic relative differentials $ dx$ and $ dy$ and a symmetric relative fundamental bidifferential $ B$, along with an open subset $ U \subset \Sigma$\footnote{With some conditions.}, such that $ dx$ and $ dy $ are regular and non-vanishing on $ \del U$. Let $ \mc{P} = \mc{S} \cap U$. Then the fibre-wise defined topological recursion correlators depend analytically on the base.
\end{theorem}

This topological recursion constructs KP $\tau$-functions. 

\begin{theorem}[{\cite{ABDKS24KP0}, \cite[Theorem 6.4]{ABDKS25}}]
    Consider a spectral curve of genus zero, $ \Sigma = \P^1$, with standard bidifferential $ B = \frac{dz_1 dz_2}{(z_1 - z_2)^2}$. Let $ o \in \Sigma \setminus \mc{P}$, and $ \xi$ a local coordinate at $ o$. Expand
    \begin{equation}
        \omega_n (z_{\llbracket n \rrbracket} )
        =
        \delta_{n,2} \frac{d\xi_1 d\xi_1}{(\xi_1 - \xi_2)^2}
        +
        \sum_{k_1, \dotsc, k_n = 1}^\infty f_{k_1, \dotsc, k_n} \prod_{i=1}^n k_i \xi_i^{k_i-1} d\xi_i.
    \end{equation}
    Then
    \begin{equation}
        Z_{o,\xi}
        \coloneqq
        \exp \sum_{n=1}^\infty \frac{1}{n!} \sum_{k_1, \dotsc, k_n = 1}^\infty f_{k_1, \dotsc, k_n} \prod_{i=1}^n p_{k_i}
    \end{equation}
    is a KP $\tau$-function.
\end{theorem}

For higher-genus spectral curve, a similar result (first conjectured by \cite{BE12}) holds, but only for a modified, non-perturbative, version of the differentials \cite{ABDKS24KP+}. As we will only work with genus zero spectral curves, we will not recall the full statement here.

\begin{assumption}
    In the rest of this paper, we will only consider compact spectral curves of genus zero, i.e. $ \Sigma = \P^1$ and $ B = \frac{dz_1 dz_2}{(z_1 - z_2)^2}$. Moreover, we will always take $ \mc{P}$ to be the set of zeroes of $ dx$.
\end{assumption}

\section{Polynomiality and wall--crossing}
\label{PolyWC}

Let $g$ be a non--negative integer, fix positive integers $n,s>0$ and a tuple of positive integers $\underline{r}=(r_1,\dots,r_s)$ with

\begin{equation}
    \sum_{i=1}^s (r_i - 1)
    =
    2g-2+n.
\end{equation}

Moreover, define the hyperplane

\begin{equation}
    \mathcal{H}_{n,s}=\{(\mathbf{x},\underline{k})\in\mathbb{Z}^n\times\mathbb{Z}^s\mid \sum_{i=1}^nx_i=\sum_{i=1}^sk_i\}.
\end{equation}

Furthermore, we define the leaky resonance arrangement inside $\mathcal{H}_{n,s}$ given by the hyperplanes

\begin{equation}
    W_{IJ}=\left\{\sum_{i\in I}x_i=\sum_{j\in J}k_j\right\}
\end{equation}

for $I\subset[n],J\subset[s]$. We call the collection $(W_{IJ})_{I,J}$ the \emph{leaky resonance arrangement}. The connected components of its complement in $\mathcal{H}_{n,s}$ will be referred to as \emph{chambers}.

We are interested in the map

\begin{equation}
    h_g^{-,\underline{r}}
    \colon
    \mathcal{H}_{n,s}\to\mathbb{Q}
    \colon
    (\mathbf{x},\underline{k})\mapsto h^{\underline{k},\underline{r}}_{g;\mathbf{x}^+,\mathbf{x}^-}
\end{equation}

parametrising $(-,\underline{r})$--leaky Hurwitz numbers.

This map was studied in \cite{AKL25} for $k_1=\cdots=k_s$ and $r_1=\cdots=r_s$ and it was proved that it is piecewise polynomial with respect to the leaky hyperplane arrangement restricted to this subspace. Moreover, the authors of \emph{loc.cit.} derived wall crossing formulas w.r.t. this restriction, i.e. they expressed the difference of polynomiality in adjacent chambers in terms of leaky Hurwitz numbers with smaller input data. 

In this section, we provide a more general piecewise polynomiality result without the conditions of $k_1=\cdots=k_s$ and $r_1=\cdots=r_s$. In the same setting as \cite{AKL25}, i.e. $k_1=\cdots=k_s$ and $r_1=\cdots=r_s$, we derive wall--crossing formulae for connected completed cycles leaky Hurwitz numbers.
In order to achieve this, we adopt methods from tropical geometry previously employed to study the piecewise polynomiality of various Hurwitz numbers in \cite{CJM11,HL20}.

\subsection{Formal set--up}
We begin with rephrasing \cref{thm-corr} in terms of ordered graphs which we call leaky monodromy graphs.

\begin{definition}
    Let $g$ be a non-negative integer, $\textbf{x}\in\mathbb{Z}\backslash\{\textbf{0}\}$, $\underline{k}$, $\underline{r}=(r_1,\dots,r_s)$ tuples of non-negative integers. Then, we call a directed graph $\Gamma$ a \emph{leaky combinatorial cover of type $(g,\textbf{x},\underline{k},\underline{r})$} if:
    \begin{enumerate}
        \item $\Gamma$ has $s+n$ vertices
        \item $\Gamma$ has $n$ leaves with adjacent edges called ends. We label the ends by $x_1,\dots,x_n$. For $x_i>0$, we call the corresponding edge an in-end and for $x_i<0$, we call it an out-end.
        \item We denote the set of edges that are not ends by $E^0(\Gamma)$.
        \item We label the set of inner vertices, i.e. the vertices that are not leaves, by $v_1,\dots,v_s$ and assign a non-negative integer $g(v_i)$ to $v_i$ called its genus. Moreover, we require $r_i-1=2g(v_i)-2+\mathrm{val}(v_i)$.
        \item The internal vertices are ordered compatibly with the orientation of the edges.
        \item To each edge, we associate a positive integer $\omega(e)\in\mathbb{N}$. At $v_i$ the weights satisfy $\sum_{e\to v} \omega(e)-\sum_{v\to e}\omega(e)=k_i$, i.e. the difference of the sum of incoming edges and outgoing edges is $k_i$.
    \end{enumerate}
    Moreover, we call a leaky tropical cover \emph{$\underline{k}$-compatible} if the ordering of vertices is exactly $v_1<\cdots<v_n$.
\end{definition}

Then, we immediately obtain from \cref{thm-corr} that
\begin{equation}
    h^{\underline{k},\underline{r}}_{g,\mathbf{x}}
    =
    \sum_\Gamma\frac{1}{|\mathrm{Aut}(\Gamma)|}\phi_\Gamma,
\end{equation}
where the sum runs of all $\underline{k}$-compatible leaky covers of type $(g,\textbf{x},\underline{k},\underline{r})$ and we have
\begin{equation}
    \phi_\Gamma=\prod m_{v_i}\prod_{e\in E^0(\Gamma)}\omega(e)
\end{equation}
for
\begin{equation}
    m_{v_i}
    =
    (r_i-1)! m_{g(v_i)}(\textbf{x}_i^+,\textbf{x}_i^-),
\end{equation}
where $\textbf{x}_i^-$ ($\textbf{x}_i^+$) is the set of weights of incoming (resp. outgoing) edges at $v_i$.

\begin{definition}
    We call a leaky combinatorial cover of type $(g,\mathbf{x},\underline{k},\underline{r})$ an \emph{$\mathbf{x}$-graph} if we forget the edge directions and the vertex orderings.
\end{definition}

Now, we fix an $\textbf{x}$-graph $\Gamma$ and consider some reference orientation. With respect to this orientation, we define the differential $d\colon \mathbb{R}E(\Gamma)\to\mathbb{R}V(\Gamma)$ that sends an edge to its head minus its tail. Thus, we obtain the following short exact sequence

\begin{equation}
    0\to\mathrm{ker}(d)\to\mathbb{R}E(\Gamma)\to \im (d)\to 0.
\end{equation}

We decompose $\mathbb{R}V(\Gamma)=\mathbb{R}^n\oplus\mathbb{R}V^0(\Gamma)$ into the contribution of ends and internal vertices. Then, we have that $(\mathbf{x},\underline{k})\in\im (d)$, when $\sum x_i=\sum k_i$.

\begin{definition}
    We define the space of \emph{flows} as
    \begin{equation}
        \mathcal{F}_\Gamma(\textbf{x},\underline{k})=d^{-1}(\textbf{x},\underline{k})
    \end{equation}

    and a hyperplane arrangement $\mathcal{A}_\Gamma(\textbf{x},\underline{k})$ in $\mathcal{F}_\Gamma(\textbf{x},\underline{k})$ as the restriction of the coordinate hyperplanes of $\mathbb{R}E(\Gamma)$ corresponding to the internal edges. In other words $\mathcal{A}_\Gamma(\textbf{x},\underline{k})$ is defined by the polynomial
    \begin{equation}
        \phi_{\mathcal{A}_\Gamma(\textbf{x},\underline{k})}=\prod e_i,
    \end{equation}
    where $e_i$ are the coordinate functions on $\mathbb{R}E(\Gamma)$ on $\mathcal{F}_\Gamma(\textbf{x},\underline{k})$.
\end{definition}

We now collect a number of facts that were proved in the case of classical double Hurwitz numbers in \cite{CJM11}. The same arguments apply for our more general case. We indicate in brackets the corresponding results.

\begin{lemma}[{\cite[Lemma 2.13, Corollary 2.14]{CJM11}}]
\label{lem-bound}
    The bounded chambers of $\mathcal{A}_\Gamma(\textbf{x},\underline{k})$ correspond to orientations of $\Gamma$ with no directed cycles. They are in bijection with leaky combinatorial covers projecting to $\Gamma$ after forgetting the edge directions and vertex orderings.
\end{lemma}

For a fixed $\textbf{x}$-graph $\Gamma$, we call the chambers of $\mathcal{A}_\Gamma(\textbf{x},\underline{k})$ the $F$-chambers of $\Gamma$. For an $F$-chamber $A$, denote by $\Gamma_A$ the corresponding directed $\textbf{x}$-graph and denote by $m_A$ ($m_A^{\textrm{com}}$) the numbers of vertex orderings turning $\Gamma_A$ into a (resp. $\underline{k}$-compatible) leaky combinatorial cover.

We obtain the following result that follows from \cref{lem-bound}.

\begin{lemma}
    We have that $m(A)$ is zero if and only if $A$ is unbounded.
\end{lemma}

Let $\mathrm{Ch}(\mathcal{A}_\Gamma(\textbf{x},\underline{k}))$ be the set of $F$-chambers of in $\mathcal{A}_\Gamma(\textbf{x},\underline{k})$. The sign of $\phi_{\mathcal{A}_\Gamma(\textbf{x},\underline{k})}=\prod_{i=1}^{|E(\Gamma)|}e_i$ alternates on adjacent $F$-chambers. We denote $\mathrm{sgn}(A)=(-1)^{N(A)}$, where $N(A)$ is the number of negative coordinates $e_i$ in $A$.

We define $\Lambda=\mathbb{Z}E(\Gamma)$ a lattice in $\mathbb{R}E(\Gamma)$ and
\begin{align}
    \begin{split}
        S_\Gamma(\textbf{x},\underline{k})=&\frac{1}{|\mathrm{Aut}(\Gamma)|}\sum_{A\in\mathrm{Ch}(\mathcal{A}_\Gamma(\textbf{x},\underline{k}))}m^{\mathrm{com}}(A)\sum_{f\in A\cap \Lambda}\prod_{e\in E^0(\Gamma)}\omega(e)\prod_{v\in V^0(\Gamma)}m_v\\
        =&\frac{1}{|\mathrm{Aut}(\Gamma)|}\sum_{A\in\mathrm{Ch}(\mathcal{A}_\Gamma(\textbf{x},\underline{k}))}m^{\mathrm{com}}(A)\mathrm{sign}(A)\sum_{f\in A\cap \Lambda}\phi'_{\mathcal{A}_\Gamma}(f)\prod_{v\in V^0(\Gamma)}m_v,
    \end{split}
\end{align}

where $\phi'_{\mathcal{A}_\Gamma}(f)=\prod_{e_i\in E^0(\Gamma)} e_i$

The following remark is an important observation.

\begin{remark}
Recall that $m_v=[z^{2g(v)}]\frac{\prod\mathcal{S}(x_iz)}{\mc{S}(z)}$, where $x_i$ are the weights of the adjacent vertices of $v$. Note that $\mathcal{S}(z)$ (and $1/\mathcal{S}(z)$) is an even function, which is why $m_v$ does not depend on the sign of the flow of $\Gamma$. Moreover, it is a polynomial in the $x_i$ of degree $2g(v)$.
\end{remark}

\subsection{Results}

To begin with, we observe that
\begin{equation}
\label{equ-poly}
    \frac{1}{|\mathrm{Aut}(\Gamma)|}\phi_{\mathcal{A}_\Gamma(f)}\prod_{v\in V^0(\Gamma)}m_v
\end{equation}

 is a polynomial in the coordinates of $\mathbb{R}E(\Gamma)$ of degree $|E^0(\Gamma)|+2\sum g_i$.

 We compute
 \begin{equation}
     |E^0(\Gamma)|=s-1+h^1(\Gamma)=s-1+g-\sum g_i
 \end{equation}

 and therefore 
 \begin{equation}
     |E^0(\Gamma)|+2\sum g_i=s-1+g+\sum g_i.
 \end{equation}
 We obtain that \cref{equ-poly} is a polynomial of degree $s-1+g+\sum g_i$.

 It is well-known, see e.g. \cite[Theorem 4.6]{AB17}, that summing a polynomial of degree $n$ over the lattice points of a $d$ dimensional lattice polytope with fixed topology yields a polynomial in the defining equations of that polytope of degree $n+d$. This is exactly the situation we are in. The polytopes given by the $F$-chambers are indeed lattice polytopes because the incidence matrix of a directed graph is unimodular.

 By the same argument as in \cite[Remark 2.11]{CJM11}, after fixing $\underline{k}$ each $F$-chamber $A$ is of dimension $h^1(\Gamma)=g-\sum g_i$.
 Thus, we obtain that $S_\Gamma(\textbf{x},\underline{k})$ is a polynomial of degree
 \begin{equation}
     g-\sum g_i+s-1+g+\sum g_i
     =
     s-1+2g
     =
     1+\sum_{i=1}^s r_i -n
 \end{equation}
 as long as the topology of $\mathcal{A}_\Gamma(\textbf{x},\underline{k})$ does not change.

 The only thing that remains to show is that the topology of $\mathcal{A}_\Gamma(\textbf{x},\underline{k})$ only changes when $(\textbf{x},\underline{k})$ crosses a wall in the leaky resonance arrangement.

 For generic hyperplane arrangements, the topology changes when one passes through non-transversalities. In our setting however, there are situations where there are always non-transversalities present. However, the topology still can only change, when one passes through additional non-transversalities. We give the following definitions.

 \begin{definition}
     Given $H_1,\dots,H_m$ hyperplanes in $\mathcal{A}_\Gamma(\textbf{x},\underline{m})$ that intersect in codimension $m-l$. This intersection is called \emph{good} if there exists a set $L$ of $l$ vertices in $\Gamma$, such that $H_1,\dots,H_m$ exactly correspond to the set of edges adjacent to the vertices in $L$. Moreover, we define the \emph{discriminant locus} $\mathcal{D}$ as the locus of points $(\textbf{x},\underline{k})\in\mathbb{R}^{n+s}$, such that $\mathcal{A}_\Gamma(\textbf{x},\underline{m})$ has non-good non-transverse intersections. The discriminant induces a hyperplane arrangement that we call the \emph{discriminant arrangement}.
 \end{definition}

 Our goal now is to classify the hyperplane in the discriminant arrangement.

 \begin{definition}
     A \emph{simple cut} of a graph $\Gamma$ is a minimal set $C$ of edges, that disconnects the ends of $\Gamma$. For an $\textbf{x}$-graph, a flow in $\mathcal{F}(\textbf{x},\underline{k})$ is \emph{disconnected} if there is a simple cut $C$, such that the flow on each edge in $C$ is zero.
 \end{definition}

We need a final lemma.

\begin{lemma}[{\cite[Lemma 3.8]{CJM11}}]
   We have that $(\textbf{x},\underline{k})\in\mathcal{D}$ if and only if $\mathcal{F}(\textbf{x},\underline{k})$ admits a disconnected flow.
\end{lemma}

Let $\Gamma$ admit a disconnected flow and consider the corresponding simple cut $C$. The leaky balancing condition immediately implies that the ends of each connected component of $\Gamma\backslash C$ satisfy $\sum_{i\in I} x_i+\sum_{j\in J}k_j$ for some $I$ and $J$. Thus, the discriminant arrangement coincides with the leaky resonance arrangement. Thus, the topology of $\mathcal{A}_\Gamma(\textbf{x},\underline{k})$ only changes when crossing a wall as desired. 

To summarise, we have proved the following:

\begin{theorem}
\label{thm-piecepoly}
    Let $g$ be a non--negative integers, fix positive integers $n,s>0$ and a tuple of positive integers $\underline{r}=(r_1,\dots,r_s)$ with $\sum_{i=1}^s (r_i - 1) = 2g-2+n$.
    
    Then, the map $h_g^{-,\underline{r}}$ is piecewise polynomial with respect to the leaky resonance arrangement. Restricted to each chamber, it is a polynomial of degree $1+\sum_{i=1}^s r_i -n$.
\end{theorem}

\begin{remark}
\begin{itemize}
    \item For $k_1=\cdots=k_s$ and disconnected numbers, \cref{thm-piecepoly} was proved in \cite{AKL25}. A version for $r=1$ was previously derived in \cite{CMR25}.
    \item The polynomiality statement in \cref{thm-piecepoly} includes the variables $k_1,\dots,k_s$. This phenomenon of \emph{universal} piecewise polynomial was previously observed in \cite{ardila2012universal,AB17} for Severi degrees and Gromov--Witten invariants of certain toric surfaces.
\end{itemize}

\end{remark}

A natural question that arises from piecewise polynomiality results on Hurwitz numbers is that of wall--crossing formulae. The idea is to consider adjacent chambers $C_1$ and $C_2$ separated by a hyperplane $\delta=0$. The counting problem is then described by a polynomial $P_{\mathfrak{c}_1}$ in $\mathfrak{c}_1$ a polynomial $P_{\mathfrak{c}_2}$ in $\mathfrak{c}_2$. The goal is to express the difference $P_{\mathfrak{c}_2}-P_{\mathfrak{c}_1}$ via Hurwitz numbers with smaller input data. Tropical geometry has proved to be a powerful tool in this direction \cite{CJM11,HL20}. The key idea is to observe that different monodromy graphs appear in $\mathfrak{c}_1$ and $\mathfrak{c}_2$. These graphs may be related via cutting up and reglueing. The cut-up monodromy graphs then contribute to smaller Hurwitz numbers.

We wish to adapt these techniques to the setting of completed cycles leaky Hurwitz numbers. There is one subtlety to note, namely the cutting--and--gluing technique above crucially requires that the monodromy graphs considered allow interchanging vertices. For our setting, this would be equivalent to the operators $\mathcal{F}_{r_i+1}^{k_i}$ commuting. Thus, we restrict ourself to the situation of $r=r_1=\cdots=r_s$ and $k=k_1=\cdots=k_s$. We denote the corresponding completed cycles $((k,\dots,k),(r,\dots,r))$--leaky Hurwitz numbers by $h_g^{k,r}(\mathbf{x})$.

In this setting the leaky resonance arrangement is described by a wall $W_{Ij}=\{\sum_{i\in I}x_i=jk\}$ for $I\subset [n]$, $j\in[s]$. Let $\mathfrak{c}_1,\mathfrak{c}_2$ be adjacent chambers separated by the wall $W_{Ij}$ and define $\delta=\sum_{i\in I}x_i-jk$. The goal is to describe

\begin{equation}
    P_{\mathfrak{c}_2}-P_{\mathfrak{c}_1}.
\end{equation}

We prove the following theorem.

\begin{theorem}
    \label{thm-WC}
    We have
    \begin{equation}
        P_{\mathfrak{c}_2}-P_{\mathfrak{c}_1}
        =
        \sum_{\alpha_1+\alpha_2+\alpha_3=s}
        \sum_{\substack{\mathbf{y},\mathbf{z}\colon\\ |\mathbf{x}_I|=|\mathbf{y}|+\alpha_1k
        \\
        |\mathbf{y}|=|\mathbf{z}|+\alpha_2 k
        \\
        |\mathbf{z}|=|\mathbf{x}_{I^{c}}|+\alpha_3k}}(-1)^{\alpha_2}\frac{\prod y_i}{\ell(\mathbf{y})!}\frac{\prod z_i}{\ell(\mathbf{z})!}h^{k,r}_{g_1;\mathbf{x}_I,\mathbf{y}} h^{k,r}_{g_2;\mathbf{y},\mathbf{z}} h^{k,r}_{g_3;\mathbf{z},\mathbf{x}_{I^c}},
    \end{equation}

    where $\mathbf{y},\mathbf{z}$ are ordered tuples of positive integers with sums $|\mathbf{y}|$ and $|\mathbf{z}|$ and $g_1=\frac{\alpha_1 r+2-\ell(\mathbf{x}_I,-\mathbf{z})}{2}$ (and analogously for $g_2,g_3$).
\end{theorem}

\begin{remark}
    There is another wall-crossing formula for leaky completed cycles Hurwitz numbers, \cite[Theorem 6.1]{AKL25}, which is quadratic in the Hurwitz numbers, and generalises \cite[Theorem 1.5]{Joh15}. It is not clear to us how these two formulas compare.
\end{remark}

\begin{proof}
    The argument is very similar to the case of classical double Hurwitz number in \cite{CJM11} and to that of tropical Jucys covers in \cite[section 4]{HL20}. Thus, we only give a sketch.

    The key idea is to track how the contribution of $F$--chambers changes when crossing a wall. This is a formalised by a polyhedral Gauss--Manin connection. We collect the important intermediate steps.

    To begin with, we define a notion of certain cuts on $\mathbf{x}$--graphs with an orientation. For this, consider such a graph $\Gamma$ with a subset of edges $\mathfrak{E}$. Then, we define a new graph whose vertices are connected components of $\Gamma\backslash\mathfrak{E}^c$ and whose edges are $\mathfrak{E}$. We denote this new graph by $\Gamma/\mathfrak{E}$ and call it the contraction of $\Gamma$ w.r.t. $\mathfrak{E}$.
    
    Now, given such a graph $\Gamma$, we define an $I$--cut to be a subset of edges $\mathfrak{E}$, such that either $\mathfrak{E}=\emptyset$ or
    
    \begin{enumerate}
        \item $\Gamma\backslash \mathfrak{E}$ is disconnected;
        \item the ends of $\Gamma$ lie on exactly two components of $\Gamma\backslash \mathfrak{E}$, one containing $I$ and the other $I^c$; and
        \item the contraction $\Gamma/\mathfrak{E}$ is acyclic with the initial vertex being the component containing $I$ and the final vertex the component containing $I^c$.
    \end{enumerate}

    We denote the set of $I$--cuts of $\Gamma$ by $\mathrm{Cut}_I(\Gamma)$.

    The next step is define the rank of an $I$--cut. This is simply one less than the number $v(\Gamma,\mathfrak{E})$ of connected components of $\Gamma\backslash\mathfrak{E}$. We write
    \begin{equation}
        r(\mathfrak{E})=v(\Gamma,\mathfrak{E})-1.
    \end{equation}

    We denote the set of bounded $F$--chambers for $\mathbf{x}$ and an $\mathbf{x}$--graph $\Gamma$ by $\mathcal{BC}_{\Gamma}(\mathbf{x})$. Each $A\in\mathcal{BC}_\Gamma(\mathbf{x})$ gives $\Gamma$ and orientation. We denote the thus obtained oriented $\mathbf{x}$--graph $\Gamma_A$.
    We fix $\mathbf{x}\in\mathfrak{c}_2$. A careful examination of the Gauss--Manin connection then proves that
    \begin{equation}
    \label{equ-wallcrossstep}
        P_{\mathfrak{c}_2}-P_{\mathfrak{c}_1}
        =
        \sum_{\Gamma}\sum_{A\in\mathcal{BC}_{\Gamma}(\mathbf{x})}\sum_{\mathfrak{E}\in\mathrm{Cut}_I(\Gamma_A)}(-1)^{\mathrm{r}(\mathfrak{E})-1}\binom{s}{\alpha_1, s_1,\dots, s_{r(\mathfrak{E})-1},\alpha_3}\sum_{\mathrm{int}(A)}\phi_{\mathcal{A}_\Gamma}\prod m_v,
    \end{equation}
    where $\alpha_1$ is the number of vertices in the initial component of $\Gamma_A\backslash\mathfrak{E}$, $\alpha_3$ the number of vertices in the final component and $s_i$ the number of vertices in the inner components. Furthermore, $\mathrm{int}(A)$ denotes the integer points in the $F$--chamber $A$.

    An inclusion--exclusion process then simplifies this formula to an expression in terms of a small subset of cuts. More precisely, we call a cut $T$ a \emph{thin cut} if all edges in $T$ are either adjacent to the initial or the final component. Given a thin cut $T$, the set of all cuts containing it is denoted by $P(T)$. Each cut contains a unique thin cut. One then obtains (non--trivially) for a given thin cut $T$ that
    \begin{equation}
    \label{equ-inclexcl}
        (-1)^{\alpha_3}\binom{s}{\alpha_1,\alpha_2,\alpha_3}=\sum_{\mathfrak{E}\in P(T)} (-1)^{r(\mathfrak{E})-1}\binom{s}{\alpha_1, s_1,\dots ,s_{r(\mathfrak{E})-1},\alpha_3}.
    \end{equation}

    Finally, we observe that a thin cut cuts a graph $\Gamma$ into three components $\Gamma_1,\Gamma_2,\Gamma_3$ with $\alpha_i$ many inner vertices. Denoting $\Gamma_1$ as a $(\mathbf{x}_I,-\mathbf{y})$--graph, $\Gamma_2$ as a $(\mathbf{y},\mathbf{z})$--graph and $\Gamma_3$ as a $(\mathbf{z},\mathbf{x}_{I^c})$--graph, we obtain
    \begin{equation}
    \label{equ-cutgraphs}
        \phi_{\mathcal{A}_\Gamma}=\binom{s}{\alpha_1,\alpha_2,\alpha_3}\frac{\prod y_i}{\ell(\mathbf{y})!}\frac{\prod z_i}{\ell(\mathbf{z})!}\phi_{\mathcal{A}_{\Gamma_1}}\phi_{\mathcal{A}_{\Gamma_2}}\phi_{\mathcal{A}_{\Gamma_3}}.
    \end{equation}

    Plugging \cref{equ-inclexcl,equ-cutgraphs} into \cref{equ-wallcrossstep} proves the theorem.
\end{proof}

\section{Closed expressions}
\label{ClosedFormulaeSec}
In this section, we study the two initial cases $(g,n)=(0,1),(0,2)$ of completed cycles leaky Hurwitz numbers. In particular for $\underline{r}=(1^s)$, we do recover leaky Hurwitz numbers in the sense of \cite{CMR25}. Our main goal is to derive closed expressions in these cases.

We note that most computations in this section are superseded by the more general TR machinery as carried out in \cref{ExplicitSection,sec:TR}. However, we believe this section is still of independent interest as it shows how the recursive structure of tropical covers allows a hands-on derivation of closed expressions via generatingfunctionology. It was also our preferred method for exploring the theory; the graphical interface of tropical graphs is especially suited for this.

\subsection{The case \texorpdfstring{$(0,1)$}{(0,1)}}

We begin with a few basic properties of the numbers $\lh^{k,r,q}_{0;\mu}$. First, we note that a simple calculation shows that

\begin{equation}
    2
    =
    \frac{|\mu|}{q}+\ell(\mu)-(r - 1+\frac{k}{q})b.
\end{equation}

In particular, when $\ell(\mu)=1$, i.e. $\mu=(m)$,

\begin{equation}
\label{equ-gen0branch}
    m
    =
    (q(r-1)+k)b+q.
\end{equation}

We obtain the following recursion for the one--part leaky orbifold Hurwitz numbers.

\begin{proposition}\label{01Recursion}
	The $ (k,r,q)$--completed cycles leaky orbifold Hurwitz numbers satisfy the following recursion:
	\begin{equation}
		\lh^{k,r,q}_{0;(m)} 
		=\frac{1}{r}
		\sum_{\substack{b_1, \dotsc, b_{r} \\ \sum b_j = b - 1}} \binom{b - 1}{b_1, \dotsc, b_{r}} \prod_{j=1}^r m_j \lh^{k,r,q}_{0;(m_j)} ,
	\end{equation}
    where $b$ and $m$, as well as $b_j$ and $m_j$ are related via \cref{equ-gen0branch}. The initial condition is $ \lh^{k,r,q}_{0,(q)} = \frac{1}{q}$.
\end{proposition}

\begin{proof}
	The proof is a standard tropical argument. Recall the expression of $\lh^{k,r,q}_{0;(m)}$ in terms of tropical covers.  Given a cover contributing to the left side, we cut off its left-most vertex to obtain $r$ covers, with a total number of vertices one less, which is reflected in the sum. Since we are in genus $0$, the multiplicity of the removed vertex is $(r-1)!$. However, we are overcounting by $r!$ due to the ordering of the $b_i$. Thus, we obtain a prefactor of $\frac{1}{r}$. The multinomial is required to record the relative position of all the vertices along the line.
\end{proof}

Next, we define the generating function

\begin{equation}
    W_{0,1}(X)
    =
    \sum_{m=0}^\infty\frac{\lh^{k,r,q}_{0;(m)}}{b!}X^m,
\end{equation}

where we note that by \cref{equ-gen0branch}, $b$ and $m$ determine each other. Moreover, we define

\begin{equation}
    Y(X)=\frac{dW_{0,1}(X)}{dX}
    =
    \sum_{m=0}^\infty m\frac{\lh^{k,r,q}_{0;(m)}}{b!}X^{m-1}.
\end{equation}

We obtain the following proposition.

\begin{proposition}
\label{prop-01spectral}
	The formal power series $Y(X)$ satisfies the algebraic relation
	\begin{equation}\label{SpectralCurve}
		X^{(q-1) \kappa}
        =
        Y^\kappa \big( 1 - \frac{k}{r} Y^{r-1} X^{r-1+k} \big)^{\rho + \kappa} ,
	\end{equation}
	with $ \kappa $ and $ \rho$ given by $ g = \gcd (qr, k) $ and $ \rho = \frac{qr}{g}$, $ \kappa = \frac{k}{g}$.\par
	The differentials $ dY$ and $ dX$ are related by
	\begin{equation}\label{DiffEqFory}
		(1 - (q(r-1)+k)Y^{r-1} X^{r-1+k} ) \frac{dY}{Y}
        =
        \Big( (q-1) + \frac{(q(r-1)+k)(r+k)}{r} Y^{r-1} X^{r-1+k} \Big) \frac{dX}{X}.
	\end{equation}
\end{proposition}

\begin{remark}
    If we take the $ \kappa$-th root of \cref{SpectralCurve},
    \begin{equation}
		X^{(q-1)}
        =
        Y \big( 1 - \frac{k}{r} (XY)^{r-1} X^{k} \big)^{\frac{qr}{k} + 1} ,
	\end{equation}
    we can take the $ k \to 0$ limit, to find
    \begin{align}
		X^{(q-1)} 
        &= 
        Y \Big( \exp \big( - \frac{1}{r} (XY)^{r-1} \big) \Big)^{qr }
        \\
        X^q 
        &= 
        XY e^{ - q (XY)^{r-1} }
	\end{align}
    Changing to $ x = \log X$ and $ y = XY$, we then find
    \begin{equation}
        x
        =
        \frac{1}{q} \log y - y^{r-1},
    \end{equation}
    which is the spectral curve for (non-leaky) $ r$-completed cycles $q$-orbifold Hurwitz numbers \cite{DKPS23}. Note that we use a different convention for $r$ compared to \cite{DKPS23}: our $ r-1$ is their $ r$.
\end{remark}

\begin{proof}
	Our starting point is \cref{01Recursion}.
	\begin{equation}
		\begin{split}
			Y(X) 
		      &
		      =
		      \sum_{b=0}^\infty \frac{m}{b!} \lh^{k,r,q}_{0;(m)} X^{m-1}
		      \\
		      &
		      =
    		\frac{1}{r} \sum_{b=0}^\infty \frac{m}{b!} \sum_{\substack{b_1, \dotsc, b_{r} \\ \sum b_j = b - 1}} \binom{b - 1}{b_1, \dotsc, b_{r}} \prod_{j=1}^{r} m_j \lh^{k,r,q}_{0;(m_j)} X^{m-1} + X^{q-1}
		      \\
		      &
		      =
		      \frac{1}{r} \sum_{b=0}^\infty \frac{m}{b} \sum_{\substack{b_1, \dotsc, b_{r} \\ \sum b_j = b - 1}}  \prod_{j=1}^{r} \Big( \frac{m_j}{b_j!} \lh^{k,r,q}_{0;(m_j)} X^{m_j-1} \Big) X^{r+k} + X^{q-1} .
		\end{split}
	\end{equation}
	Since $ m = (q(r-1)+k)b+q$, we can view both the multiplication by $m $ and $b$ as derivatives acting on the generating series:
	\begin{equation}
		X^{q-1} \frac{1}{q(r-1)+k} X \frac{d}{dX} X^{1-q} (Y - X^{q-1})  = \frac{1}{r} \frac{d}{dX} \Big( Y^{r} X^{r+k}\Big).
	\end{equation}
	This turns into
	\begin{equation}
		Y' - (q-1) \frac{Y}{X}
		=
		\frac{q(r-1)+k}{r} Y^{r} X^{r-1+k} \Big(r Y' + (r+k) \frac{Y}{X} \Big)
	\end{equation}
	which after reordering gives \cref{DiffEqFory}.\par
	Define for now a formal power series $\bar{Y}$ in $ X$ by
	\begin{equation}\label{FormalSC}
		\bar{Y} = X^{q-1}\big( 1 - \frac{k}{r} \bar{Y}^{r-1} X^{r-1+k} \big)^{- \frac{qr}{k} - 1}  .
	\end{equation}
	This formula uniquely defines $\bar{Y}$ as a power series, because it encodes a recursion for its coefficients. Since $\bar{Y}$ clearly satisfies \eqref{SpectralCurve}, we want to prove that $ y = \bar{Y}$.\par
	We now take the logarithmic derivative of \eqref{FormalSC}, which must also hold.
	\begin{equation}\label{DerivativeSC}
		\frac{\bar{Y}'}{\bar{Y}} 
        =
        (q-1) \frac{1}{X} + \frac{qr+k}{r+1} \frac{\bar{Y}^{r-1} X^{r-1+k} }{ 1 - \frac{k}{r} \bar{Y}^{r-1} X^{r-1+k} } \Big( (r-1) \frac{\bar{Y}'}{\bar{Y}} + (r-1+k) \frac{1}{X} \Big).
	\end{equation}
	Reordering and simplifying a bit gives
	\begin{align}
		\Big( 1 - \frac{(qr+k)(r-1)}{r} \frac{\bar{Y}^{r-1} X^{r-1+k} }{ 1 - \frac{k}{r} \bar{Y}^{r-1} X^{r-1+k} } \Big) \frac{\bar{Y}'}{\bar{Y}} 
		&=
		\Big( (q-1) + \frac{(qr+k)(r-1+k)}{r} \frac{\bar{Y}^{r-1} X^{r-1+k} }{ 1 - \frac{k}{r} \bar{Y}^{r-1} X^{r-1+k} } \Big) \frac{1}{X}
		\\
		\frac{ 1 -(q(r-1) +k ) \bar{Y}^{r-1} X^{r-1+k} }{ 1 - \frac{k}{r} \bar{Y}^{r-1} X^{r-1+k} } \frac{\bar{Y}'}{\bar{Y}} 
		&=
		\frac{ q-1 + \frac{qr^2+rk + qrk + k^2 -qr - qk}{r} \bar{Y}^{r-1} X^{r-1+k} }{ 1 - \frac{k}{r} \bar{Y}^{r-1} X^{r-1+k} } \frac{1}{X}
	\end{align}
	which up to multiplying by $ \big( 1 - \frac{k}{r} \bar{Y}^{r-1} X^{r-1+k} \big) $ is the same as \eqref{DiffEqFory}.\par
	So we have two power series which both satisfy the same first-order differential equation, which is equivalent to the recursion for the coefficients of \cref{01Recursion}. Since both have as first coefficient $ 1$, in front of $X^{q-1}$, they must be equal.
\end{proof}

From \cref{prop-01spectral}, we may now derive a closed formula for $\lh^{k,r,q}_{0;(m)}$. For $r=2$, it specialises to the one found in \cite[Theorem 1]{CMS25}.

\begin{theorem}
    Let $r,k,m,q$ be positive integers and define $A=\frac{k}{r}$. Then, $ \lh^{k,r,q}_{0;(m)} $ is non-zero if and only if $b = \frac{m-q}{q(r-1)+k}$ is an integer, and in that case
    \begin{equation}
        \lh^{k,r,q}_{0;(m)}
        =
        \frac{b!}{m(b(r-1)+1)}\binom{\frac{m}{A}}{b}A^b.
    \end{equation}
\end{theorem}

\begin{proof}
    We take the $g$-th power of \cref{SpectralCurve} to obtain

    \begin{equation}
        X^{(q-1)k}
        =
        Y^k(1-\frac{k}{r}Y^{r-1}X^{r-1+k})^{qr+k}.
    \end{equation}

   We fix an $(r-1)$-st root of $ X$ and define a new variable $W=YX^{\frac{r-1+k}{r-1}}$. Taking $k$--th roots and rearranging, our spectral curve turns into

   \begin{equation}
       X^{\frac{q(r-1)+k}{r-1}}
       =
       W(1-\frac{k}{r} W^{r-1})^{\frac{qr+k}{k}}.
   \end{equation}

    We use the Lagrange--Bürmann formula to compute $W(X^{\frac{q(r-1)+k}{r-1}})$. We obtain
    
    \begin{equation}
        [(X^{\frac{q(r-1)+k}{r-1}})^n]W
        =
        \frac{1}{n}[W^{n-1}]\left(1-\frac{k}{r}W^{r-1} \right)^{-n\left(\frac{qr+k}{k}\right)}
    \end{equation}

    Shiftung $n$ by $1$ we obtain
    \begin{align}
        \begin{split}
            [(X^{\frac{q(r-1)+k}{r-1}})^{n+1}]W
            &=
            \frac{1}{n+1}[W^{n}]\left(1-\frac{k}{r} W^{r-1} \right)^{-(n+1)\left(\frac{qr+k}{k}\right)}
            \\
            &=
            \frac{1}{n+1}\binom{-(n+1)\left(\frac{qr+k}{k}\right)}{\frac{n}{r-1}}\Big(-\frac{k}{r}\Big)^{\frac{n}{r-1}}
        \end{split}
    \end{align}
    which may be rearranged, writing $ b = \frac{n}{r-1}$, to
    
    \begin{equation}
        [(X^{q(r-1)+k})^{b+\frac{1}{r-1}}]W
        =
        \frac{1}{b(r-1)+1}\binom{(b(r-1)+1)\left(\frac{qr+k}{k}\right)+b-1}{b} \Big(\frac{k}{r}\Big)^b.
    \end{equation}

    We recall that $Y = W X^{-\frac{r-1+k}{r-1}}$ and see

    \begin{equation}
        Y
        =
        W X^{-\frac{r-1+k}{r-1}}
        =
        X^{-\frac{r-1+k}{r-1}}\sum_{b\ge0}\frac{1}{b(r-1)+1}\binom{(b(r-1)+1)\left(\frac{qr+k}{k}\right)+b-1}{b}\Big(\frac{k}{r}\Big)^b X^{(q(r-1)+k)(\frac{b(r-1)+1}{r-1})}
    \end{equation}

    Computing
    \begin{equation}
        (q(r-1)+k)(\frac{b(r-1)+1}{r-1})-\frac{r-1+k}{r-1}
        =
        q(b(r-1)+1) + bk + \frac{k}{r-1} - \frac{k}{r-1} - 1
        =
        b (q(r-1)+k) + q - 1,
    \end{equation}
    we see that this is an integer, and we finally have
    \begin{equation}
        Y
        =
        \sum_{b\ge0} \frac{1}{b(r-1)+1} \binom{(b(r-1)+1)\left(\frac{qr+k}{k}\right)+b-1}{b} \Big(\frac{k}{r}\Big)^b X^{b(q(r-1)+k)+q-1}.
    \end{equation}
    We note that $(b(r-1)+1)\left(\frac{qr+k}{k}\right)+b-1=\frac{m}{A}$.
    Recalling the normalisation of $Y(X)$, we thus have for $m=b(q(r-1)+k)+q$ that the coefficient of $X^{m-1}$ of $Y$ has an extra factor of $\frac{m}{b!}$. Taking this into account, we obtain the desired expression for $\lh^{k,r,q}_{0;(m)}$.
\end{proof}

\subsection{The case \texorpdfstring{$(0,2)$}{(0,2)}}
In same flavour as in the previous section, we now derive a closed expression in case $(0,2)$. The idea is again to turn a recursion for $\lh_{0,(m_1,m_2)}^{k,r,q}$ arising from the involved tropical covers into an expression in terms of generating functions. We note that as above $m_1+m_2=m$ is determined by $b$ via the Riemann Hurwitz formula, we have 
\begin{equation}\label{RH02}
    \begin{split}
    	2
    	&=
    	\frac{m_1+m_2}{q}+2-(r-1+\frac{k}{q})b
    	\\
    	b
    	&=
    	\frac{m_1+m_2}{q(r-1)+k}
    \end{split}
\end{equation}

We have the following recursion.

\begin{proposition}
\label{prop-rec02}
  For $\ell(\mu)=2$, i.e. $\mu=(m_1,m_2)$ and $g=0$, the $(k,r,q)$-numbers satisfy the following recursion:
	\begin{equation} \label{Rec02}
    \begin{split}
       \lh_{0,(m_1,m_2)}^{k,r,q}
       =&
       \sum_{\sum_{j=1}^{r} \alpha_j=m_1-k} \binom{b-1}{b_1,\dots,b_{r}} \alpha_1 \lh_{0,(m_2,\alpha_1)}^{k,r,q} \prod_{j=2}^{r} \alpha_j \lh_{0,(\alpha_j)}^{k,r,q}\\
       &+
       \sum_{\sum_{j=1}^{r} \alpha_j=m_2-k} \binom{b-1}{b_1,\dots,b_{r}} \alpha_1 \lh_{0,(m_1,\alpha_1)}^{k,r,q} \prod_{j=2}^{r} \alpha_j \lh_{0,(\alpha_j)}^{k,r,q}\\
       &+
       \sum_{\sum_{j=1}^{r-1} \alpha_j = m_1 + m_2 - k} \binom{b-1}{b_1, \dotsc, b_{r-1}} \prod_{j = 1}^{r-1} \alpha_j \lh_{0,(\alpha_j)}^{k,r,q}\,.
    \end{split}
    \end{equation}
    Here, any $ b $ associated to a $ \lh_{0,(m_1,m_2)}$ is given by \cref{RH02}, while $ b $ associated to $ \lh_{0,(m)}$ are given by \cref{equ-gen0branch}\,.
\end{proposition}

\begin{proof}
  The proof is again a standard tropical consideration. We cut the first vertex. For any tropical leaky cover contributing to $h_{0,(m_1,m_2)}$, we can only have vertices with one incoming edge from the left and $r$ outgoing edges to the right and exactly one vertex with two incoming edges from the left and $r-1$ outgoing edges to the right. This is due to the first Betti number of the source graph being $0$.\par
  The first term corresponds to tropical covers whose first vertex cuts the strand weighted $m_1$, the second term corresponds to tropical covers whose first vertex cuts the strand weighted $m_2$. We obtain a factor of $\frac{1}{r}$ from the extra vertex as in \cref{01Recursion}, but also a factor $ \binom{r}{1}$ for the choice which factor is $ h_{0,2}$, and these cancel.\par
  The last term encapsulates the case of the first vertex joining the strands weighted $m_1$ and $m_2$. In this case there are $r-1$ identical right-pointing strands, so $ (r-1)!$ orderings, which compensates the vertex multiplicity exactly.
\end{proof}

Our next goal is to package the numbers $\lh_{0,(m_1,m_2)}^{k,r,q}$ into a generating function and to translate the recursion of \cref{prop-rec02} into a functional equation that we can solve.

For this, it turns out to be convenient to introduce a variable $z$ via

\begin{equation}
	z \coloneq X \big( 1 - \frac{k}{r} Y^{r-1} X^{r-1+k} \big)^{-\frac{r}{k}} .
\end{equation}
We will see later, in \cref{rkSC}, that this is a projective coordinate on the spectral curve, and therefore a very natural variable. However, we would like to note that we first found this variable from the considerations in this section.

An important property is given by the following lemma.

\begin{lemma}\label{LogDerz}
	The logarithmic derivative of $z$ is given by
	\begin{equation}
		\frac{dz}{z} = \frac{dX}{X(1- (qr+k) X^k (XY)^{r})}
	\end{equation}
\end{lemma}	

\begin{proof}
To begin with, we observe

\begin{equation}
    \frac{dz}{z}=\frac{dX}{X} + \frac{ Y^{r} X^{r+k}}{ 1 - \frac{k}{r+1} Y^{r} X^{r+k}} \Big( r \frac{dY}{Y} + (r+k) \frac{dX}{X} \Big)
\end{equation}
    Using \cref{prop-01spectral} to express $\frac{dY}{Y}$ in terms of $\frac{dX}{X}$ and some further algebra proves the lemma.
\end{proof}

We now introduce the generating function

\begin{equation}
W_{0,2}(X_1,X_2)=\sum_{l,m=1}^\infty\frac{\lh_{0,(l,m)}^{k,r,q}}{b!}X_1^lX_2^m.
\end{equation}

We get the following crucial identity.

\begin{proposition}\label{Semistable}
For any $k$, $r$, and $q$, we have the following:
\begin{equation}
\label{equ-02}
d_1d_2W_{0,2}(X_1,X_2)=\frac{dz_1dz_2}{(z_1-z_2)^2}-\frac{dX_1dX_2}{(X_1-X_2)^2}.
\end{equation}
\end{proposition}

\begin{proof}
We note that for any power series $ f(X) = \sum_{n=0}^\infty f_n X^n$, we have
    \begin{equation}
    	\begin{split}
    		\sum_{l,m =1}^\infty f_{l+m-k} X_1^l X_2^m 
    		&=
    		\sum_{n=1}^\infty f_n \sum_{l = 1}^{n+k-1} X_1^l X_2^{n+k-l} 
    		\\
    		&= 
    		-\sum_{n=0}^\infty f_n \frac{X_1^{n+k-1} - X_2^{n+k-1}}{X_1^{-1}-X_2^{-1}}
    		\\
    		&= - \frac{X_1^{k-1} f(X_1) - X_2^{k-1} f(X_2)}{X_1^{-1} - X_2^{-1}}\,.
    	\end{split}
    \end{equation}

    Using this, we may rephrase \cref{prop-rec02} to obtain

    \begin{align}
    \begin{split}
         &\frac{1}{q(r-1)+k} \Big( X_1 \frac{d}{dX_1} + X_2 \frac{d}{dX_2} \Big) W_{0,2} (X_1,X_2)
      =\\
      &X_1^{k+1} \frac{dW_{0,2}}{X_1} (X_1,X_2) (X_1 Y(X_1))^{r-1} + X_2^{k+1} \frac{dW_{0,2}}{dX_2}(X_1,X_2) (X_2 Y(X_2))^{r-1}
      \\
      &  - 
      \frac{ X_1^{k-1} (X_1 Y(X_1))^{r-1} - X_2^{k-1} (X_2 Y(X_2))^{r-1}}{X_1^{-1} - X_2^{-1}} \,.
      \end{split}
    \end{align}

    After re-ordering, we obtain

    \begin{align}
    \begin{split}
        &\Big(X_1^{r-1+k} Y(X_1)^{r-1} - \frac{1}{q(r-1)+k}\Big) X_1 \frac{dW_{0,2}}{dX_1} + \Big( X_2^{r-1+k} Y(X_2)^{r-1} - \frac{1}{q(r-1)+k} \Big) X_2 \frac{dW_{0,2}}{dX_2} \\
        &= \frac{X_1^{r+k-2} Y(X_1)^{r-2} - X_2^{r+k-2} Y(X_2)^{r-1}}{X_1^{-1} - X_2^{-1}} \,.
    \end{split}
    \end{align}

    Using \cref{LogDerz} and denoting the Euler operator $E_z = z_1 \frac{\partial}{\partial_{z_1}} + z_2\frac{\partial}{\partial_{z_2}}$ we finally obtain

    \begin{equation}
        E_zW_{0,2}=\frac{X_1^{r+k-2} Y(X_1)^{r-1} - X_2^{r+k-2} Y(X_2)^{r-1}}{X_1^{-1} - X_2^{-1}}.
    \end{equation}

    Applying $E_z$ to $\log (X_1^{-1} - X_2^{-1}) - \log (z_1^{-1} - z_2^{-1})$ yields the same right hand side. Inverting $E_z$ and applying $d_1d_2$ yields the proposition.
\end{proof}

Next, we express $X$ in $z$. Namely, we have after considering the spectral curve of \cref{prop-01spectral} that
\begin{equation}
    z^q=XY(1-\frac{k}{r}Y^{r-1}X^{r-1+k})
\end{equation}
and therefore that
\begin{equation}
    1+\frac{k}{r}z^{q(r-1)+k}
    =
    (1-\frac{k}{r} Y^{r-1}X^{r-1+k})^{-1}.
\end{equation}

Thus, we immediately see that
\begin{equation}
\label{equ-xviaz}
    X
    =
    z(1+\frac{k}{r}z^{q(r-1)+k})^{-\frac{r}{k}}.
\end{equation}

Similarly, we obtain

\begin{equation}
\label{equ-yviaz}
    Y
    =
    z^{q-1}(1+\frac{k}{r}z^{q(r-1)+k})^{\frac{r+k}{k}}
\end{equation}

With this, we now aim to find closed expressions for $\lh_{0,(l,m)}^{k,r,q}$. We observe that
\begin{equation}
    d_1 d_2 \log \big( \frac{X_1^{-1} - X_2^{-1}}{z_1^{-1} - z_2^{-1}} \big)
    =
    \frac{dz_1 dz_2}{(z_1 - z_2)^2} - \frac{dX_1 dX_2}{(X_1 - X_2)^2}.
\end{equation}

Thus, it suffices to compare $W_{0,2}$ with the expansion of $\log ( X_1^{-1} - X_2^{-1}) - \log ( z_1^{-1} - z_2^{-1} ) $. We introduce the Euler operator $E_X = X_1 \partial_{X_1} + X_2 \partial_{X_2} $ and observe that
\begin{equation}
    E_X W_{0,2}
    =
    \sum_{l,m=1}^\infty (l+m) \frac{\lh_{0,(l,m)}^{k,r,q}}{b!} X_1^l X_2^m,
\end{equation}

i.e. the generating function changes only by multiplying each coefficient with the degree $|\mu|$.

In order to extract $\lh_{0,(l,m)}^{k,r,q}$, we observe

\begin{equation}
    \lh_{0,(l,m)}^{k,r,q}
    =
    \frac{b!}{l+m}\Res_{X_1=0} \Res_{X_2=0} E_X W_{0,2} \frac{dX_1 dX_2}{X_1^{l+1} X_2^{m+1}}
\end{equation}

and with \cref{prop-rec02}

\begin{equation}
    \lh_{0,(l,m)}^{k,r,q}
    =
    \frac{b!}{l+m} \Res_{X_1=0} \Res_{X_2=0} E_X \log \big( \frac{X_1^{-1} - X_2^{-1}}{z_1^{-1} - z_2^{-1}} \big) \frac{dX_1 dX_2}{X_1^{l+1} X_2^{m+1}} .
\end{equation}

Since
\begin{equation}
    E_X \log (X_1^{-1} - X_2^{-1})
    =
    -1,
\end{equation}

the complexity of the problem lies in the computation of $E_X \log (z_1^{-1} - z_2^{-1})$.

\begin{example}
\label{ex-closedform}
We consider the case $(k,r,q)=(1,2,1)$. We have the spectral curve
\begin{equation}
    1
    =
    Y(1-\frac{X^2Y}{2})^3
\end{equation}
with the parametrisation

\begin{equation}
    X(z)
    =
    \frac{z}{(1+\frac{z^2}{2})^2}
    =
    \frac{4z}{(2+z^2)^2}
    \quad
    \textrm{and}
    \quad
    Y(z)
    =
    (1+\frac{z^2}{2})^3
    =
    \frac{(2+z^2)^3}{8}.
\end{equation}

We consider the Euler operator $ E_X = X_1 \partial_{X_1} + X_2 \partial_{X_2} $. In the coordinates $z_i$ for $X_i = X(z_i)$, we have

\begin{equation}
    E_X
    =
    \frac{z_1(2+z_1)^2}{(2-3z_1^2)} \partial_{z_1} + \frac{z_2(2+z_2)^2}{(2-3z_2^2)} \partial_{z_2}.
\end{equation}

We obtain

\begin{equation}
    E_X \log (X_1^{-1} - X_2^{-1})
    =
    -1
    \quad
    \textrm{and}
    \quad
    E_X \log (z_1^{-1} - z_2^{-1})
    =
    \frac{4z_1 z_2 (3z_1z_2 + 2)}{(3z_1^2 - 2) (3z_2^2 - 2)} - 1.
\end{equation}

Therefore, we have

\begin{equation}
    E_X \log \big( \frac{X_1^{-1} - X_2^{-1}}{z_1^{-1} - z_2^{-1}} \big)
    =
    \frac{4z_1 z_2 (3z_1 z_2 + 2)}{(3z_1^2 - 2) (3z_2^2 - 2)}.
\end{equation}

We compute

\begin{equation}
\begin{split}
    &\Res_{X_1=0} \Res_{X_2=0}  E_X  \log \big(\frac{X_1^{-1} - X_2^{-1}}{z_1^{-1} - z_2^{-1}} \big) \frac{dX_1 dX_2 }{X_1^{l+1} X_2^{m+1}}
    \\
    &\quad =
    \Res_{z_1=0} \Res_{z_2=0} \frac{(2 + z_1^2)^{2(l+1)} (2 + z_2^2)^{2(m+1)} 4 z_1 z_2 (3 z_1 z_2 + 2)}{(4z_1)^{l+1} (4z_2)^{m+1} (3z_1^2 - 2) (3z_2^2 - 2)}4 \frac{2 - 3z_1^2}{(2 + z_1^2)^3} dz_1 4 \frac{2 - 3z_2^2}{(2 + z_2^2)^3} dz_2
    \\
    &\quad =
    4 \Res_{z_1=0} \Res_{z_2=0} \frac{(2 + z_1^2)^{2l-1} (2 + z_2^2)^{(2m-1)}(3 z_1 z_2 + 2)}{(4 z_1)^l (4 z_2^m)} dz_1 dz_2.
\end{split}
\end{equation}

To obtain closed expressions, we extract the coefficient of $z_1^{-1} z_2^{-1}$ in the power series and observing that for $(k,r,q)=(1,2,1)$ and $\mu=(l,m)$, we have $b=\frac{l+m}{2}$. In total, we have the following formulae, that depend on the parities of $l$ and $m$:

\begin{align}
		\lh^{1,1,1}_{0,(2l_1,2l_2)}
		&=
		\frac{3}{2^{l_1+l_2-1}} \binom{4l_1-1}{l_1-1} \binom{4l_2-1}{l_2-1} (l_1 + l_2 - 1)! \label{h02EvenSimplecase}
		\\
		\lh^{1,1,1}_{0,(2l_1-1,2l_2-1)}
		&=
		\frac{1}{2^{l_1+l_2-2}} \binom{4l_1 -3}{l_1 -1} \binom{4l_2 -3}{l_2 -1} (l_1 + l_2 -2)!
		\\
		\lh^{1,1,1}_{0,(2l_1,2l_2-1)}
		&=
		0 .
	\end{align}
\end{example}

With this example under the belt, we now treat the general case. A similar calculation shows that for arbitrary $(r,k,q)$, we have

\begin{equation}
    E_X
    =
    \frac{z_1 (1 + \frac{k}{r} z_1^{q(r-1)+k})}{1 + (\frac{k}{r} - (q(r-1)+k)) z_1^{q(r-1)+k}} \partial_{z_1} + \frac{z_2 (1 + \frac{k}{r} z_2^{q(r-1)+k})}{1 + (\frac{k}{r} - (q(r-1)+k)) z_2^{q(r-1)+k}} \partial_{z_2}.
\end{equation}

We obtain after rearranging, writing $A=\frac{k}{r}$ and $B=q(r-1)+k$, that

\begin{equation}
    E_X \log \big( \frac{X_1^{-1} - X_2^{-1}}{z_1^{-1} - z_2^{-1}} \big)
    =
    \frac{B( \sum_{j=1}^{B-1} z_1^j z_2^{B-j} + (B-A) (z_1z_2)^B)}{(1 + (A - B) z_1^{B}) (1 + (A - B) z_2^{B})}.
\end{equation}

It remains to compute

\begin{equation}
    \Res_{X_1=0} \Res_{X_2=0}  \frac{B (\sum_{j=1}^{B-1} z_1^j z_2^{B-j} + (B - A) (z_1 z_2)^B)}{(1 + (A - B) z_1^{B}) (1 + (A - B) z_2^{B})} \frac{ dX_1 dX_2 }{X_1^{l+1} X_2^{m+1}}.
\end{equation}

Writing $X=z(1-Az^B)^{-\frac{1}{A}}$, we obtain

\begin{equation}
\begin{split}               
    \Res_{z_1=0} \Res_{z_2=0} & \frac{(1 + A z_1^B)^{\frac{l}{A} - 1}(1 + A z_2^B)^{\frac{m}{A} - 1} (B (\sum_{j=1}^{B-1} z_1^j z_2^{B-j} + (B - A) (z_1 z_2)^B)}{z_1^{l+1} z_2^{m+1}} dz_1 dz_2
    \\
    & \quad=
    B\sum_{j=1}^{B-1} \Res_{z_1=0} \Res_{z_2=0} \frac{(1 + A z_1^B)^{\frac{l}{A} - 1}(1 + A z_2^B)^{\frac{m}{A} - 1}}{z_1^{l+1-j} z_2^{m+1-B+j}} dz_1 dz_2
    \\
    & \qquad
    +B (B - A) \Res_{z_1=0} \Res_{z_2=0} \frac{(1 + A z_1^B)^{\frac{l}{A} - 1}(1 + A z_2^B)^{\frac{m}{A} - 1}}{z_1^{l+1-B} z_2^{m+1-B}} dz_1 dz_2.
\end{split}
\end{equation}

Extracting the $z_1^{-1}z_2^{-1}$ coefficient, we have proved the following formula, which generalises \cref{ex-closedform}.

\begin{theorem}
    We have for $A = \frac{k}{r}$ and $B = q(r-1)+k$ that
    \begin{equation}
        \lh_{0,(l,m)}^{k,r,q}
        =\frac{b!}{l+m} A^{\frac{l+m}{B}-2} B \begin{cases}
            0, &\textrm{if $l+m\not\equiv 0\pmod{B}$}\\
            A \binom{\frac{l}{A} - 1}{ \lfloor \frac{l}{B} \rfloor } \binom{\frac{m}{A} - 1}{ \lfloor \frac{m}{B} \rfloor } &\textrm{if $l\equiv - m \not\equiv 0 \pmod{B}$}\\           
            (B - A) \binom{\frac{l}{A} - 1}{\frac{l-B}{B}} \binom{\frac{m}{A} - 1}{\frac{m-B}{B}} &\textrm{if $l \equiv m \equiv 0\pmod{B}$}.
        \end{cases}
    \end{equation}
\end{theorem}

\section{Cut-and-join as Hamiltonian flow}
\label{CJSection}

In this section, we will use the Fock space expression for the multidifferentials of leaky Hurwitz numbers, \cref{def:Multidifferentials}, to obtain spectral curves through the flow generated by a Hamiltonian encoding the cut-and-join operator.

Let us consider a general cut-and-join operator
\begin{equation}
	W
	=
	\sum_{l \in \Z'} \sum_{k = 0}^\infty w_k (l + \frac{k}{2} ,\hbar) E_{l + k,l} ,
\end{equation}
where the $ w_k (j, \hbar)$ are analytic functions in $ j$ whose coefficients are power series in $ \hbar$, satisfying \cref{TopologicalAssumption}. Note that we restrict $ k $ to be non-negative. This is essential for our approach, because it guarantees that for a given left partition $ \mu$, there are only finitely many tropical covers contribution to the relevant Hurwitz numbers.


We define the generating functions
\begin{equation}\label{CJGenFunction}
	\mf{W}(z,j,\hbar )
	\coloneqq
	\sum_{k=0}^\infty w_k(j, \hbar ) z^k ,
	\qquad
	S(z)
	\coloneqq
	\sum_q s_q z^q
\end{equation}

Because $ k \ge 0$, the cut-and-join operator annihilates the covacuum. So $ \< 0| e^{-W}  = \< 0|$, and hence it makes sense to calculate $ e^{-W} J(X) e^W = e^{-\ad_W} J(X)$.

\begin{proposition}\label{CurrentConjugation}
	In this general setup,
	\begin{equation}
		e^{-\beta \ad_W} J_m
		=
		\sum_{l \in \Z'} \sum_{\mu \in \Z} \sum_{n=0}^\infty \Phi^{\mf{W}}_{n,\mu} (\mf{J}_m) \big( \hbar (l- \frac{\mu}{2} )\big) \beta^n E_{l-\mu,l} 
        + \sum_{n=0}^\infty \tilde{\Phi}^{\mf{W}}_n (\mf{J}_m) \beta^n c,
	\end{equation}
	where $ \Phi^{\mf{W}}_{n,\mu} (\mf{J}_{m})$ satisfies the recursion $ \Phi^{\mf{W}}_{0,\mu}(\mf{J}_{m}) (s) = \mf{J}_{m,\mu} \coloneqq \delta_{\mu,m}$ and
	\begin{equation}\label{PhiRecursion}
		(n+1) \Phi^{\mf{W}}_{n+1,\mu} (\mf{J}_{m}) (s) 
		=
		\sum_{k = 0}^\infty w_k (\frac{s}{\hbar} + \frac{\mu+k}{2}, \hbar ) \Phi^{\mf{W}}_{n, \mu+k}(\mf{J}_{m}) ( s + \frac{\hbar k}{2}) - w_k ( \frac{s}{\hbar} - \frac{\mu+k}{2}, \hbar ) \Phi^{\mf{W}}_{n,\mu+k} (\mf{J}_{m}) \big( s - \frac{\hbar k}{2}\big) ,
	\end{equation}
    while the central term is given by
    \begin{equation}\label{CentralTerm}
        (n+1) \tilde{\Phi}^{\mf{W}}_{n+1} (\mf{J}_{m})
        =
        (n+1) \frac{1}{\varsigma ( \hbar \del_s) } \Phi^{\mf{W}}_{n+1,0} (\mf{J}_m) (s) \Big|_{s=0}
        \coloneqq
        \sum_{k=0}^\infty \sum_{ l = \frac{1}{2}}^{k - \frac{1}{2}} w_k (l - \frac{k}{2}, \hbar ) \Phi^{\mf{W}}_{n,k} (\mf{J}_{m}) ( \hbar ( l - \frac{k}{2})).
    \end{equation}
\end{proposition}

\begin{remark}\label{WhatIsCentralTerm}
    The middle part of \cref{CentralTerm} is not well-defined: because $ \varsigma (z) = z + \mc{O} (z^3)$, we find that $ \frac{1}{\varsigma (\hbar \del_s)} = \frac{1}{\hbar \del_s} + \mc{O} (\del_s)$. We would thus need to take an antiderivative, which is only defined up to a constant, and then evaluate that constant. Rather, the middle expression is convenient \emph{notation} for the right one.\par
    It is also consistent with the original definition of the completed cycles in \cite{OP06a}: compare the current proof with \cite[Section 0.4.3]{OP06a}.
\end{remark}

\begin{proof}
	Consider
	\begin{equation}
	A = \sum_{l \in \Z'} \sum_{\mu \in \Z} a_\mu \big( \hbar (l -\frac{\mu}{2})\big) E_{l-\mu,l}.
	\end{equation} 
	Then direct calculation shows that
	\begin{equation}
        \begin{split}
		      [A,W]
		      &=
		      \sum_{l \in \Z'} \sum_{\mu \in \Z} \sum_{k=0}^\infty \Big( w_k (l + \frac{k}{2}, \hbar ) a_\mu \big( \hbar ( l + k - \frac{\mu}{2} ) \big) - w_k ( l - \mu + \frac{k}{2}, \hbar ) a_\mu \big( \hbar( l - \frac{\mu}{2} ) \big) \Big) E_{l - \mu + k,l}
            \\
            &\quad + \sum_{k=0}^\infty \sum_{ l = \frac{1}{2}}^{k - \frac{1}{2}} w_k (l - \frac{k}{2}, \hbar ) a_k ( \hbar ( l - \frac{k}{2})) c.
        \end{split}
	\end{equation}
	Writing
	\begin{equation} 
		A' 
		=
		[A,W]
		=
		\sum_{l \in \Z'} \sum_{\mu' \in \Z} a'_{\mu'} \big( \hbar ( l -\frac{\mu'}{2}) \big) E_{l - \mu',l} + a'' c,
	\end{equation} 
	we then get
	\begin{align}
		a'_{\mu'} ( s)
		&=
		\sum_{k=0}^\infty w_k (\frac{s}{\hbar} + \frac{\mu' +k}{2}, \hbar ) a_{\mu'+k} ( s + \frac{\hbar k}{2}) - w_k ( \frac{s}{\hbar} - \frac{\mu'+k}{2}, \hbar ) a_{\mu'+k} \big( s - \frac{\hbar k}{2}\big) ,
        \\
        a'' 
        &=
        \sum_{k=0}^\infty \sum_{ l = \frac{1}{2}}^{k - \frac{1}{2}} w_k (l - \frac{k}{2}, \hbar ) a_k ( \hbar ( l - \frac{k}{2})).
	\end{align}
	Therefore, writing
	\begin{equation}
		\frac{1}{n!} \ad_{-W}^n J_m 
		\eqqcolon 
		\sum_{l \in \Z'} \sum_{\mu \in \Z} \Phi^{\mf{W}}_{n,\mu} (\mf{J}_{m}) \big( \hbar (l -\frac{\mu}{2})\big) E_{l - \mu} + \tilde{\Phi}^{\mf{W}}_{n} (\mf{J}_m) c,
	\end{equation}
	we do find that $ \Phi^{\mf{W}}_{0,\mu} (\mf{J}_{m}) (s) = \delta_{m,\mu}$ and that its recursion holds.\par
    For the central term, we find that
    \begin{equation}
        \begin{split}
            (n+1) \tilde{\Phi}^{\mf{W}}_{n+1} (\mf{J}_{m})
            &=
            \sum_{k=0}^\infty \sum_{ l = \frac{1}{2}}^{k - \frac{1}{2}} w_k (l - \frac{k}{2}, \hbar ) \Phi^{\mf{W}}_{n,k} (\mf{J}_{m}) ( \hbar ( l - \frac{k}{2}))
            \\
            &=
            \sum_{k=0}^\infty \sum_{l \in \Z'_+} w_k (l - \frac{k}{2}, \hbar ) \Phi^{\mf{W}}_{n,k} (\mf{J}_{m}) ( \hbar ( l - \frac{k}{2})) - w_k (l + \frac{k}{2}, \hbar ) \Phi^{\mf{W}}_{n,k} (\mf{J}_{m}) ( \hbar ( l + \frac{k}{2}))
            \\
            &=
            - \sum_{l \in \Z'_+} e^{l \hbar \del_s} (n+1) \Phi^{\mf{W}}_{n+1,0} (\mf{J}_{m}) (s) \Big|_{s=0}.
        \end{split}
    \end{equation}
    The first equality is a finite, hence well-defined sum, because for any fixed $ m$, $ \Phi^{\mf{W}} ( \mf{J}_m)_{n,k} $ can only be non-zero for $ k \leq m$.
    On the other hand, formally we can rewrite
    \begin{equation}
        - \sum_{l \in \Z'_+} e^{l \hbar \del_s}
        =
        - e^{\frac{\hbar}{2} \del_s} \frac{1}{1 - e^{\hbar \del_s}}
        =
        \frac{1}{e^{\frac{\hbar}{2} \del_s} - e^{-\frac{\hbar}{2}\del_s}}
        =
        \frac{1}{\varsigma (\hbar \del_s)},
    \end{equation}
    which yields the result.
\end{proof}

The same proof in fact yields the following more general result.

\begin{proposition}\label{MoreGeneralConjugation}
    Let\footnote{The apparently inconsistent sign for exponents of $ z$ in this formula compared to that for $ \mf{W}$ in \eqref{CJGenFunction} comes from the minus sign in \eqref{LeakyCompletedOperator}, explained directly after that latter equation.}
    \begin{equation}
        \mf{P}(z,s) 
        = 
        \sum_{\mu \in \Z} \mf{P}_\mu (s)z^{-\mu}
        \in
        \C \llbracket z, s, \hbar \rrbracket [z^{-1} ]
    \end{equation}
    be such that $ \mf{P}_0 (s) = 0$, and define
    \begin{equation}\label{GeneralFunctionFlow}
		\Phi^{\mf{W}} (\mf{P}) (z,s,\beta) 
		\coloneqq
		\sum_{n =0}^\infty \sum_{\mu \in \Z} \Phi^{\mf{W}}_{n,\mu} (\mf{P})(s) \beta^n z^{-\mu}
        \in \C \llbracket s, \hbar \rrbracket [ z^{\pm}] \llbracket \beta \rrbracket,
	\end{equation}
    through the recursion \eqref{PhiRecursion} with the initial condition $ \Phi^{\mf{W}}_{0,\mu} (\mf{P})(s) = \mf{P}_\mu(s)$.
    Then
    \begin{align}
        \Res_{z=0} \mc{E}(\hbar \del_s , z) \mf{P} (z,s) \Big|_{s=0}
        &=
        \sum_{l \in \Z' } \sum_{\mu \in \Z} \mf{P}_\mu (\hbar(l-\frac{\mu}{2})) E_{l-\mu,l},
        \\
        \begin{split}
            \Res_{z=0} \mc{E}(\hbar \del_s , z) \Phi^{\mf{W}}(\mf{P}) (z,s, \beta ) \Big|_{s=0}
            &=
            \sum_{l \in \Z' } \sum_{\mu \in \Z} \Phi^{\mf{W}}_{n,\mu} (\mf{P}) \beta^n (\hbar(l-\frac{\mu}{2})) E_{l-\mu,l}
            \\
            & \qquad \qquad
            +
            \frac{1}{\varsigma (\hbar \del_s)} [z^0] \Phi^{\mf{W}}(\mf{P})(z,s, \beta ) \Big|_{s=0} c
        \end{split}
    \end{align} 
    are well-defined, and 
    \begin{equation}
        \Res_{z=0} \mc{E}(\hbar \del_s , z) \Phi^{\mf{W}}(\mf{P}) (z,s, \beta ) \Big|_{s=0}
        =
        e^{- \beta \ad_W} \Res_{z=0} \mc{E}(\hbar \del_s , z) \mf{P} (z,s) \Big|_{s=0}.
    \end{equation}
\end{proposition}

In this notation, for $ m \neq 0$,
\begin{align} 
    \mf{J}_m (z)
    &= 
    \sum_{\mu \in \Z} \mf{J}_{m,\mu} z^{-\mu}
    =
    \sum_{\mu \in \Z} \delta_{m,\mu} z^{-\mu}
    =
    z^{-m}
    \\
    \mf{J} (z,s) 
    &=
    \sum_{\mu \in \Z \setminus \{ 0\} } \big( \frac{z}{X} \big)^m \frac{dX}{X} , 
\end{align}
and for these functions, \cref{MoreGeneralConjugation} reduces to \cref{CurrentConjugation}.

\begin{lemma}\label{MoyalFlow}
    The function  $ \Phi^{\mf{W}}(\mf{P})(z,s,\beta) $ from \cref{GeneralFunctionFlow} satisfies the Heisenberg equation
    \begin{equation}\label{HeisenbergEquation}
        \del_\beta \Phi^{\mf{W}} (\mf{P}) (z,s,\beta)
        =
        \Big[ \mf{W}(z, \frac{s}{\hbar}, \hbar), \Phi^{\mf{W}} (\mf{P}) (z,s,\beta) \Big]_\star,
    \end{equation}
    given by the commutator for the Moyal star product associated to the Poisson bivector $ \Pi = z \del_z \wedge \del_s$, with initial condition $ \Phi^{\mf{W}} (\mf{P}) (z,s,0) = \mf{P}(z,s)$.
\end{lemma}
\begin{remark}
    The relation between the KP hierarchy and the Moyal bracket is not new: the KP hierarchy was formulated in terms of the Moyal bracket in \cite{Str95}.
\end{remark}
\begin{proof}
    In \cref{PhiRecursion}, we interpret shifts of arguments as exponents of derivatives, and we see that the shifts of the $s$-variables are proportional to the exponents of the $z$-variables. We can write this as
    \begin{equation}
        \begin{split}
            \del_\beta \Phi^{\mf{W}} (\mf{P}) (z,s,\beta)
            &=
            \Big\lfloor_{\substack{\bar{s} \to s\\ \bar{z} \to z}}
            \mf{W} ( \bar{z} e^{\frac{\hbar}{2}\del_s}, \frac{\bar{s}}{\hbar}, \hbar ) \Phi^{\mf{W}} (\mf{P}) ( z e^{-\frac{\hbar}{2} \overleftarrow{\del_{\bar{s}}}} , s ,\beta) 
            \\
            &\hspace{2cm}-
            \mf{W} ( \bar{z} e^{-\frac{\hbar}{2}\del_s}, \frac{\bar{s}}{\hbar}, \hbar ) \Phi^{\mf{W}} (\mf{P}) ( z e^{\frac{\hbar}{2} \overleftarrow{\del_{\bar{s}}}} , s,\beta)
            \\
            &=
            \Big\lfloor_{\substack{\bar{s} \to s\\ \bar{z} \to z}}
            e^{-\frac{\hbar}{2} \del_{\bar{s}} z \del_z} \mf{W} ( \bar{z} e^{\frac{\hbar}{2}\del_s}, \frac{\bar{s}}{\hbar}, \hbar ) \Phi^{\mf{W}} (\mf{P}) (z,s,\beta) 
            \\
            &\hspace{2cm}-
            e^{\frac{\hbar}{2} \del_{\bar{s}} z \del_z } \mf{W} ( \bar{z} e^{-\frac{\hbar}{2}\del_s}, \frac{\bar{s}}{\hbar}, \hbar ) \Phi^{\mf{W}} (\mf{P}) (z,s,\beta)
            \\
            &=
            \Big\lfloor_{\substack{\bar{s} \to s\\ \bar{z} \to z}}
            e^{\frac{\hbar}{2} ( \bar{z} \del_{\bar{z}} \del_s -\del_{\bar{s}} z \del_z)} \mf{W} ( \bar{z} , \frac{\bar{s}}{\hbar}, \hbar ) \Phi^{\mf{W}} (\mf{P}) (z,s,\beta) 
            \\
            &\hspace{2cm}-
            e^{\frac{\hbar}{2} ( \del_{\bar{s}} z \del_z - \bar{z} \del_{\bar{z}} \del_s )} \mf{W} ( \bar{z} , \frac{\bar{s}}{\hbar}, \hbar ) \Phi^{\mf{W}} (\mf{P}) (z,s,\beta)
            \\
            &=
            \mf{W}(z, \frac{s}{\hbar}, \hbar ) \star \Phi^{\mf{W}} (\mf{P}) (z,s,\beta) - \Phi^{\mf{W}} (\mf{P}) (z,s,\beta) \star \mf{W} (z,\frac{s}{\hbar}, \hbar),
        \end{split}
    \end{equation}
    which is the commutators for the Moyal star product.
\end{proof}

\begin{remark}
    \Cref{MoreGeneralConjugation} does not give a full map $ \C \llbracket z, s, \hbar \rrbracket [z^{-1} ] \to \widehat{\mf{a}}_\infty$ intertwining the flow on both sides. The problem is that the central term is not well-defined in general. We can define it to be zero in case $ \mf{P}_0 = 0$, and we can determine it for $ \Phi^{\mf{W}} (\mf{P}) $ only by \emph{using} the $ \beta$-flow.
\end{remark}




\Cref{TopologicalAssumption} implies that
\begin{equation}\label{LeadingOrderAssumption}
	\mf{W}(z, \frac{s}{\hbar}, \hbar) = \frac{1}{\hbar} \mf{w}(z,s) + \mc{O}(\hbar^1),
\end{equation}
where $ \mf{w} (z, s) \in \C \llbracket z, s\rrbracket$ is independent of $ \hbar$.

\begin{proposition}\label{LeadingOrderFlowHamiltonian}
    Under \cref{TopologicalAssumption}, for any $ \mf{P} \in \C \llbracket z, s, \hbar \rrbracket [ z^{-1} ]$, $\mu \in \Z$ and $ n \in \N$, $ \Phi^{\mf{W}}_{n,\mu} (\mf{P})(s)$ is a power series in $ \hbar$.
	Writing $ \bar{\Phi}^{\mf{w}}_{n,\mu} (\mf{P})(s) = \Phi^{\mf{W}}_{n,\mu} (\mf{P})(s) \big|_{\hbar = 0}$ and $\bar{\Phi}^{\mf{w}}(\mf{P}) (z,s,\beta) = \Phi^{\mf{W}}(\mf{P}) (z,s,\beta) \big|_{\hbar = 0}$, this satisfies a partial differential equation
	\begin{equation}\label{PDE}
		\del_\beta \bar{\Phi}^{\mf{w}}(\mf{P}) (z,s,\beta)
		=
		(z\del_z \mf{w}(z,s)) \del_s \bar{\Phi}^{\mf{w}}(\mf{P}) (z,s,\beta) - (\del_s \mf{w}(z,s)) z \del_z \bar{\Phi}^{\mf{w}}(\mf{P}) (z,s,\beta)
	\end{equation}
	with initial condition $ \bar{\Phi}^{\mf{w}}(\mf{P}) (z,s,0) = \mf{P}(z,s)$ In particular, $ \bar{\Phi}^{\mf{w}} (\mf{J}_m) (z,s,0) = z^{-m}$ and $ \bar{\Phi}^{\mf{w}}(\mf{J}_m) (z,s,\beta ) = \bar{\Phi}_1 (z, s, \beta)^m $.\par
	This differential equation can be restated as a Hamiltonian flow
	\begin{equation}\label{HamFlow}
		\del_\beta \bar{\Phi}^{\mf{w}} (\mf{P})
        = 
        V_\mf{w} \bar{\Phi}^{\mf{w}} (\mf{P}) ,
        \quad
        \text{i.e.}
        \quad
        \bar{\Phi}^{\mf{w}} (\mf{P}) (z, s, \beta )
        =
        e^{\beta V_{\mf{w}}} \mf{P} (z,s),
	\end{equation}
	where $ V_\mf{w}$ is the Hamiltonian vector field associated to the Hamiltonian function $ \mf{w}$ by the standard symplectic form on $ \C^* \times \C$.
	\begin{equation}\label{SymplecticForm}
		\Omega 
		= 
		\frac{dz}{z} \wedge ds .
	\end{equation}
    Finally, the central term is a Laurent series in $ \hbar$ with leading order $ \hbar^{-1}$, whose coefficient is found by
    \begin{equation}\label{CentralTermAsFunctions}
        \del_\beta \tilde{\bar{\Phi}}^{\mf{w}}(\mf{P}) ( \beta) 
        =
        \Res_{z=0} ( \del_z \mf{w}(z,s)) \bar{\Phi}^{\mf{w}}(\mf{P}) (z, s, \beta ) dz \Big|_{s = 0};
        \qquad
        \tilde{\bar{\Phi}}^{\mf{w}}(\mf{P}) ( \beta = 0) = 0\,.
    \end{equation}
    This is solved by
    \begin{equation}\label{CentralTermSolution}
        \tilde{\bar{\Phi}}^{\mf{w}} (\mf{P}) (\beta)
        =
        \int_{b = 0}^\beta \Big( \Res_{z=0} ( \del_z \mf{w}(z,s)) \bar{\Phi}^{\mf{w}}(\mf{P}) (z, s, b ) dz \Big|_{s = 0} \Big) db.
    \end{equation}
\end{proposition}

\begin{proof}
    The Hamiltonian flow follows directly from the semi-classical limit of \cref{HeisenbergEquation}, but we give another explicit calculation below, as it highlights some useful features.\par
	By assumption, $ \Phi^{\mf{W}}_{0,\mu} (\mf{P}) = \mf{P}_\mu$ is a power series in $ \hbar$.
	The leading orders in $\hbar $ of the recursion \eqref{PhiRecursion} give, writing $ w_k ( \frac{s}{\hbar}, \hbar ) = \frac{1}{\hbar} \bar{w}_k (s) + \mc{O}(\hbar)$,
	\begin{equation}
		\begin{split}
			(n+1) \Phi^{\mf{W}}_{n+1,\mu} (\mf{P}) (s) 
			&=
			\sum_{k = 0}^\infty w_k (\frac{s}{\hbar} + \frac{\mu+k}{2}, \hbar ) \Phi^{\mf{W}}_{n, \mu+k} (\mf{P}) ( s + \frac{\hbar k}{2}) - w_k ( \frac{s}{\hbar} - \frac{\mu+k}{2}, \hbar ) \Phi^{\mf{W}}_{n,\mu+k} (\mf{P}) \big( s - \frac{\hbar k}{2}\big)
			\\
			&=
			\sum_{k = 0}^\infty \big( \frac{1}{\hbar} \bar{w}_k (s ) +  \frac{\mu+k}{2} \bar{w}'_k(s) + \mc{O}(\hbar) \big) \big( \Phi^{\mf{W}}_{n, \mu+k} (\mf{P}) ( s ) + \frac{\hbar k}{2} \Phi^{\mf{W}}_{n, \mu+k} (\mf{P})' ( s ) + \mc{O}(\hbar^2) \big) 
			\\
			&\qquad
			- \big( \frac{1}{\hbar} \bar{w}_k (s ) - \frac{\mu+k}{2} \bar{w}'_k(s) + \mc{O}(\hbar) \big) \big( \Phi^{\mf{W}}_{n, \mu+k} (\mf{P}) ( s ) - \frac{\hbar k}{2} \Phi^{\mf{w}}_{n, \mu+k} (\mf{P})' ( s ) + \mc{O}(\hbar^2) \big)
			\\
			&=
			\sum_{k = 0}^\infty  k \bar{w}_k(s) \Phi^{\mf{W}}_{n,\mu+k} (\mf{P})' (s) + (\mu+k) \bar{w}'_k(s) \Phi^{\mf{W}}_{n, \mu+k} (\mf{P}) ( s ) + \mc{O}(\hbar) ,
		\end{split}
	\end{equation}
	which shows that by recursion all $ \Phi^{\mf{W}}_{n,\mu} (\mf{P})(s)$ are power series in $ \hbar$.
	In terms of $ \bar{\Phi}^{\mf{w}}_{n,\mu} (\mf{P})(s)$,
	\begin{equation}
		(n+1) \bar{\Phi}^{\mf{w}}_{n+1,\mu} (\mf{P}) (s) 
		=
		\sum_{k = 0}^\infty k \bar{w}_k ( s ) \del_s \bar{\Phi}^{\mf{w}}_{n,\mu+k} (\mf{P}) \big( s \big) + (\mu + k) \del_s \bar{w}_k (s ) \bar{\Phi}^{\mf{w}}_{n, \mu+k} (\mf{P}) ( s ) ,
	\end{equation}
	which after taking generating series immediately yields the partial differential equation.\par
    For the central term, we find from \cref{CentralTerm} that
    \begin{equation}
        \begin{split}
            (n+1) \tilde{\Phi}^{\mf{W}}_{n+1} (\mf{P})
            &=
            \sum_{k=0}^\infty \sum_{ l = \frac{1}{2}}^{k - \frac{1}{2}} \big( \frac{1}{\hbar} \bar{w}_k (0 ) + \mc{O} (\hbar^0) \big) \big( \bar{\Phi}^{\mf{w}}_{n,k} (\mf{P}) ( 0) + \mc{O} (\hbar) \big)
            \\
            &=
            \frac{1}{\hbar} \sum_{k=0}^\infty k \bar{w}_k (0) \bar{\Phi}^{\mf{w}}_{n,k} (\mf{P}) (0) + \mc{O} (\hbar^0).\qedhere
        \end{split}
    \end{equation}
\end{proof}

\begin{definition}\label{DefSymplectomorphism}
	We denote the symplectic flow generated by this Hamiltonian vector field
	\begin{equation}
		\Xi_\beta (z,s)
		= 
		e^{\beta V_\mf{w}} z 
		=
		\bar{\Phi}^{\mf{w}}(\mf{J}_{-1}) (z,s,\beta),
		\qquad
		\Upsilon_\beta (z,s)
		=
		e^{\beta V_\mf{w}} s .
	\end{equation}
\end{definition}

Because they are Hamiltonian flows, $ (\Xi_\beta ,\Upsilon_\beta )$ is a symplectomorphism onto its image for $ \Omega$, i.e. 
\begin{equation}
	\Omega 
    =
    \frac{dz}{z} \wedge ds
    =
    \frac{d\Xi_\beta}{\Xi_\beta} \wedge d\Upsilon_\beta .
\end{equation}

This symplectic flow is well-defined for formal $ \beta$, i.e. we can consider $ \Xi_\beta (z,s) , \Upsilon_\beta (z,s) \in \C \llbracket z, s, \beta \rrbracket$. In the general setup with $ \mf{w} \in \C \llbracket z, s\rrbracket $, we have no guarantees that the $ \beta = 1$ flow, necessary for the next proposition, exists. However, we will see that in a fairly wide class of cases, we will be able to write the $ \beta = 1$ flow explicitly as analytic functions.

We will consider both the cut-and-join operator $ W$ and the `right ramification operator' $ \check{S} = \sum_{q=1}^\infty s_q \frac{J_q}{q\hbar} $ as inducing flows. We will require that $ \mf{w} (z, s= 0) = 0$, to ensure that the central term from the first flow vanishes. If this term can be computed otherwise, it can be added to $ y(z)$ by the same argument.

\begin{proposition}\label{GeneralSC}
	If $ \mf{w}(z,s) \in s\C \llbracket z, s\rrbracket $, the spectral curve is given by 
	\begin{equation}
		\begin{cases}
			X(z)
			&=
			\Xi_1 (z, S(z))
			\\
			y (z)
			&=
			\Upsilon_1 (z, S(z)) - \Upsilon_1 (z, 0)
		\end{cases} 
	\end{equation}
	if this is well-defined.
\end{proposition}

\begin{proof}

    We use a double flow in terms of functions, and take the vacuum expectation value, i.e. the central term, of $ \mc{E}$.
    For the $ \mf{w}$-flow, note that $ \del_z \mf{w} (z,s) \Big|_{s=0} = 0$, so that using \cref{CentralTermAsFunctions},
    \begin{equation}
        \del_\beta \tilde{\bar{\Phi}}^{\mf{w}}(\mf{J}) ( \beta) 
        =
        \Res_{z=0} ( \del_z \mf{w}(z,s)) \bar{\Phi}^{\mf{w}}(\mf{J}) (z, s, \beta ) dz \Big|_{s = 0}
        =
        0,
    \end{equation}
    and hence the central term for this flow vanishes.\par
    For the flow with $ \mf{s} (z,s) \coloneqq \hbar \check{\mf{S}}(z,\frac{s}{\hbar} ) \Big|_{\hbar = 0} = \int_0^z S(z') \frac{dz'}{z'}$, on the other hand, we have $ z \del_z \mf{s} (z,s) = S(z)$. This latter is independent of $s$, so that in particular its flow is given by
    \begin{equation}
        \bar{\Phi}^{\mf{s}} (\mf{P} ) (z, s, \beta)
        =
        e^{\beta V_\mf{s}} \mf{P}(z,s)
        =
        \mf{P} (z, s + \beta S(z) )
    \end{equation}
    so that \cref{CentralTermSolution} gives for the central term:
    \begin{equation}
        \begin{split}
            \tilde{\bar{\Phi}}^{\mf{s}} (\mf{P}) (\beta)
            &=
            \int_{b = 0}^\beta \Big( \Res_{z=0} ( \del_z \mf{s}(z,s)) \bar{\Phi}^{\mf{s}}(\mf{P}) (z, s, b ) dz ) \Big|_{s = 0} \Big) db
            \\
            &=
            \int_{b = 0}^\beta \Res_{z=0} S(z) \mf{P} (z, b S(z) ) \frac{dz}{z} \wedge db
            \\
            &=
            \Res_{z=0} \int_{s = 0}^{\beta S(z)}  \mf{P} (z, s ) ds \wedge \frac{dz}{z}.
        \end{split}
    \end{equation}

    This yields:
    \begin{equation}
        \begin{split}
            \omega_{0,1}(X)
            &=
            [\hbar^{-1}] \Big\< J(X) e^{W} e^{\check{S}} \Big\>
            \\
            &=
            [\hbar^{-1}] \Res_{z = 0} \Big\< e^{-\ad_{\check{S}}} e^{-\ad_W} \mc{E}( \hbar \del_s, z) \mf{J}(z,s) \Big|_{s = 0} \Big\> 
            \\
            &=
            [\hbar^{-1}] \Res_{z = 0} \Big\< \mc{E}( \hbar \del_s, z) \Big\> \Phi^{\check{\mf{S}}} (\Phi^{\mf{W}} (\mf{J}))(z,s, 1, 1) \Big|_{s = 0} 
            \\
            &=
            [\hbar^{-1}] \Res_{z = 0} \frac{dz}{z \varsigma (\hbar \del_s)} \Phi^{\check{\mf{S}}} (\Phi^{\mf{W}} (\mf{J}))(z,s, 1, 1) \Big|_{s = 0}
            \\
            &=
            \tilde{\bar{\Phi}}^{\mf{s}}(e^{V_{\mf{w}}} \mf{J} )(z,s, 1)
            \\
            &=
            \Big( \int_{s=0}^{S(z)} \Res_{z=0} \sum_{\mu \in \Z \setminus \{ 0\} } e^{V_\mf{w}} \big( \frac{z}{X} \big)^m ds \wedge \frac{dz}{z} \Big) \frac{dX}{X}
            \\
            &=
            \Big( \Res_{z = 0}  \int_{s = 0}^{S(z)} \sum_{\mu \in \Z \setminus \{ 0\} } \big( \frac{\Xi_1(z,s )}{X} \big)^m ds \wedge\frac{dz}{z  } \Big) \frac{dX}{X}.
        \end{split}
    \end{equation}
    Now we use that $ ( \Xi_1, \Upsilon_1 ) $ is a symplectomorphism, and write $ z(\Xi_1) $ for the local inverse to $ \Xi_1 (z)$ near $ z = \Xi_1 = 0$ to find
    \begin{equation}
        \begin{split}
            \omega_{0,1}(X)
            &=
            \Big( \Res_{\Xi_1 = 0}  \int_{\Upsilon_1 = \Upsilon_1 (z(\Xi_1),0)}^{\Upsilon_1(z(\Xi_1),S(z(\Xi_1)))} \sum_{\mu \in \Z \setminus \{ 0\} } \big( \frac{\Xi_1}{X} \big)^m d\Upsilon_1 \wedge\frac{d\Xi_1}{\Xi_1  } \Big) \frac{dX}{X}
            \\
            &=
            \sum_{\mu \in \Z \setminus \{ 0\} } \Big( \Res_{\Xi_1 = 0} \big( \Upsilon_1(z(\Xi_1),S(z(\Xi_1))) - \Upsilon_1 (z(\Xi_1),0) \big) \big( \frac{\Xi_1}{X} \big)^m \frac{d\Xi_1}{\Xi_1  } \Big) \frac{dX}{X}
            \\
            &=
            \Big( \Upsilon_1(z(X),S(z(X))) - \Upsilon_1 (z(X),0) \Big) \frac{dX}{X},
        \end{split}
    \end{equation}
    using that the $ \mu = 0$ term vanishes, as $ z(0) = S(0) = 0$. Equivalently, using the parametrisation given in the proposition,
    \begin{equation}
        \omega_{0,1} (X (z)) = y(z) \frac{dX(z)}{X(z)}.\qedhere
    \end{equation}
\end{proof}

\begin{corollary}
	Because of Hamiltonian flow,
	\begin{equation}
		\mf{w}( X, y) 
		= 
		\mf{w} (z, S(z)) .
	\end{equation}
\end{corollary}

In particular cases, this equation can be a first step for obtaining a plane curve representation of the spectral curve. This can be useful for finding e.g. the quantum curve of the problem. We do not pursue this direction here.

\section{Explicit closed formulas }
\label{ExplicitSection}

The goal of this section is to find explicit closed formulas for the multidifferential generating series of leaky Hurwitz numbers. 

\subsection{Explicit closed algebraic formulas for multidifferentials}
\label{ExplicitMultidifferentials}

This section is heavily inspired by \cite{BDKS22}, and we will mostly use the same formalism. In fact, most of the results of this section are direct extensions of those in \cite{BDKS22}, and the proofs given there go through verbatim. As such, we do not give all the details, but rather only note the difference.

We start with \cite[Sections 3 \& 4.1-2]{BDKS22}, which work verbatim (except that we write $ s $ and $t$ for their $ y$ and $ r$, respectively),  with the understanding that $ \phi$ and $ \bar{\phi}$, and hence also $ U $ and $U^+$, depend explicitly on $z$. This does not change the proofs in any way.




For completeness, let us explicitly define $ \phi$ and $ \bar{\phi}$ in terms of our flow formalism.

\begin{definition}\label{Defsphi}
    Define 
    \begin{equation}
        \phi_{m,n,\mu}(s)
        \coloneqq
        \Phi^{\mf{W}}_{n,\mu+m} (\mf{J}_{m})(s) ,
        \quad
        \phi_m (z,s)
        \coloneqq
        \sum_{n=0}^\infty \sum_{\mu \in \Z} \phi_{m,n,\mu} z^{-\mu} ,
        \quad
        \bar{\phi}_m(z,s)
        \coloneqq
        \phi_m (z,s)\Big|_{\hbar = 0}.
    \end{equation}
    Then by \cref{PhiRecursion}, $ \phi_{m,0,\mu} = \delta_{\mu, 0} $ and
    \begin{equation}\label{Generalphi}
		(n+1) \phi_{m,n+1,\mu} (s) 
		=
		\sum_{k = 0}^\infty w_k (\frac{s}{\hbar} + \frac{m+\mu+k}{2}, \hbar ) \phi_{m,n, \mu+k} ( s + \frac{\hbar k}{2}) - w_k ( \frac{s}{\hbar} - \frac{m+\mu+k}{2}, \hbar ) \phi_{m,n,\mu+k} \big( s - \frac{\hbar k}{2}\big) .
	\end{equation}
    Moreover, from \cref{HamFlow,DefSymplectomorphism}, we find directly that
    \begin{equation}\label{phiAsRatio}
        \bar{\phi}_{m} (z,s) = (\frac{z}{\Xi_1 (z,s)})^{m}.
    \end{equation}
\end{definition}

\begin{lemma}\label{MoyalFlowNice}
    If, for some natural number $k$, the flows  $ (\Xi_\beta(z,s))^k$ of $z^k$ and $ \Upsilon_\beta (z,s)$ of $ s $ in \cref{DefSymplectomorphism} are complete, and $ \mf{W}(z,\frac{s}{\hbar}, \hbar) $ is a Laurent series in $\hbar $ with leading order $ \hbar^{-1}$ whose coefficients are analytic in $ z^k$ and $ s$, then the functions $ \phi_{lk} (z,s)$ are well-defined power series in $ \hbar$, whose coefficients are analytic functions in $ z$ and $ s$.
\end{lemma}

\begin{remark}
    The assumptions in this lemma may look a bit odd, but we will see in \cref{ChangeOfCoords} that they hold in case of fixed leakiness $ k$.
\end{remark}

\begin{proof}
    By assumption, the problem only depends on $ \zeta = \frac{z^k}{k}$, and since $ \frac{d\zeta}{\zeta} = \frac{dz}{z}$, we can work in this coordinate.\par
    We use \cref{MoyalFlow}: up to a factor $ \zeta^{-l}$, $ \phi_{lk} (\zeta,s)$ is the time-one flow of $ \zeta^l$ under the Moyal adjoint operator $ [\mf{W} (\zeta, \frac{s}{\hbar}, \hbar) , \cdot ]_\star$.\par
    Expanding $ \mf{W} (\zeta, \frac{s}{\hbar}, \hbar) = \hbar^{-1} \mf{w} + \sum_{j=0}^\infty \mf{W}_j \hbar^j$ and $ [ \cdot, \cdot ]_\star = \hbar \{ \cdot, \cdot \} + \sum_{j = 3}^\infty \hbar^{j} [\cdot, \cdot ]_j$, we want to solve the system $ \del_\beta f = [ \mf{W}, f]_\star$ for $f = \sum_{j=0}^\infty f_j \hbar^j$ with $ f(\beta = 0) = \zeta^l$ order by order in $ \hbar$.\par
    The leading order is the Hamiltonian flow $ \del_\beta f_0 = \{ \mf{w}, f_0 \}$, and this is complete by assumption. The next order is
    \begin{equation*}
        (\del_\beta - X_{\mf{w}}) f_1
        =
        \{\mf{W}_0, f_0\},
    \end{equation*}
    which can be solved with the method of characteristics as
    \begin{equation*}
        f_1 (z,s, \beta ) = \int_0^\beta \{ \mf{W}_0, f_0 \} (b, \Xi_{\beta -b} (z,s), \Upsilon_{\beta - b} (z,s) )db.
    \end{equation*}
    In general, we always get $ (\del_\beta - X_{\mf{w}}) f_j$ in terms of previously calculated terms, and $ \mf{W}_i$ with $ i \leq j$ and hence we should integrate the $ \beta$ flow along characteristic lines. Since we assume the flow is complete, the integrand is analytic, and hence so is the integral.
\end{proof}

To adapt \cite[Sections 4.3-4]{BDKS22}, we use \cite[Lemma 4.6]{BDKS22} as definition.

\begin{definition}
	We define the functions $ L_t (v,z, s, \hbar)$ by
	\begin{equation}\label{FunctionL}
		L_t (v,z,s,\hbar) \coloneqq \bar{\phi}_v(z,s)^{-1} \del_s^t \phi_v(z,s),
	\end{equation}
	which is well-defined, since following \cref{Generalphi}, $ \phi_m$, and hence $ \bar{\phi}_m$, is well-defined for $ m \in \C$, not just $m \in \N$.
\end{definition}

\begin{lemma}\label{LtExpressions}
	We have that
	\begin{align}
		L_0(v,z,s,\hbar) 
		&= 
		1 + \mc{O} (\hbar) \,,
		\\
		L_t(v,z,s,\hbar)
		&=
		\bar{\phi}_v(z,s)^{-1} \del_s^t \bar{\phi}_v (z,s) L_0(v,z,s,\hbar)
		=
		\Big( \del_s + v \frac{\bar{\phi}_1' (z, s)}{\bar{\phi}_1 (z, s)} \Big)^t L_0 (v,z,s, \hbar) .
	\end{align}
	Moreover, $ L_0$, and hence all $ L_t$, are power series in $ \hbar$ with coefficients that are polynomial in $v$ and the derivatives of $ \log \bar{\phi}_1 (z, s)$ and the $\hbar$-coefficients of $\frac{\phi_1(z,s)}{\bar{\phi}_1(z,s)}$. 
\end{lemma}

\begin{proof}
	The first identity holds from the definition of $ \bar{\phi}$ as the leading $ \hbar$ order term of $ \phi$. The second identity follows directly from the definition, while the last follows from $ \bar{\phi}_v = (\bar{\phi}_1)^{v}$.\par
	All $ L_t$ are power series in $ \hbar $ because the $ \Phi_{n,\mu}^{\mf{W}}(\mf{J}_m)(s)$ are by \cref{LeadingOrderFlowHamiltonian}. The polynomiality statement follows from the recursive formula.
    
    We can also write $ L_0(-m,z,s,\hbar ) = \bar{\Phi}^{\mf{W}}(z^{m})(z,s,1)^{-1} \Phi^{\mf{W}}(z^{m})(z,s,1)$. Because of the structure of the $\star$-product, with all higher-order terms containing $ s$-derivatives, $ z^m = z^{\star m}$. Since the $ \Phi$-flow preserves the $\star$-product, and the $ \bar{\Phi}$-flow preserves the undeformed product,
    \begin{equation}
        L_0(-m,z,s,\hbar)
        =
        \frac{\big( \Phi^{\mf{W}}(z)(z,s,1) \big)^{\star m} }{ \big( \bar{\Phi}^{\mf{W}}(z)(z,s,1) \big)^{m}} .
    \end{equation}
    Now, consider this explicitly as an $ \hbar$-series. To obtain a power $ \hbar^n$, we can take positive powers of $\hbar $ from some $\Phi^{\mf{W}}(z)(z,s,1)$, or take a higher order of the product (adding higher $\hbar$-coefficients of $ \frac{\phi_1 (z,s)}{\bar{\phi}_1(z,s)} $ or adding derivatives, respectively). For a fixed $n$, we therefore choose how to distribute these powers over these options; each of the finitely many options of decomposing $n$ has a coefficient which is composed of binomial coefficients like $ \binom{m}{n_i}$ choosing which factors $j$ of the product contribute non-trivially (i.e. do not just contribute $ \bar{\Phi}^{\mf{W}}(z)$ with the undeformed product). The remaining $ m-j$ factors are divided by $\bar{\Phi}^{\mf{W}}(z)^m$, so the ratio is independent of $m$. 
    In particular, because the $\hbar^n $ coefficient of the $\star$-product contains $ 2n$ derivatives, the degree of the polynomial in $m$ is bounded by $ 2n$.
\end{proof}

\begin{corollary}\label{LtAnalytic}
    Under the assumptions of \cref{MoyalFlowNice}, the $L_t$ are power series in $ \hbar$ whose coefficients are polynomials in $v$ and analytic functions in $ z$ and $s$.
\end{corollary}
\begin{proof}
    Comparing with \cref{LtExpressions}, the only new claim is the analyticity. This holds for $ \phi_m$, i.e. $ \phi_v$ with $ v = m$ an integer. We can use polynomiality of $ L_t$ in $ v$ to express it solely in terms of $ \phi_m$.
\end{proof}

With this, we get a minor extension of \cite[Proposition 4.3]{BDKS22}, which they call their \emph{principal identity}. We write it in terms of a one-form in $z$, rather than a function, cf. \cite[Proposition 4.4]{ABDKS24}.

\begin{proposition}[{\cite[Proposition 4.3]{BDKS22}}]
	Let $H(u, z) \in \C [u]\llbracket \hbar, z \rrbracket [z^{-1}] dz $. Then
	\begin{equation}
		\sum_{m \in \Z} X^m\frac{dX}{X} \sum_{t = 0}^\infty \Res_{\tilde{z} = 0} \del_s^t \phi_m (\tilde{z},s) \big|_{s = 0} [u^t] e^{u S(\tilde{z})} \frac{H(u,\tilde{z})}{\tilde{z}^m}
		=
		\sum_{j,t = 0}^\infty \big( d \frac{1}{dx} \big)^j [v^j] L_t(v, z, S(z),\hbar) [u^t] H(u,z)
	\end{equation}
	where
	\begin{equation}
		X(z)
        =
        e^{x(z)}
		\coloneqq
		z  \big( \bar{\phi}_1(z,S(z)) \big)^{-1} .
	\end{equation}
\end{proposition}

Apart from minor change in notation, the only difference is that we allow $ \phi_m$ to depend on $ \tilde{z}$. The proof again goes through verbatim. Similarly, we get the following direct extensions of \cite[Corollary 4.4, Theorem 4.8, Corollary 4.10]{BDKS22}.


\begin{proposition}
	For $ n \geq 2$,
	\begin{equation}\label{HgnAsGraphSum}
		\omega_{g,n} + \delta_{n,2} \delta_{g,0} \frac{dX_1 dX_2}{(X_1 - X_2)^2}
		=
		[\hbar^{2g-2+n}] U_n \dotsb U_1 \sum_{\gamma \in \Gamma_n} \prod_{\{ v_j, v_l\} \in E_\gamma} w_{j,l} ,
	\end{equation}
	where
	\begin{align}
		(Uf) (X(z))
		&=
		\sum_{j,t = 0}^\infty \big( d \frac{1}{dx}\big)^j \Big( [v^j] L_t(v, z, S(z),\hbar) [u^t]\frac{e^{u (\mc{S}(u \hbar z \del_z) -1) S(z)}}{u \hbar \mc{S}(u\hbar)} f(u,z) \frac{dz}{z} \Big)
		\\
		w_{j,l} 
		&=
		e^{\hbar^2 u_j u_l \mc{S}(u_j \hbar Q_j D_j) \mc{S}(u_l \hbar Q_l D_l) z_j z_l/ (z_j - z_l)^2} -1 .
	\end{align}
	The sum is over all connected simple graphs on $n$ labelled vertices.\par
	For fixed $g$ and $n$, after taking the coefficient $[\hbar^{2g-2+n}]$, all sums in this formula become finite, and it becomes a rational expression in $z_1, \dotsc, z_n$, the derivatives of the functions $S_i = S(z_i)$, and the derivatives of the functions $ \log \bar{\phi}_1 (s_i, z_i)$, and the $\hbar$-coefficients of $\frac{\phi_1(z,s)}{\bar{\phi}_1(z,s)}$.
\end{proposition}

\begin{corollary}
	All diagonal poles (i.e., poles at $z_i = z_j$ for $i \neq j$) on the right side of \cref{HgnAsGraphSum} cancel out; in other words, the expression in the right side of \cref{HgnAsGraphSum} actually belongs to subring $ \C \llbracket z_1, \dotsc, z_n \rrbracket$.
\end{corollary}

For $ n=1$, we need to be a bit more careful in our calculation, due to a pole in $u$.

We will use a similar approach to \cref{GeneralSC} and \cite[Section 6.2]{BDKS22}.

\begin{proposition}
	In terms of $ x = \log X$,
	\begin{equation}\label{FormulaForn=1}
		\omega_1 (x)
		=
	    U 1 + \frac{1}{\hbar}\sum_{j=0}^\infty (d \frac{1}{dx})^j [v^j] \int_{\Upsilon=0}^{y(z(x))} L_0(v, z(x,\Upsilon), s(x,\Upsilon),\hbar) d\Upsilon \wedge dx .
\end{equation}
\end{proposition}

\begin{proof}
    As in the proof of \cref{GeneralSC}, we consider a double adjunction with $ W$ and $ \check{S}$. By \cref{Defsphi}, 
    \begin{equation}
        \Phi^{\mf{W}} (\mf{J} (X) ) (z,s, \beta = 1)
        =
        \sum_{m \in \Z \setminus \{ 0\} } \phi_m (z,s) \Big( \frac{X}{z} \Big)^m \frac{dX}{X}.
    \end{equation}
    On the other hand, because the generating series $ \hbar \check{\mf{S}} $ does not depend on $\hbar$ and $s$, we find, using the appropriate version of \cref{PhiRecursion}, that
    \begin{equation}
        (n+1) \Phi^{\check{\mf{S}}}_{n+1,\mu} (\mf{P}) (s) 
		=
		\sum_{q = 0}^\infty \frac{s_q}{\hbar} \Big( \Phi^{\check{\mf{S}}}_{n, \mu+q}(\mf{P}) ( s + \frac{\hbar q}{2}) - \Phi^{\check{\mf{S}}}_{n,\mu+q} (\mf{P}) \big( s - \frac{\hbar q}{2}\big) \Big),
	\end{equation}
    so that
    \begin{equation}
        \begin{split}
            \del_\beta \Phi^{\check{\mf{S}}} ( \mf{P}) (z,s,\beta) 
            &=
            \frac{1}{\hbar} \Big( \check{\mf{S}} (z e^{\frac{\hbar}{2} \del_s}) - \check{\mf{S}} (z e^{-\frac{\hbar}{2} \del_s}) \Big) \Phi^{\check{\mf{S}}}(\mf{P}) ( z, s, \beta)
            \\
            &=
            \sum_{t=0}^\infty \del_s^t \Phi^{\check{\mf{S}}} (\mf{P}) ( z, s, \beta) [u^t] \frac{\varsigma ( u \hbar z \del_z)}{\hbar} \check{\mf{S}}(z) 
            \\
            &=
            \sum_{t=0}^\infty \del_s^t \Phi^{\check{\mf{S}}} (\mf{P}) ( z, s, \beta) [u^t] u \mc{S} ( u \hbar z \del_z) S(z) ,
        \end{split}
    \end{equation}
    i.e.
    \begin{equation}
        \Phi^{\check{\mf{S}}} ( \mf{P}) (z,s,\beta) 
        =
        \sum_{t=0}^\infty \del_s^t \mf{P} ( z, s) [u^t] e^{u \beta \mc{S} ( u \hbar z \del_z) S(z)}
    \end{equation}
	Therefore, we find
	\begin{equation}
		\begin{split}
			\omega^\beta_{1} (X) 
            &\coloneqq
            \Big\< e^{- \beta \ad_{\check{S}}} e^{- \ad_W} J(X) \Big\>
            \\
			&=
			\sum_{m \in \Z} \sum_{t=-1}^\infty \Res_{z=0} \frac{dz}{z^{m+1}} \del_s^t \phi_m(z,s )\Big|_{s = 0} [u^t ] \frac{e^{u \beta \mc{S}(u \hbar z \del_z)S(z)}}{ \varsigma (u \hbar)} X^m \frac{dX}{X}
			\\
			&=
			\sum_{m \in \Z} \sum_{t=-1}^\infty \Res_{z=0} \frac{dz}{z^{m+1}} \del_s^t \phi_m(z,s )\Big|_{s = 0} [u^t ] \Big( \frac{e^{u \beta \mc{S}(u \hbar z \del_z)S(z)}}{u\hbar \mc{S}(u \hbar)} - \frac{e^{u \beta S(z)}}{u\hbar} \Big) X^m \frac{dX}{X}
			\\
			& \qquad
			+ \sum_{m \in \Z} \sum_{t=-1}^\infty \Res_{z=0} \frac{dz}{z^{m+1}} \del_s^t \phi_m(z,s )\Big|_{s = 0} [u^t ] \frac{e^{u \beta S(z)}}{u\hbar}  X^m \frac{dX}{X} .
		\end{split}
	\end{equation}

	The penultimate line is regular in $u$, so we can set $ \beta = 1$ and use the principal identity to find

	\begin{equation}
		\begin{split}
			&\sum_{m \in \Z} \sum_{t=0}^\infty \Res_{z=0} \frac{dz}{z^{m+1}} \del_s^t \phi_m(z,s )\Big|_{s = 0} [u^t ] \Big( \frac{e^{u \mc{S}(u \hbar z \del_z)S(z)}}{u\hbar \mc{S}(u \hbar)} - \frac{e^{u S(z)}}{u\hbar} \Big) X^m \frac{dX}{X}
			\\
			&\qquad =
			\frac{dX}{X} \sum_{m \in \Z} X^m \sum_{t=0}^\infty \Res_{z=0} \frac{dz}{z^{m+1}} \del_s^t \phi_m(z,s )\Big|_{s = 0} [u^t ] e^{uS(z)}\Big( \frac{e^{u (\mc{S}(u \hbar z \del_z)-1)S(z)}}{u\hbar \mc{S}(u \hbar)} - \frac{1}{u\hbar} \Big)
			\\
			&\qquad =
			\sum_{j,t=0}^\infty \big( d \frac{1}{dx} \big)^j \Big( [v^j] L_t (v, z, S(z),\hbar) [u^t] \Big( \frac{e^{u (\mc{S}(u \hbar z \del_z)-1)S(z)}}{u\hbar \mc{S}(u \hbar)} - \frac{1}{u\hbar} \Big) \frac{dz}{z} \Big) 
			\\
			&\qquad =
			U1 .
		\end{split}
	\end{equation}

	For the other summand, we need to actually interpret the $ \del_s^{-1}$ in the $ t = -1$ term. As explained in \cref{WhatIsCentralTerm}, we can make sense of this through the recursion on $n $ in \cref{CentralTerm}, or after taking generating series, through the $\beta $ flow. Then this is a well-defined initial value problem with initial value $0$ at $ \beta = 0$. So we actually want to calculate (again setting $ \beta = 1$)

	\begin{equation}
		\begin{split}
			\sum_{t=0}^\infty \Res_{z=0} \frac{dz}{z^{m+1}} \del_s^t \phi_m(z,s )\Big|_{s = 0} [u^t ] \frac{e^{u S(z)} - 1}{u}
			&=
			\Res_{z=0} \frac{1}{z^m} \frac{e^{S(z) \del_s} - 1}{\del_s} \phi_m (z,s) \Big|_{s=0} \frac{dz}{z}
			\\
			&=
			\Res_{z=0} \int_{s=0}^{S(z)} \frac{\phi_m ( z,s )}{z^m} ds \wedge \frac{dz}{z}
			\\
			&=
			\Res_{\Xi = 0} \int_{\Upsilon=0}^{y(z(\Xi))} \frac{L_0(m,z(\Xi,\Upsilon), s(\Xi,\Upsilon),\hbar)}{\Xi^m} d\Upsilon \wedge \frac{d\Xi}{\Xi},
		\end{split}
	\end{equation}
	where we used \cref{FunctionL,phiAsRatio} in the last equality, so that
	\begin{equation*}
			\omega_1 (x)
            =
            \omega_1^\beta (X) \big|_{\beta = 1}
			=
			U 1 + \frac{1}{\hbar}\sum_{j=0}^\infty (d \frac{1}{dx})^j [v^j] \int_{\Upsilon=0}^{y(z(x))} L_0(v, z(x,\Upsilon), s(x,\Upsilon),\hbar) d\Upsilon \wedge dx .
			\qedhere
	\end{equation*}
\end{proof}

Comparing \cref{HgnAsGraphSum,FormulaForn=1} with \cite[Proof of proposition 4.4]{ABDKS24}, using \cref{TrivialTR}, we find that the following explicit formulas for multidifferentials.

\begin{proposition}\label{LeakyMultidifferentials}
	The multidifferential generating series of leaky Hurwitz numbers with cut-and-join operator \eqref{CJ} are given by
	\begin{equation}
		\omega_n (z_{\llbracket n\rrbracket})
		=
		\Big( \prod_{i=1}^n T_i \Big)
		\mb{W}_n^{\mc{S}_0} (z_{\llbracket n\rrbracket}, u_{\llbracket n\rrbracket}) 
		+
		\frac{\delta_{n,1}}{\hbar} \sum_{j=0}^\infty ( d \frac{1}{dx} )^j [v^j] \int_{s = 0}^{S(z)} L_0 (v, z, s, \hbar) ds \wedge \frac{dz}{z},
	\end{equation}
	where for a formal series $ f(u,z) $ in $u$ with coefficients one-forms in $ z$,
	\begin{equation}
		Tf (z)
		=
		\sum_{j,t = 0}^\infty \big(d \frac{1}{dx}\big)^j  [v^j] L_t(v, z, S(z), \hbar) [u^t] f(u,z)
		=
		\sum_{j,t = 0}^\infty \big(- d \frac{1}{dx}\big)^j  [v^j] L_t(-v, z, S(z), \hbar) [u^t] f(u,z) .
	\end{equation}
\end{proposition}

\begin{corollary}\label{MultidifferentialsMeromorphic}
    Under the assumptions of \cref{MoyalFlowNice}, all $ \omega_{g,n}$ are meromorphic on the spectral curve given in \cref{GeneralSC}.
\end{corollary}



\subsection{Explicit spectral curves for fixed leakiness}
\label{CurvesForFixedk}

Now, we consider a fixed leakiness $k \geq 0$, and take general linear combinations of  completed cycles. I.e., we consider a cut-and-join operator

\begin{equation}\label{CJ}
	W = \sum_{l \in \Z'} w (l + \frac{k}{2}, \hbar) E_{l + k,l}.
\end{equation}

The case $ w(j, \hbar) = \frac{\hbar^{r-1} j^r}{r} $ is $r$-completed cycles, cf. \cref{LeakyCompletedOperator,CCexample}.

In this case, we will be able to write the spectral curve explicitly.





\begin{lemma}
	For fixed leakiness $ k$, $ \phi_{m,n,\mu}$ is non-zero only if $ \mu = - nk$. So define
	\begin{equation}\label{phi_m}
		\phi_{m,n}(s)
		\coloneqq
		\phi_{m,n,-nk}(s) .
	\end{equation}
	Then $ \phi_{m,n}(s)$ is given recursively by $ \phi_{m,0}(s) = 1$ and
	\begin{equation}\label{phi_mnRec}
		(n+1) \phi_{m,n+1} (s) 
		=
		w (\frac{s}{\hbar} + \frac{m-nk}{2}, \hbar ) \phi_{m,n} ( s + \frac{\hbar k}{2}) - w ( \frac{s}{\hbar} + \frac{-m+nk}{2}, \hbar ) \phi_{m,n} \big( s - \frac{\hbar k}{2}\big) .
	\end{equation}
	This recursion is solved by
	\begin{equation}\label{phi_mn}
		\phi_{m,n} (s )
		=
		\sum_{i=0}^n \frac{(-1)^{n-i}}{i!(n-i)!} \prod_{j=0}^{i-1} w \big( \frac{s}{\hbar} + \frac{m + (2j+1-n)k}{2}, \hbar \big) \prod_{j=i}^{n-1} w \big( \frac{s}{\hbar} + \frac{-m + (2j+1-n)k}{2}, \hbar \big) .
\end{equation}
\end{lemma}

\begin{proof}
	The recursion \eqref{phi_mnRec} is a direct specialisation of \cref{PhiRecursion}, using that with the choice of parameters there, $ \mu = m - (n+1)k$. It is straightforward to check that the given solution solves the recursion.
\end{proof}

\begin{proposition}\label{ChangeOfCoords}
	If $ w(j, \hbar )$ is of the shape $ w(\frac{s}{\hbar}, \hbar) = \frac{1}{\hbar} \bar{w}(s)+ \mc{O} (\hbar^0)$ for some $ \bar{w} (s) \in s \C\llbracket s \rrbracket$ with positive radius of convergence, then all $ \phi_{m,n}(s)$ are power series in $ \hbar$, and
	\begin{equation}
		\bar{\phi}_m ( z, s) 
        =
        \sum_{n=0}^\infty \phi_{m,n} (s) z^{nk} \Big|_{\hbar = 0}
        =
        \Big( \frac{\bar{w}(s + k \bar{w}(s)z^k)}{\bar{w}(s)}\Big)^{\frac{m}{k}} .
	\end{equation}
\end{proposition}

\begin{proof}
	Taking the partial differential equation \eqref{PDE} for $ \bar{\Phi}_m$, and using that in our case $ \mf{w} (z,s) = z^k \bar{w}(s)$, so $ \bar{\phi}_m$ is weighted homogeneous in $ \beta $ and $z$, and $ \bar{\phi}_m (z,s) = z^m \bar{\Phi}(z,s,1) $, we find the equation
	\begin{equation}\label{PDEforphibar}
		(\frac{1}{k} + z^k \bar{w}'(s)) z \del_z \bar{\phi}_m (z,s)
		=
		k z^k \bar{w}(s) \del_s \bar{\phi}_m (z,s) + m z^k \bar{w}'(s) \bar{\phi}_m (z,s) 
	\end{equation}
	with initial condition $ \bar{\phi}_m (0,s) = 1$. It is straightforward to check that the given solution does indeed solve this equation.
\end{proof}	

\begin{example}\label{FindingpsiFork=0}
	The $ k \to 0$ limit of \cref{ChangeOfCoords} is consistent with the non-leaky story:
	\begin{equation}
		\begin{split}
		    \lim_{k \to 0} \log \bar{\phi}_1 ( z, s)
            &=
            \lim_{k \to 0} \frac{\log (\bar{w} (s + k \bar{w}(s)z^k ))  - \log (\bar{w}(s))}{k}
            =
            \big( \log (\bar{w}(s))\big)' \bar{w}(s)
            \\
            &=
            \frac{\bar{w}'(s)}{\bar{w}(s)} \bar{w}(s)
            =
            \bar{w}'(s) ,
		\end{split}
	\end{equation}
	and $ \bar{w}'(s)$ equals the $ \psi(s)$ of \cite{BDKS22}.\par
	We can also see this directly from the Hamiltonian flow perspective: $ \mf{w}(z,s) = \bar{w}(s)$ is independent of $ z$ in this case, so $ z \del_z \mf{w}(z,s) = 0$, and
	\begin{equation}
		V_\mf{w}
		=
		- \bar{w}'(s) z \del_z
		=
		- \psi(s) z \del_z .
	\end{equation}
	By \cref{GeneralSC}, this yields the spectral curves
	\begin{equation}
		\begin{cases}
			X(z) &= \Big\lfloor_{s \to S(z)} e^{- \psi(s) z \del_z } z  = z e^{- \psi(S(z))}
			\\
			y(z) &= \Big\lfloor_{s \to S(z)} e^{- \psi(s) z \del_z } s = S(z)
		\end{cases} .
	\end{equation}
\end{example}

\begin{proposition}\label{Symplectomorphism}
	Let $ \mf{w}(z,s) = z^k \bar{w}(s)$, with $ \bar{w} $ an analytic function on $ \C$, such that $ \bar{w} (0) = 0$. The flows of $ z$ and $ s$ along $ V_\mf{w}$ are given by
	\begin{align}
		\Xi_\beta (z,s)
		&= e^{\beta V_\mf{w}} z 
		=
		z \Big( \frac{\bar{w}(s)}{\bar{w}(s+ k z^k \beta \bar{w}(s))}\Big)^{\frac{1}{k}}
		\\
		\Upsilon_\beta (z,s)
		&=
		e^{\beta V_\mf{w}} s
		=
		s + k z^k \beta \bar{w} (s) .
	\end{align}
	On a small enough punctured neighbourhood $U$ of $ (0,0) $ in $ \C^* \times \C$ (punctured in the $ z$-variable), this is well-defined, because the argument of the $k$-th root is close to $ 1$; we take the branch specified by $ \Xi (z,0) = z$.\par
    Moreover, $ (\Xi_\beta(z,s))^k$ is globally well-defined, so \cref{MoyalFlowNice,LtAnalytic,MultidifferentialsMeromorphic} hold in this situation.
\end{proposition}

\begin{proof}
	Straightforward calculation, given \cref{ChangeOfCoords}.
\end{proof}

\begin{corollary}\label{SpectralCurveShape}
	Under the assumptions of \cref{Symplectomorphism}, the spectral curve is given by
	\begin{equation} \label{SpectralCurveParam}
		\begin{cases}
			e^{x(z)}
            =
            X(z)
			&=
			\Xi_1 (z, S(z))
			=
			z \Big( \frac{\bar{w}(S(z))}{\bar{w}(S(z)+ k \bar{w}(S(z))z^k)}\Big)^{\frac{1}{k}}
			\\
			y (z)
			&=
			\Upsilon_1 (z, S(z))
			=
			S(z) + k \bar{w}(S(z)) z^k
		\end{cases} .
	\end{equation}
\end{corollary}

In general $ X(z)$ is not a univalued function on all of $ \C$, but $ \frac{dX}{X}$ is meromorphic, and this is sufficient for the general topological recursion we work with. In particular, we see that \cref{MoyalFlowNice} applies in this situation.






\begin{remark}
	We are writing $ \omega_{0,1} = y \frac{dX}{X}$ here, so actually the spectral curve is $ (x = \log X, y)$. Or, more properly in the degenerate generalised topological recursion formalism, $ ( \P^1, \frac{dX}{X}, dy, B, \mc{P} = \{ p \, | \, d\frac{X}{X}(p) = 0 \} )$.
\end{remark}


\begin{example}\label{rkSC}
	In the $r$-completed cycles case, $ w(j) = \hbar^{r-1} \frac{j^r}{r}$,
	\begin{equation}
		\bar{\phi}_m (z,s) 
		=
		\Big( \frac{r(s + k \frac{s^r}{r} z^k )^r}{r s^r}\Big)^{\frac{m}{k}}
		=
		\big( 1 + \frac{k}{r} s^{r-1} z^k \big)^{\frac{mr}{k}},
	\end{equation}
    so the spectral curve is
    \begin{equation}
        \begin{cases}
            X(z)
            &=
            z (1 + \frac{k}{r} S(z)^{r-1} z^k )^{-\frac{r}{k}}
            \\
            y(z)
            &=
            S(z) + k S(z)^r z^k
        \end{cases}.
    \end{equation}
    For $ S(z) = z^q$, this recovers \cref{prop-01spectral} in the parametrisation $ y = XY $, through \cref{equ-xviaz,equ-yviaz}.
\end{example}

\section{Topological recursion for fixed leakiness} \label{sec:TR}

We consider a setup with fixed $ k > 0$. We then have the spectral curve \eqref{SpectralCurveParam}, and hence we get multidifferentials defined by (logarithmic) topological recursion. In this section, we will identify exactly for a very wide class what leaky Hurwitz these multidifferentials generate.

Our main result is the following theorem, whose proof will take up most of this section.

\begin{theorem}\label{LeakyTRThm}
	Consider the leaky Hurwitz numbers with cut-and-join operator \eqref{CJ} such that $ w(\frac{s}{\hbar},\hbar) = \frac{1}{\hbar} \bar{w}(s) $, where $ \bar{w}(s)$ is a rational function of $s$ with no repeated factors, and with rational $S(z)$. The multidifferentials of \cref{Correlators} are calculated by topological recursion on the spectral curve \eqref{SpectralCurveParam}.
\end{theorem}

\begin{remark}
    In principle, we will prove that the differentials satisfy logarithmic topological recursion. But since $dy$ is given as the differential of a well-defined function in \eqref{SpectralCurveParam}, this family of spectral curves has no logarithmic singularities, and logarithmic topological recursion on it coincides with `usual' topological recursion.
\end{remark}

\Cref{LeakyTRThm} does not cover some natural cases, such as the $ r$-completed cycles $ k$-leaky Hurwitz numbers \cref{LeakyTRThm} for any $ r$ > 1: in this case, $ \bar{w} (s) = \frac{s^r}{r}$ clearly has a repeated factor. However, we can extend our result to include more cases, including these, by limit arguments.

\begin{corollary}\label{LeakyTRLimit}
    For rational functions $ \bar{w}(s)$ and $S(z)$ such that $ d x$ from \eqref{SpectralCurveParam} has simple zeroes and poles, the multidifferentials of \cref{Correlators} for the leaky Hurwitz numbers with cut-and-join operator \eqref{CJ} such that $ w(\frac{s}{\hbar}) = \frac{1}{\hbar} \bar{w}(s) $ are calculated by topological recursion on the spectral curve \eqref{SpectralCurveParam}.
\end{corollary}

\begin{proof}
    Assume $ \bar{w}(s)$ and $S(z)$ are rational functions such that $ dx$ has simple zeroes and poles. As having simple zeroes and poles is an open condition, we can deform $ \bar{w}$ to a rational function of the same degree with no repeated factors, while preserving the condition that $ dx $ has simple zeroes and poles. Topological recursion (in any formulation) is manifestly analytic under such deformations, as also follows from \cref{TRLimits}.
\end{proof}

\begin{remark}
    With the more sophisticated limit arguments of \cite{BBCKS25,ABDKS25}, this corollary can be extended to cases where different special points in the sense of \cref{DefSpecialPoints} collide. As we do not currently know of any specifically interesting cases that require this, we omit this extension.
\end{remark}

\begin{example}
    In the $r$-completed cycles, $k$-leaky, $q$-orbifold case, $ w(\frac{s}{\hbar}) = \frac{1}{\hbar} \frac{s^r}{r}$, and $ S(z) = z^q$. Then
    \begin{equation}
		dx
		=
		\frac{1 - (qr+k) \frac{r-1}{r} z^{qr - q +k }}{z ( 1 + \frac{k}{r} z^{qr-q + k} )} dz
	\end{equation}
    has simple zeroes and poles, so its multidifferentials are calculated by topological recursion.
\end{example}

\begin{remark}
    In the non-leaky case, the completed cycles Hurwitz numbers seemed to have a non-trivial $\hbar$-deformation for their cut-and-join operator, despite being the numbers reconstructed by topological recursion: looking at \cite[Remark 1.4]{BDKS22} and \cite[Table 2]{BDKS24}, we see that the generating series capturing the cut-and-join equation is (in our conventions)
    \begin{equation}
        \hat{\psi} (s; \hbar ) = \frac{(s + \hbar/2)^r - (s - \hbar/2)^r}{r\hbar}.
    \end{equation}
    We claim that this is an artifact: the function $ \hat{\psi} $ is not the right generating function, but rather our function  $ \mf{W}$ or in this case $w$ (which appears also in \cite[Equation (58)]{BDKS24} as an auxiliary function). In this non-leaky, $k=0$, setting, we can recover $ \psi$ as $ \lim_{k\to 0} \log \bar{\phi}_1 (z,s)$ as in \cref{FindingpsiFork=0}, while $ e^{\hat{\psi}}$ is given using \cref{phi_mn} by the limit
    \begin{equation}
        \begin{split}
            \lim_{k \to 0} \phi_1 (z,s)
            &=
            \lim_{k \to 0} \sum_{n=0}^\infty \phi_{1,n} (s) z^{nk}
            \\
            &=
            \lim_{k \to 0} \sum_{n=0}^\infty \sum_{i=0}^n \frac{(-1)^{n-i}z^{kn}}{i!(n-i)!} \prod_{j=0}^{i-1} \frac{\hbar^{r-1}\big( \frac{s}{\hbar} + \frac{1 + (2j+1-n)k}{2} \big)^r }{r}  \prod_{j=i}^{n-1} \frac{\hbar^{r-1} \big( \frac{s}{\hbar} + \frac{-1 + (2j+1-n)k}{2} \big)^r }{r} 
            \\
            &=
            \sum_{n=0}^\infty \sum_{i=0}^n \frac{(-1)^{n-i}}{i!(n-i)!}  \Big( \frac{( s + \frac{\hbar}{2} )^r}{r\hbar} \Big)^i \Big( \frac{( s - \frac{\hbar}{2} )^r}{r\hbar} \Big)^{n-i}
            \\
            &=
            \exp \Big( \frac{(s + \hbar /2)^r - (s - \hbar/2)^r}{r \hbar} \Big),
        \end{split}
    \end{equation}
    and this limit only depends on $ \bar{w} (s) = \hbar w( \frac{s}{\hbar} )$ in these cases.
\end{remark}

\begin{example}
    The original leaky Hurwitz numbers studied in \cite{CMR25}, expressed in Fock space in \cref{CMRFockSpace}, have cut-and-join operator $ \hbar M_k = \hbar (\mc{F}_2^k - \frac{k^2 - 1}{24} \mc{F}^k_0)$. This means that $ w(j) = \hbar (j^2 - \frac{k^2 -1}{24}) $, and hence $ w(\frac{s}{\hbar}) = \hbar^{-1} s^2 - \hbar \frac{k^2 -1}{24} \neq \frac{1}{\hbar} \bar{w}(s)$. Therefore, these in-completed leaky Hurwitz numbers are \emph{not} computed by topological recursion on the spectral curve
    \begin{equation}
		\begin{cases}
			e^{x(z)}
            =
            X(z)
			&=
			z \big( 1+ k S(z)z^k\big)^{-\frac{2}{k}}
			\\
			y (z)
			&=
			S(z) + k (S(z))^2 z^k
		\end{cases}
	\end{equation}
    for generic $ S$, even though this does compute the genus-$0$ sector. An appropriate version of blobbed topological recursion \cite{BS17,ABDKS25blobbed} would probably fit this example, but this goes beyond the current work.
\end{example}

To prove \cref{LeakyTRThm}, we will first need several lemmas. As the case $ n = 1$ is a bit different from $ n > 1$, and more complicated, we will first deal with the latter case.

We will construct the spectral curve via $x$--$y$ duality and symplectic duality, and using the expressions from \cref{Thmx-y,ThmSD} for the changes of (extended) multidifferentials under these transformations. These transformations preserve topological recursion, and we will show that the resulting multidifferentials agree with the one found in \cref{LeakyMultidifferentials}.

Explicitly, we will go through the following chain of transformations starting from \cref{TrivialTR}:\footnote{There is some abuse of notation here: strictly speaking, we should only work with $ dx_j$ and $ dy_j$, because of logarithmic singularities and their incorporation in log topological recursion.}

\begin{equation}\label{ChainOfTransformations}
	\begin{split}
		\begin{cases}
			x_0 & = \log z
			\\
			y_0 & = S(z)
		\end{cases}
		&\to
		\begin{cases}
			x_1 & = y_0
			\\
			y_1 & = x_0
		\end{cases}
		\to
		\begin{cases}
			x_2 & = x_1
			\\
			y_2 & = y_1 + \psi (x_1)
		\end{cases}
		\to
		\begin{cases}
			x_3 & = y_2
			\\
			y_3 & = x_2
		\end{cases}
		\\
		&\to
		\begin{cases}
			x_4 & = x_3
			\\
			y_4 & = y_3 + \chi (x_3)
		\end{cases}
		\to
		\begin{cases}
			x_5 & = y_4
			\\
			y_5 & = x_4
		\end{cases}
		\to
		\begin{cases}
			x_6 & = x_5
			\\
			y_6 & = y_5 + \psi^\vee (x_5)
		\end{cases}
		\to
		\begin{cases}
			x & = y_6
			\\
			y & = x_6
		\end{cases}
		,
	\end{split}
\end{equation}
where
\begin{equation}
	\psi( \theta ) 
	= 
	- \psi^\vee (\theta ) 
	= 
	\frac{1}{k} \log \bar{w} (\theta ) , 
	\qquad 
	\chi (\theta ) 
	= 
	k e^{k \theta} .
\end{equation}

We see that all of these functions are one of the following five:
\begin{equation}\label{FiveFunctions}
	\zeta
	=
	\log z , 
	\quad 
	s 
	=
	S(z) , 
	\quad 
	g
	=
	\zeta + \psi (s) , 
	\quad 
	x
	=
	\zeta + \psi (s) + \psi^\vee \big( s + \chi (\zeta + \psi (s)) \big) , 
	\quad 
	y
	=
	s + \chi (\zeta + \psi (s)) .
\end{equation}
So, we will go through a sequence of three times $x$--$y$ duality and symplectic duality, followed by a final $ x$--$y$ duality. We will first consider the $ n > 1$ case, and then give an additional argument for the harder $ n =1 $ case.

For the remainder of this section, we will consider $ \zeta = \log z$ and $ s$ as formally independent variables, and consider $ g$, $ x$, and $ y$ as functions of these two variables through the last three equalities in \cref{FiveFunctions}. We will explicitly specialise $ s$ to $ S(z)$ at the end of the argument.

\subsection{The \texorpdfstring{$ n > 1$}{n>1} cases}

\begin{lemma}\label{DualitiesPreserveTR}
	Assume that $ S(z)$ is a rational function and $ \bar{w}(s)$ is rational with no repeated factors, and take
	\begin{equation}\label{TransformationDeformations}
		\tilde{\psi} ( \theta )
		=
		- \tilde{\psi}^\vee (\theta)
		=
		\frac{1}{k \mc{S}( k \hbar \del_\theta )} \log \bar{w}(\theta )
		\,,
		\qquad
		\tilde{\chi} (\theta )
		=
		\chi (\theta )
		=
		k e^{k \theta} .
	\end{equation}
	Then the multidifferentials for $ n > 1$ obtained by logarithmic topological recursion on the spectral curve \cref{SpectralCurveParam} are given by
	\begin{equation}
		\tilde{\omega}_{g,n} (z_{\llbracket n\rrbracket}) 
		=
		[\hbar^{2g-2+n}] \Big( \prod_{i=1}^n \tilde{T}_i \Big) \mb{W}_n^{\mc{S}_0} (z_{\llbracket n\rrbracket}, u_{\llbracket n\rrbracket}) ,
	\end{equation}
	where
	\begin{equation}
	\begin{split}
		\tilde{T} f (z)
		&=
		\sum_{j,p,q,t=0}^\infty (-d \frac{1}{dx})^j [v^j] 
		e^{v ( \mc{S}(v\hbar \del_{y}) \tilde{\psi}^\vee (y) - x )} \big( -d \frac{1}{dy} \big)^p e^{vy_2} [b^p ]
		e^{b ( \mc{S}(b \hbar \del_{y_2}) \tilde{\chi}(y_2) - y )} 
		\\
		&\hspace{2cm}
		\big( -d \frac{1}{dy_2} \big)^q e^{bx_2} [a^q ]
		e^{a ( \mc{S}(a\hbar \del_{y_0}) \tilde{\psi}(y_0) - y_2 )} \big( -d \frac{1}{dy_0} \big)^t e^{ax_0} [u^t ]
		f(u,z)
		\\
		&=
		\Big\lfloor_{s \to S(z)} \sum_{j,p,q,t=0}^\infty (-d \frac{1}{dx})^j [v^j] 
		e^{v ( \mc{S}(v\hbar \del_{y}) \tilde{\psi}^\vee(y) - x )} \big( -d \frac{1}{dy} \big)^p e^{vg} [b^p ]
		e^{b ( \mc{S}(b\hbar \del_g) \tilde{\chi}(g) - y )}
		\\
		&\hspace{2cm}
		\big( -d \frac{1}{dg} \big)^q e^{bs} [a^q ]
		e^{a ( \mc{S}(a\hbar \del_{s}) \tilde{\psi}(s) - g )} \big( -d \frac{1}{ds} \big)^t e^{a\zeta} [u^t ]
		f(u,z) .
	\end{split}
	\end{equation}
\end{lemma}

\begin{proof}
	By \cref{TrivialTR}, the $ \mb{W}^{\mc{S}_0}_n$ are the extended multidifferentials associated by topological recursion to $ \mc{S}_0$.

	By \cref{Extendedx-y,ExtendedSD}, we find that for the transformation $ (x, y) \mapsto (\hat{x} = y, \hat{y} = x + \psi(y))$, the extended correlators should transform as
	\begin{equation}\label{DoubleTransform}
		\begin{split}
			\hat{\mb{W}}_n (z_{\llbracket n\rrbracket}, v_{\llbracket n\rrbracket})
			&=
			(-1)^n \prod_{i=1}^n \bigg(e^{v_i ( \mc{S}(v_i\hbar \del_{y_i}) \tilde{\psi}(y_i) - \psi (y_i) )} \sum_{r=0}^\infty e^{-v_i x_i} \big( -d_i \frac{1}{dy_i} \big)^r e^{v_ix_i} [u_i^r ] \bigg) \mb{W}_n (z_{\llbracket n\rrbracket}, u_{\llbracket n\rrbracket} )
			\\
			&\hspace{6cm}
			+ \delta_{n,1} e^{v_1 ( \mc{S}(v_1\hbar \del_{y_1}) \tilde{\psi}(y_1) - \psi (y_1) )} \frac{dy_1}{v_1\hbar}
			\\
			&=
			(-1)^n \prod_{i=1}^n \bigg(e^{v_i ( \mc{S}(v_i\hbar \del_{y_i}) \tilde{\psi}(y_i) - \hat{y}_i )} \sum_{r=0}^\infty \big( -d_i \frac{1}{dy_i} \big)^r e^{v_ix_i} [u_i^r ] \bigg) \mb{W}_n (z_{\llbracket n\rrbracket}, u_{\llbracket n\rrbracket} ) 
			\\
			&\hspace{6cm}
			+ \delta_{n,1} e^{v_1 ( \mc{S}(v_1\hbar \del_{y_1}) \tilde{\psi}(y_1) - \psi (y_1) )} \frac{dy_1}{v_1\hbar} .
		\end{split}
	\end{equation}
	with $ \tilde{\psi}$ the deformation of $ \psi $ given in \cref{ThmSD}. We are in the situation of the second (resp. fourth) case of \cref{ThmSD} for $ \psi$ and $ \psi^\vee$ (resp. $\chi$), so we find the deformations given in \cref{TransformationDeformations}. So, we apply this three times, for $ \psi$, $ \chi$, and $ \psi^\vee$, and combine it finally with \cref{NonExtendedx-y} to obtain the formula in the lemma.
\end{proof}

Comparing \cref{LeakyMultidifferentials,DualitiesPreserveTR}, we see that we want to prove $ \omega^{(g)}_n = \tilde{\omega}^{(g)}_n$, or $ T = \tilde{T}$. In fact, we will show this last equality term by term in $t$, i.e. we will prove equality between the following two maps on one-forms $F(z)$ in $z$.

\begin{align}
	V_t F (z)
	&=
	\sum_{j = 0}^\infty \big(- d \frac{1}{dx}\big)^j  [v^j] L_t(-v, s,z,\hbar) F(z)
	\\
	\begin{split}
		\tilde{V}_t F (z)
		&=
		\sum_{j,p,q=0}^\infty (-d \frac{1}{dx})^j [v^j] 
		e^{v ( \mc{S}(v\hbar \del_{y}) \tilde{\psi}^\vee(y) - x )} \big( -d \frac{1}{dy} \big)^p e^{vg} [b^p ]
		e^{b ( \mc{S}(b\hbar \del_g) \tilde{\chi}(g) - y )}
		\\
		&\hspace{2cm}
		\big( -d \frac{1}{dg} \big)^q e^{bs} [a^q ]
		e^{a ( \mc{S}(a\hbar \del_{s}) \tilde{\psi}(s) - g )} \big( -d \frac{1}{ds} \big)^t e^{a\zeta}
		F(z) .
	\end{split}
\end{align}

\begin{lemma} \label{VrIsMultiplication}
	For any $t \geq 0$, and any formal power series in $ \tilde{\psi} ,  \tilde{\chi}, \tilde{\psi}^\vee \in \C \llbracket \theta, \hbar \rrbracket $, with $ \tilde{\psi} \big|_{\hbar = 0} = \psi$ and similarly for the others, the map $\tilde{V}_t $ is given by
	\begin{equation}
		\begin{split}
			\tilde{V}_t F (z)
			&=
			\sum_{j,p,q=0}^\infty (-d \frac{1}{dx})^j [v^j] 
			\bigg(  \Big\lfloor_{\eta \to y} e^{-v \psi^\vee (\eta)} \del_\eta^p  e^{v \mc{S}(v \hbar \del_\eta) \tilde{\psi}^\vee (\eta)}\bigg)
			[b^p ] \bigg( \Big\lfloor_{\gamma \to g}e^{-v \gamma} e^{-b \chi (\gamma)} \del_\gamma^q e^{v \gamma} e^{b \mc{S}(b \hbar \del_\gamma)\tilde{\chi} (\gamma) } \bigg) 
			\\
			&\hspace{4cm}
			[a^q] \bigg(  \Big\lfloor_{\sigma \to s} e^{-b \sigma} e^{-a \psi (\sigma)} \del_\sigma^t e^{b \sigma} e^{a \mc{S}(a \hbar \del_\sigma)\tilde{\psi} (\sigma) } \bigg) F(z) .
		\end{split}
	\end{equation}
\end{lemma}

\begin{proof}

	We apply \cite[(73)]{ABDKS24} three times, in the slightly modified version
	\begin{equation}\label{FormalIntByParts}
		\sum_{j=0}^\infty ( - d \frac{1}{dx^\dagger} )^j [v^j] A \del_y^r B
		=
		\sum_{j=0}^\infty ( - d \frac{1}{dx^\dagger} )^j [v^j] B e^{- v x^\dagger} (- d \frac{1}{dy} )^r e^{v x^\dagger} A
	\end{equation}
	for a function $ B(z,v)$ and a one-form in $z$ $A(z,v)$.
	
	First, we use this for
	\begin{equation}
		A_1 
		= 
		e^{-a \psi (s )} F(z) \,, 
		\qquad 
		B_1
		=
		e^{b s} e^{a \mc{S} (a \hbar \del_s )\tilde{\psi} (s)}
	\end{equation}
	to get
	\begin{equation*}
		\begin{split}
			&
			\sum_{q = 0}^\infty \big( -d \frac{1}{dg} \big)^q e^{bs} [a^q ]
			e^{a ( \mc{S}(a\hbar \del_s) \tilde{\psi}(s) - g )} \big( -d \frac{1}{ds} \big)^t e^{a \zeta} F(z)
			\\
			&\qquad=
			\sum_{q = 0}^\infty \big( -d \frac{1}{dg} \big)^q [a^q ] \bigg(  \Big\lfloor_{\sigma \to s} e^{-a \psi (\sigma)} \del_\sigma^t e^{b \sigma} e^{a \mc{S}(a \hbar \del_\sigma)\tilde{\psi} (\sigma) } \bigg) F(z) .
		\end{split}
	\end{equation*}
	Then we apply it again for
	\begin{equation}
		A_2
		=
		e^{-b y} [a^q ] \bigg(  \Big\lfloor_{\sigma \to s} e^{-a \psi (\sigma)} \del_\sigma^t e^{b \sigma} e^{a \mc{S}(a \hbar \del_\sigma)\tilde{\psi} (\sigma) } \bigg) F(z) \,,
		\qquad
		B_2
		=
		e^{vg} e^{b \mc{S}(b \hbar \del_g) \tilde{\chi} (g)}
	\end{equation}
	to find
	\begin{equation*}
		\begin{split}
			&\sum_{p,q = 0}^\infty \big( -d \frac{1}{dy} \big)^p e^{vg} [b^p ]
			e^{b ( \mc{S}(b\hbar \del_g) \tilde{\chi}(g) - y )} \big( -d \frac{1}{dg} \big)^q e^{bs} [a^q ]
			e^{a ( \mc{S}(a\hbar \del_s) \tilde{\psi}(s) - g )} \big( -d \frac{1}{ds} \big)^t e^{a\zeta} F(z)
			\\
			&=
			\sum_{p,q = 0}^\infty \big( -d \frac{1}{dy} \big)^p e^{vg} [b^p ]
			e^{b ( \mc{S}(b\hbar \del_g) \tilde{\chi}(g) - y )} \big( -d \frac{1}{dg} \big)^q [a^q ]
			\bigg(  \Big\lfloor_{\sigma \to s} e^{-a \psi (\sigma)} \del_{\sigma}^t e^{b \sigma} e^{a \mc{S}(a \hbar \del_{\sigma})\tilde{\psi} (\sigma) } \bigg) F(z) 
			\\
			&=
			\sum_{p,q=0}^\infty \big( -d \frac{1}{dy} \big)^p  [b^p ] 
			\bigg(  \Big\lfloor_{\gamma \to g} e^{-b \chi (\gamma)} \del_{\gamma}^q e^{v \gamma} e^{b \mc{S}(b \hbar \del_{\gamma})\tilde{\chi} (\gamma) } \bigg) [a^q]
			\bigg(  \Big\lfloor_{\sigma \to s} e^{-b \sigma} e^{-a \psi (\sigma)} \del_{\sigma}^t e^{b \sigma} e^{a \mc{S}(a \hbar \del_{\sigma})\tilde{\psi} (\sigma) } \bigg) F(z) .
		\end{split}
	\end{equation*}
	Finally, we do this a third time with
	\begin{equation}
		\begin{split}
			A_3
			&=
			e^{-vx} [b^p ] 
			\bigg(  \Big\lfloor_{\gamma \to g} e^{-b \chi (\gamma)} \del_\gamma^q e^{v \gamma} e^{b \mc{S}(b \hbar \del_\gamma) \tilde{\chi} (\gamma) } \bigg) [a^q]
			\bigg(  \Big\lfloor_{\sigma \to s} e^{-b \sigma} e^{-a \psi (\sigma)} \del_\sigma^t e^{b \sigma} e^{a \mc{S}(a \hbar \del_\sigma )\tilde{\psi} (\sigma) } \bigg) F(z)
			\,,
			\\
			B_3
			&=
			e^{v \mc{S}(v \hbar \del_y) \tilde{\psi}^\vee (y)} 
		\end{split}
	\end{equation}
	to find
	\begin{align*}
			\tilde{V}_t F (z)
			&=
			\sum_{j,p,q=0}^\infty (-d \frac{1}{dx})^j [v^j] 
			e^{v ( \mc{S}(v\hbar \del_{y}) \tilde{\psi}^\vee(y) - x )} \big( -d \frac{1}{dy} \big)^p e^{vg} [b^p ]
			e^{b ( \mc{S}(b\hbar \del_g) \tilde{\chi}(g) - y )} \big( -d \frac{1}{dg} \big)^q e^{bs} 
			\\
			&\hspace{4cm}
			[a^q ]
			e^{a ( \mc{S}(a \hbar \del_s) \tilde{\psi}(s) - g )} \big( -d \frac{1}{ds} \big)^t e^{a \zeta} F(z)
			\\
			&=
			\sum_{j,p,q=0}^\infty (-d \frac{1}{dx})^j [v^j] 
			e^{v ( \mc{S}(v\hbar \del_{y}) \tilde{\psi}^\vee(y) - x )} \big( -d \frac{1}{dy} \big)^p [b^p ]
			\bigg( \Big\lfloor_{\gamma \to g} e^{-b \chi (\gamma)} \del_{\gamma}^q e^{v \gamma} e^{b \mc{S}(b \hbar \del_\gamma)\tilde{\chi} (\gamma) } \bigg) 
			\\
			&\hspace{4cm} 
			[a^q] \bigg(  \Big\lfloor_{\sigma \to s} e^{-b \sigma} e^{-a \psi (\sigma)} \del_\sigma^t e^{b \sigma} e^{a \mc{S}(a \hbar \del_\sigma)\tilde{\psi} (\sigma) } \bigg) F(z)
			\\
			&=
			\sum_{j,p,q=0}^\infty (-d \frac{1}{dx})^j [v^j] 
			\bigg(  \Big\lfloor_{\eta \to y} e^{-v \psi^\vee (\eta)} \del_\eta^p  e^{v \mc{S}(v \hbar \del_\eta) \tilde{\psi}^\vee (\eta)}\bigg)
			[b^p ] \bigg( \Big\lfloor_{\gamma \to g}e^{-v \gamma} e^{-b \chi (\gamma)} \del_\gamma^q e^{v \gamma} e^{b \mc{S}(b \hbar \del_\gamma)\tilde{\chi} (\gamma) } \bigg) 
			\\
			&\hspace{4cm}
			[a^q] \bigg(  \Big\lfloor_{\sigma \to s} e^{-b \sigma} e^{-a \psi (\sigma)} \del_\sigma^t e^{b \sigma} e^{a \mc{S}(a \hbar \del_\sigma)\tilde{\psi} (\sigma) } \bigg) F(z) . \qedhere
	\end{align*}
\end{proof}

\begin{lemma}\label{SubstitutionOfDerivatives}
	Under the same conditions as \cref{VrIsMultiplication},
	\begin{equation}
		\begin{split}
			&\sum_{p,q=0}^\infty 
			\bigg(  \Big\lfloor_{\eta \to y} e^{-v \psi^\vee (\eta)} \del_\eta^p  e^{v \mc{S}(v \hbar \del_{\eta}) \tilde{\psi}^\vee (\eta)}\bigg)
			[b^p ]
			\bigg( \Big\lfloor_{\gamma \to g} e^{-v \gamma} e^{-b \chi (\gamma)} \del_\gamma^q e^{v \gamma} e^{b \mc{S}(b \hbar \del_\gamma) \tilde{\chi} (\gamma) } \bigg) [a^q]
			\\
			&\hspace{6cm}
			\bigg(  \Big\lfloor_{\sigma \to s} e^{-b \sigma} e^{-a \psi (\sigma)} \del_\sigma^t e^{b \sigma} e^{a \mc{S}(a \hbar \del_\sigma)\tilde{\psi} (\sigma) } \bigg)
			\\
			&\qquad=
			e^{v (\zeta - x)}
			\del_s^t \Big\lfloor_{\bar{\sigma} \to s}
			e^{ (\del_\zeta+v) \mc{S}( \hbar (\del_\zeta + v) \del_{\bar{\sigma}} )\tilde{\psi} (\bar{\sigma}) }
			e^{ \del_s \mc{S}( \hbar \del_\zeta \del_s ) \tilde{\chi} (\zeta ) }
			e^{v \mc{S}(v \hbar \del_s) \tilde{\psi}^\vee (s)}.
		\end{split}
	\end{equation}
\end{lemma}

\begin{proof}
	This is an explicit calculation equating functions, using twice that $ \sum_{p = 0}^\infty [u^p]f(u) \del_\theta^p h(\theta) = f(\del_\theta ) h(\theta)$ to obtain
	\begin{align*}
			\sum_{p,q=0}^\infty
			&
			\bigg(  \Big\lfloor_{\eta \to y} e^{-v \psi^\vee (\eta)} \del_\eta^p  e^{v \mc{S}(v \hbar \del_{\eta}) \tilde{\psi}^\vee (\eta)}\bigg)
			[b^p ]
			\bigg( \Big\lfloor_{\gamma \to g} e^{-v \gamma} e^{-b \chi (\gamma)} \del_\gamma^q e^{v \gamma} e^{b \mc{S}(b \hbar \del_\gamma) \tilde{\chi} (\gamma) } \bigg) [a^q]
			\\
			&\hspace{7cm}
			\bigg(  \Big\lfloor_{\sigma \to s} e^{-b \sigma} e^{-a \psi (\sigma)} \del_\sigma^t e^{b \sigma} e^{a \mc{S}(a \hbar \del_\sigma)\tilde{\psi} (\sigma) } \bigg)
			\\
			&=
			\sum_{p,q=0}^\infty \Big\lfloor_{\substack{\eta \to y \\ \gamma \to g \\ \sigma \to s}}
			[a^q b^p ]
			\bigg( e^{-b s} e^{-a \psi (s)} \del_\sigma^t e^{b \sigma} e^{a \mc{S}(a \hbar \del_\sigma)\tilde{\psi} (\sigma) } \bigg)
			\bigg( e^{-v g} e^{-b \chi (g)} \del_\gamma^q e^{v \gamma} e^{b \mc{S}(b \hbar \del_\gamma ) \tilde{\chi} (\gamma) } \bigg)
			\\
			&\hspace{7cm}
			\bigg( e^{-v \psi^\vee (y)} \del_{\eta}^p  e^{v \mc{S}(v \hbar \del_{\eta}) \tilde{\psi}^\vee (\eta)}\bigg)
			\\
			&=
			e^{-v (\psi^\vee (y)+ g)}
			\sum_{p,q=0}^\infty \Big\lfloor_{\substack{\eta \to y \\ \gamma \to g \\ \sigma \to s}}
			[a^q b^p ]
			e^{-b (s+ \chi (g))} e^{-a \psi (s)} 
			\bigg( \del_\sigma^t e^{b \sigma} e^{a \mc{S}(a \hbar \del_\sigma)\tilde{\psi} (\sigma) } \bigg)
			\\
			&\hspace{7cm}
			\bigg( \del_\gamma^q e^{v \gamma} e^{b \mc{S}(b \hbar \del_\gamma)\tilde{\chi} (\gamma) } \bigg)
			\bigg( \del_\eta^p  e^{v \mc{S}(v \hbar \del_{\eta}) \tilde{\psi}^\vee (\eta)}\bigg)
			\\
			&=
			e^{-v x}
			\Big\lfloor_{\substack{\eta \to y \\ \gamma \to g \\ \sigma \to s}}
			e^{- y \del_\eta} e^{- \psi (s) \del_\gamma} 
			\del_\sigma^t e^{\sigma \del_\eta} e^{ \del_\gamma \mc{S}( \hbar \del_\gamma \del_\sigma )\tilde{\psi} (\sigma) }
			e^{v \gamma} e^{ \del_\eta \mc{S}( \hbar \del_\gamma \del_\eta)\tilde{\chi} (\gamma) }
			e^{v \mc{S}(v \hbar \del_\eta) \tilde{\psi}^\vee (\eta)}
			\\
			&=
			e^{-v x}
			\Big\lfloor_{\substack{\eta \to y \\ \gamma \to g \\ \sigma \to s}}
			e^{- y \del_\eta} e^{- \psi (s) \del_\gamma} 
			\del_\sigma^t \Big\lfloor_{\bar{\sigma} \to \sigma} 
			e^{\sigma \del_\eta} e^{ \del_\gamma \mc{S}( \hbar \del_\gamma \del_{\bar{\sigma}} )\tilde{\psi} (\bar{\sigma} ) }
			e^{v \gamma} e^{ \del_\eta \mc{S}( \hbar \del_\gamma \del_\eta)\tilde{\chi} (\gamma) }
			e^{v \mc{S}(v \hbar \del_\eta) \tilde{\psi}^\vee (\eta)}
			\\
			&=
			e^{-v x}
			\Big\lfloor_{\substack{\eta \to y \\ \gamma \to g \\ \sigma \to s}}
			\del_\sigma^t \Big\lfloor_{\bar{\sigma} \to \sigma}
			e^{ \del_\gamma \mc{S}( \hbar \del_\gamma \del_{\bar{\sigma}} )\tilde{\psi} ( \bar{\sigma}) }
			e^{v (\gamma - \psi (s))} e^{ \del_\eta \mc{S}( \hbar \del_\gamma \del_\eta)\tilde{\chi} (\gamma - \psi (s)) }
			e^{v \mc{S}(v \hbar \del_\eta) \tilde{\psi}^\vee (\eta - y + \sigma)}
			\\
			\intertext{Here, $\eta$ and $ \sigma$ only appear in the combination $ \eta - y + \sigma$, so we can exchange $ \del_\eta $ for $ \del_\sigma$:}
			&=
			e^{-v (x + \psi (s) )}
			\Big\lfloor_{\substack{\eta \to y \\ \gamma \to g \\ \sigma \to s}}
			\del_\sigma^t \Big\lfloor_{\bar{\sigma} \to \sigma}
			e^{ \del_\gamma \mc{S}( \hbar \del_\gamma \del_{\bar{\sigma}} )\tilde{\psi} ( \bar{\sigma}) }
			e^{v \gamma } e^{ \del_\sigma \mc{S}( \hbar \del_\gamma \del_\sigma) \tilde{\chi} (\gamma - \psi (s)) }
			e^{v \mc{S}(v \hbar \del_\sigma) \tilde{\psi}^\vee (\eta - y + \sigma)}
			\\
			&=
			e^{v (-x - \psi (s) + g)}
			\Big\lfloor_{\substack{\gamma \to g \\ \sigma \to s}}
			\del_\sigma^t 
			\Big\lfloor_{\bar{\sigma} \to \sigma}
			e^{ (\del_\gamma+v) \mc{S}( \hbar (\del_\gamma + v) \del_{\bar{\sigma}} )\tilde{\psi} (\bar{\sigma}) }
			e^{ \del_\sigma \mc{S}( \hbar \del_\gamma \del_\sigma)\tilde{\chi} (\gamma - \psi (s)) }
			e^{v \mc{S}(v \hbar \del_\sigma) \tilde{\psi}^\vee ( \sigma)}
			\\
			&=
			e^{v (\zeta - x)}
			\Big\lfloor_{\sigma \to s}
			\del_\sigma^t \Big\lfloor_{\bar{\sigma} \to \sigma}
			e^{ (\del_g+v) \mc{S}( \hbar (\del_g + v) \del_{\bar{\sigma}} )\tilde{\psi} (\bar{\sigma}) }
			e^{ \del_\sigma \mc{S}( \hbar \del_g \del_\sigma ) \tilde{\chi} (g - \psi (s)) }
			e^{v \mc{S}(v \hbar \del_\sigma) \tilde{\psi}^\vee (\sigma)} 
			\\
			\intertext{Now, $ g$ and $ s$ only appear in the combination $ g - \psi (s) = \zeta$, so we can write that and change $ \del_g$ into $ \del_\zeta$:}
			&=
			e^{v (\zeta - x)}
			\Big\lfloor_{\sigma \to s}
			\del_\sigma^t \Big\lfloor_{\bar{\sigma} \to \sigma}
			e^{ (\del_\zeta+v) \mc{S}( \hbar (\del_\zeta + v) \del_{\bar{\sigma}} )\tilde{\psi} (\bar{\sigma}) }
			e^{ \del_\sigma \mc{S}( \hbar \del_\zeta \del_\sigma ) \tilde{\chi} (\zeta ) }
			e^{v \mc{S}(v \hbar \del_\sigma) \tilde{\psi}^\vee (\sigma)} 
			\\
			&=
			e^{v (\zeta - x)}
			\del_s^t \Big\lfloor_{\bar{\sigma} \to s}
			e^{ (\del_\zeta+v) \mc{S}( \hbar (\del_\zeta + v) \del_{\bar{\sigma}} )\tilde{\psi} (\bar{\sigma}) }
			e^{ \del_s \mc{S}( \hbar \del_\zeta \del_s ) \tilde{\chi} (\zeta ) }
			e^{v \mc{S}(v \hbar \del_s) \tilde{\psi}^\vee (s)} .
			\qedhere
	\end{align*}
\end{proof}

At this point, it is useful to use the actual values of the functions $ \psi$, $ \chi$, and $ \psi^\vee$ and their deformations.

\begin{lemma}\label{ExplicitDualityChain}
	With the assumptions of \cref{DualitiesPreserveTR},
	\begin{equation}
		\begin{split}
			&
			\del_s^t \Big\lfloor_{\bar{\sigma} \to s}
			e^{ (\del_\zeta+v) \mc{S}( \hbar (\del_\zeta + v) \del_{\bar{\sigma}} )\tilde{\psi} (\bar{\sigma}) }
			e^{ \del_s \mc{S}( \hbar \del_\zeta \del_s ) \tilde{\chi} (\zeta ) }
			e^{v \mc{S}(v \hbar \del_s) \tilde{\psi}^\vee (s)}
			\\
			&\qquad=
			\sum_{n=0}^\infty \sum_{i = 0}^n \frac{(-1)^{n-i} }{i!(n-i)! \hbar^n } e^{ nk \zeta}
			\del_s^t
			\prod_{j = 0}^{i-1} \bar{w} ( s + \frac{-v+ (2j+1-n) k}{2} \hbar)
			\prod_{j = i}^{n-1} \bar{w} ( s + \frac{ v+ (2j+1-n) k}{2} \hbar)
		\end{split}
	\end{equation}
\end{lemma}

\begin{proof}
	Starting from the right, we find
	\begin{equation}
		\begin{split}
			v \mc{S}(v \hbar \del_s ) \tilde{\psi}^\vee (s)
			&=
			- \frac{v \mc{S}(v \hbar \del_s )}{k \mc{S}( k \hbar \del_s )} \log \bar{w} (s)
			\\
			&=
			- e^{\frac{v-k}{2} \hbar \del_s} \frac{1 - e^{-v \hbar \del_s }}{1 - e^{-k \hbar \del_s }} \log \bar{w} (s)
			\\
			&=
			\sum_{j = 0}^ \infty \log \bar{w} (s + \frac{-v- (2j + 1) k}{2} \hbar) -\log \bar{w} (s + \frac{v- (2j + 1) k}{2} \hbar) ,
		\end{split}
	\end{equation}
	so that
	\begin{equation}
		e^{v \mc{S}(v \hbar \del_s) \tilde{\psi}^\vee (s)}
		=
		\prod_{j = 0}^ \infty  \frac{\bar{w} (s + \frac{-v- (2j + 1) k}{2} \hbar)}{\bar{w} (s + \frac{v- (2j + 1) k}{2} \hbar)} .
	\end{equation}
	Similarly, the left factor is
	\begin{equation}
		e^{ (\del_\zeta+v) \mc{S}( \hbar (\del_\zeta + v) \del_{\bar{\sigma}} )\tilde{\psi} (\bar{\sigma}) }
		=
		\prod_{j = 0}^ \infty  \frac{\bar{w} ( \bar{\sigma} + \frac{\del_\zeta + v- (2j + 1) k}{2} \hbar)}{\bar{w} ( \bar{\sigma} + \frac{- \del_\zeta -v- (2j + 1) k}{2} \hbar)} .
	\end{equation}
	For the middle factor, we find
	\begin{equation}
		\del_s \mc{S}( \hbar \del_\zeta \del_s ) \tilde{\chi} (\zeta )
		=
		\del_s \mc{S}( \hbar \del_\zeta \del_s) k e^{k\zeta}
		=
		k e^{k\zeta} \del_s \mc{S}( \hbar k \del_s)
		=
		e^{k\zeta}  \frac{e^{\frac{k \hbar}{2} \del_s} - e^{- \frac{k\hbar}{2} \del_s}}{\hbar}
	\end{equation}
	and therefore
	\begin{equation}
		e^{\del_s \mc{S}( \hbar \del_\zeta \del_s)\tilde{\chi} (\zeta)}
		=
		\sum_{n=0}^\infty \sum_{i = 0}^n \frac{(-1)^{n-i}}{i!(n-i)!\hbar^n} e^{nk\zeta}  e^{\frac{(2i-n)k \hbar}{2} \del_s} 
	\end{equation}
	Putting all of this together,
	\begin{align*}
			&
			\del_s^t \Big\lfloor_{\bar{\sigma} \to s}
			e^{ (\del_\zeta + v) \mc{S}( \hbar (\del_\zeta + v) \del_{\bar{\sigma}} )\tilde{\psi} (\bar{\sigma}) }
			e^{ \del_s \mc{S}( \hbar \del_\zeta \del_s) \tilde{\chi} (\zeta) }
			e^{v \mc{S}(v \hbar \del_s ) \tilde{\psi}^\vee (s)}
			\\
			&\qquad=
			\del_s^t
			\prod_{j = 0}^ \infty \frac{\bar{w} ( s + \frac{\del_\zeta + v- (2j + 1) k}{2} \hbar)}{\bar{w} ( s + \frac{- \del_\zeta -v- (2j + 1) k}{2} \hbar)}
			\sum_{n=0}^\infty \sum_{i = 0}^n \frac{(-1)^{n-i}}{i!(n-i)!\hbar^n} e^{nk\zeta}  e^{\frac{(2i-n)k \hbar}{2} \del_s}
			\prod_{l = 0}^ \infty \frac{\bar{w} (s + \frac{-v- (2l + 1) k}{2} \hbar)}{\bar{w} (s + \frac{v- (2l + 1) k}{2} \hbar)}
			\\
			&\qquad=
			\sum_{n=0}^\infty \sum_{i = 0}^n \frac{(-1)^{n-i} e^{ nk \zeta} }{i!(n-i)! \hbar^n } 
			\del_s^t
			\prod_{j = 0}^ \infty \frac{\bar{w} ( s + \frac{ v- (2j + 1 - n) k}{2} \hbar)}{\bar{w} ( s + \frac{-v- (2j + 1 + n) k}{2} \hbar)}
			\prod_{l = 0}^ \infty \frac{\bar{w} ( s + \frac{-v- (2l + 1 + n - 2i) k}{2} \hbar)}{\bar{w} ( s + \frac{v- (2l + 1 + n - 2i) k}{2} \hbar)}
			\\ 						
			&\qquad=
			\sum_{n=0}^\infty \sum_{i = 0}^n \frac{(-1)^{n-i} e^{ nk \zeta} }{i!(n-i)! \hbar^n } 
			\del_s^t
			\prod_{l = 0}^{n-i-1} \bar{w} ( s + \frac{ v- (2l + 1 - n) k}{2} \hbar)
			\prod_{l = n-i}^{n-1} \bar{w} ( s + \frac{-v- (2l + 1 - n) k}{2} \hbar)
			\\
			&\qquad=
			\sum_{n=0}^\infty \sum_{i = 0}^n \frac{(-1)^{n-i} e^{ nk \zeta}}{i!(n-i)! \hbar^n } 
			\del_s^t
			\prod_{j = 0}^{i-1} \bar{w} ( s + \frac{-v+ (2j+1-n) k}{2} \hbar)
			\prod_{j = i}^{n-1} \bar{w} ( s + \frac{ v+ (2j+1-n) k}{2} \hbar) . \qedhere
	\end{align*}
\end{proof}

\subsection{The \texorpdfstring{$n=1$}{n=1} case}

In the case $n=1$, we can use the same proof idea, but we need to be more careful because of the extra term in \cref{DoubleTransform}.
\begin{equation}
		\begin{split}
			\tilde{\omega}_1 (z)
			&=
			-\sum_{j=0}^\infty (-d \frac{1}{dx})^j [v^j] \Bigg( 
			e^{v ( \mc{S}(v\hbar \del_{y}) \tilde{\psi}^\vee(y) - \psi^\vee (y) )} \frac{dy}{v\hbar}
			- \bigg(e^{v ( \mc{S}(v\hbar \del_y) \tilde{\psi}^\vee (y) - x )} \sum_{p=0}^\infty \big( -d \frac{1}{dy} \big)^p e^{vg} [b^p ] \bigg)
			\\
			& \qquad
			\cdot \Bigg( e^{b ( \mc{S}(b\hbar \del_g) \tilde{\chi}(g) - \chi (g) )} \frac{dg}{b\hbar}
			-\bigg(e^{b ( \mc{S}(b\hbar \del_g) \tilde{\chi}(g) - y )} \sum_{q=0}^\infty \big( -d \frac{1}{dg} \big)^q e^{bs} [a^q ] \bigg)
			\\
			&\qquad 
			\cdot \bigg( e^{a ( \mc{S}(a\hbar \del_s) \tilde{\psi}(s) - \psi (s) )} \frac{ds}{a\hbar}
			- e^{a ( \mc{S}(a\hbar \del_s) \tilde{\psi}(s) - g )} \sum_{t=0}^\infty \big( -d \frac{1}{ds} \big)^t e^{a\zeta} [u^t ] \mb{W}^{\mc{S}_0}_1 (z,u) 
			\bigg) \Bigg) \Bigg)
			\\
			&=
			\tilde{T}\mb{W}^{\mc{S}_0}_1(z)
			-\sum_{j,p,q=0}^\infty (-d \frac{1}{dx})^j [v^j] e^{v ( \mc{S}(v\hbar \del_{y}) \tilde{\psi}^\vee(y) - x)}
			\big( -d \frac{1}{dy} \big)^p e^{vg} [b^p a^q] 
			\\
			&\qquad \qquad
			\Bigg( \frac{dy}{v\hbar} - 
			e^{b ( \mc{S}(b\hbar \del_g) \tilde{\chi}(g) - y )} \big( -d \frac{1}{dg} \big)^q e^{bs}
			\bigg( \frac{dg}{b\hbar}- e^{a ( \mc{S}(a\hbar \del_s) \tilde{\psi}(s) - \psi (s) )} \frac{ds}{a\hbar} \bigg) \Bigg)
		\end{split}
	\end{equation}


We ensure that all objects involved in this equation are power series in $ a$, $ b$, and $v$, by writing

\begin{equation}\label{ExtraTermForn=1}
	\begin{split}
		\bar{\omega}_1 (z) 
		\coloneqq
		\hbar ( \tilde{\omega}_1 (z) - \tilde{T}\mb{W}^{\mc{S}_0}_1(z) )
		&=
		\sum_{j,p,q=0}^\infty (-d \frac{1}{dx})^j [v^j b^p a^q] 
		\Bigg( - \frac{e^{v ( \mc{S}(v\hbar \del_{y}) \tilde{\psi}^\vee(y) - \psi^\vee (y) )} - 1}{v} dy
    \\
    &\quad
    + 
    e^{v ( \mc{S}(v\hbar \del_{y}) \tilde{\psi}^\vee(y) - x)} \big( -d \frac{1}{dy} \big)^p 
		e^{vg} \bigg( \frac{e^{b ( \mc{S}(b\hbar \del_g) \tilde{\chi}(g) - \chi (g) )} -1}{b} dg 
    \\
    &\qquad \quad
    - 
    e^{b ( \mc{S}(b\hbar \del_g) \tilde{\chi}(g) - y )}  \big( -d \frac{1}{dg} \big)^q
    e^{bs} \frac{e^{a ( \mc{S}(a\hbar \del_s) \tilde{\psi}(s) - \psi (s) )} - 1 }{a} ds \bigg) \Bigg)
	\end{split}
\end{equation}

We will use the following equality: using that $ g = x - \psi^\vee (y)$ and $ s = y - \chi (g)$, for function $ F(x,v)$ and $ G(y,b)$ which are power series in $ v$, resp. $b$,

\begin{align}
  0
  &=
  \sum_{j=0}^\infty ( - d \frac{1}{dx} )^j [v^j] ( v +  d \frac{1}{dx}) e^{-v \psi^\vee (y) } F dx 
  =
  \sum_{j=0}^\infty ( - d \frac{1}{dx} )^j [v^j] e^{-v \psi^\vee (y)} \big( v F dg + dF \big)
  \label{IntByPartsx}
  \\
  0
  &=
  \sum_{p=0}^\infty ( - d \frac{1}{dy} )^p [b^p] e^{-b \chi (g)} \big( b G ds + dG \big) 
  \label{IntByPartsy}
\end{align}

\begin{lemma}
	With the assumption that $ \tilde{\psi}^\vee = - \tilde{\psi}$, the remaining term for $ n =1$ is
	\begin{equation}
		\begin{split}
			\bar{\omega}_1 (z)
      &=
      \sum_{j=0}^\infty (-d \frac{1}{dx})^j [v^j ] 
			\Bigg( 
      e^{-v \psi^\vee (y)}  \frac{e^{-\chi (g) \del_y} - 1}{\del_y} e^{v \psi^\vee (y )} dg 
      - 
      e^{v (\zeta - x)} s \, d\zeta
      \\
      &\quad
      + 
      e^{-xv}
      \int_{\sigma = 0}^s \Big\lfloor_{\eta \to \sigma} e^{ \mc{S}(\hbar \del_\sigma \del_\zeta) \tilde{\psi}(\sigma) \del_\zeta} e^{v\zeta} e^{ \mc{S}(\hbar \del_\zeta \del_\eta) \tilde{\chi}(\zeta) \del_\eta }
			e^{v \mc{S}(v\hbar \del_{\eta}) \tilde{\psi}^\vee(\eta) } d\sigma \wedge d\zeta
      \Bigg) .
		\end{split}
	\end{equation}
\end{lemma}

\begin{proof}
	We will use \cref{FormalIntByParts} twice. We first take, for the $b$-sum
	\begin{equation}
		A
		=
		e^{b(s-y)} [a^q ] 
		\frac{e^{a ( \mc{S}(a\hbar \del_s) \tilde{\psi}(s) - \psi (s) )} - 1}{a} ds \,
		\qquad
		B
		=
		e^{vg} e^{b \mc{S}(b\hbar \del_g) \tilde{\chi}(g) } .
	\end{equation}
	So
  \begin{equation}\label{FirstStepForn=1}
		\begin{split}
			\bar{\omega}_1 (z)
      &=
      \sum_{j,p=0}^\infty (-d \frac{1}{dx})^j [v^j b^p ] 
			\Bigg( - \frac{e^{v ( \mc{S}(v\hbar \del_{y}) \tilde{\psi}^\vee(y) - \psi^\vee (y) )} - 1}{v} dy
      \\
      &\quad
      + 
      e^{v ( \mc{S}(v\hbar \del_{y}) \tilde{\psi}^\vee(y) - x)} \big( -d \frac{1}{dy} \big)^p 
			\bigg( e^{vg} \frac{e^{b ( \mc{S}(b\hbar \del_g) \tilde{\chi}(g) - \chi (g) )} -1}{b} dg 
      \\
      &\qquad \quad
      - 
      e^{-b\chi(g)} \Big\lfloor_{\gamma \to g} \frac{e^{( \mc{S}(\hbar \del_s \del_\gamma) \tilde{\psi}(s) - \psi (s) )\del_\gamma} - 1}{\del_\gamma} e^{v\gamma} e^{b \mc{S}(b\hbar \del_\gamma) \tilde{\chi}(\gamma) }
			ds \bigg) \Bigg) .
		\end{split}
	\end{equation}
	We use \cref{IntByPartsy} with $ G = \frac{e^{-\psi (s) \del_g} - 1}{b \del_g} e^{vg} (e^{b \mc{S}(b \hbar \del_g) \tilde{\chi} (g)} - 1)$ to get
	\begin{equation}
    \begin{split}
        0
        &=
        \sum_{p=0}^\infty ( - d \frac{1}{dy} )^p [b^p] e^{-b \chi (g)} \Big(
        \frac{e^{-\psi (s) \del_g} - 1}{\del_g} e^{vg} (e^{b \mc{S}(b \hbar \del_g) \tilde{\chi} (g)} - 1) ds
        \\
        &\qquad \qquad
        +
        \frac{e^{-\psi (s) \del_g} - 1}{b} e^{vg} (e^{b \mc{S}(b \hbar \del_g) \tilde{\chi} (g)} - 1) dg
        -
        \psi'(s) \frac{e^{-\psi (s) \del_g}}{b} e^{vg} (e^{b \mc{S}(b \hbar \del_g) \tilde{\chi} (g)} - 1)ds \Big)
        \\
        &=
        \sum_{p=0}^\infty ( - d \frac{1}{dy} )^p [b^p] e^{-b \chi (g)} \Big(
        \frac{e^{-\psi (s) \del_g} - 1}{\del_g} e^{vg} (e^{b \mc{S}(b \hbar \del_g) \tilde{\chi} (g)} - 1) ds 
       	\\
       	&\qquad \qquad
        -
        e^{vg} \frac{e^{b \mc{S}(b \hbar \del_g) \tilde{\chi} (g)} - 1}{b} dg
        +
        e^{-\psi (s) \del_g}e^{vg} \frac{e^{b \mc{S}(b \hbar \del_g) \tilde{\chi} (g)} - 1}{b} d\zeta
        \Big) ,
    \end{split}
	\end{equation}
	and we add this to \cref{FirstStepForn=1} to find
	\begin{equation}
		\begin{split}
			\bar{\omega}_1 (z)
      &=
      \sum_{j,p=0}^\infty (-d \frac{1}{dx})^j [v^j b^p ] 
			\Bigg(
      - \frac{e^{v ( \mc{S}(v\hbar \del_{y}) \tilde{\psi}^\vee(y) - \psi^\vee (y) )} - 1}{v} dy
      \\
      &\quad
      + 
      e^{v ( \mc{S}(v\hbar \del_{y}) \tilde{\psi}^\vee(y) - x)} \big( -d \frac{1}{dy} \big)^p 
			\bigg( e^{vg} \frac{e^{-b  \chi (g)} -1}{b} dg 
			+
			e^{-b \chi (g)} e^{-\psi (s) \del_g}e^{vg} \frac{e^{b \mc{S}(b \hbar \del_g) \tilde{\chi} (g)} - 1}{b} d\zeta
      \\
      &\qquad \quad
      - 
      e^{-b\chi(g)} \Big\lfloor_{\gamma \to g} \Big(  \frac{e^{-\psi (s) \del_\gamma} - 1}{\del_\gamma} e^{v\gamma} 
      + e^{-\psi (s) \del_\gamma} \frac{e^{ \mc{S}(\hbar \del_s \del_\gamma) \tilde{\psi}(s) \del_\gamma} - 1}{\del_\gamma} e^{v\gamma} e^{b \mc{S}(b\hbar \del_\gamma) \tilde{\chi}(\gamma) }
			\Big) ds \bigg) \Bigg) .
		\end{split}
	\end{equation}
	Let us now again use \cref{FormalIntByParts} for
	\begin{equation}
		\begin{split}
			A'
			&=
			e^{-vx} 
			\bigg( e^{vg} \frac{ e^{-b \chi (g)} -1 }{b} dg 
    	+
    	e^{-b \chi (g)}
    	e^{-\psi (s) \del_g}e^{vg} \frac{e^{b \mc{S}(b \hbar \del_g) \tilde{\chi} (g)} - 1}{b} d\zeta
    	\\
			& \qquad
			+ 
			e^{-b \chi (g)} \Big\lfloor_{\gamma \to g} \Big(
      - \frac{e^{-\psi (s) \del_\gamma} - 1}{\del_\gamma} e^{v\gamma} 
      - 
    	e^{-\psi (s) \del_\gamma} \frac{e^{ \mc{S}(\hbar \del_s \del_\gamma) \tilde{\psi}(s) \del_\gamma} - 1}{\del_\gamma} e^{v\gamma} e^{b \mc{S}(b\hbar \del_\gamma) \tilde{\chi}(\gamma) }
      \Big)
			ds \bigg)
			\\
			B'
			&=
			e^{v \mc{S}(v\hbar \del_{y}) \tilde{\psi}^\vee(y) }
		\end{split}
	\end{equation}
	to obtain
	\begin{align}
			\bar{\omega}_1 (z)
      &=
      \sum_{j=0}^\infty (-d \frac{1}{dx})^j [v^j ] \Big\lfloor_{\substack{\gamma \to g \\ \eta \to y}}
			\Bigg(
      - \frac{e^{v ( \mc{S}(v\hbar \del_{y}) \tilde{\psi}^\vee(y) - \psi^\vee (y) )} - 1}{v} dy
      + 
      e^{-vx} 
			\bigg( e^{vg} \frac{ e^{- \chi (g) \del_\eta} -1 }{\del_\eta} dg 
			\nonumber
			\\
			&\qquad
			+
    	e^{- \chi (g) \del_\eta}
    	e^{-\psi (s) \del_\gamma}e^{v\gamma} \frac{e^{ \mc{S}( \hbar \del_\gamma \del_\eta) \tilde{\chi} (\gamma) \del_\eta} - 1}{\del_\eta} d\zeta
			+ 
			e^{- \chi (g) \del_\eta} \Big(
      - \frac{e^{-\psi (s) \del_\gamma} - 1}{\del_\gamma} e^{v\gamma} 
      \nonumber
      \\
      &\qquad
      - 
    	e^{-\psi (s) \del_\gamma} \frac{e^{ \mc{S}(\hbar \del_s \del_\gamma) \tilde{\psi}(s) \del_\gamma} - 1}{\del_\gamma} e^{v\gamma} e^{ \mc{S}(\hbar \del_\gamma \del_\eta) \tilde{\chi}(\gamma) \del_\eta }
      \Big)
			ds \bigg)
			e^{v \mc{S}(v\hbar \del_{\eta}) \tilde{\psi}^\vee(\eta) }
      \Bigg) 
      \nonumber
      \\
      &=
      \sum_{j=0}^\infty (-d \frac{1}{dx})^j [v^j ] \Big\lfloor_{\substack{\gamma \to g \\ \eta \to y}}
			\Bigg(
      - \frac{e^{v ( \mc{S}(v\hbar \del_{y}) \tilde{\psi}^\vee(y) - \psi^\vee (y) )} - 1}{v} dy
      + 
      e^{-v \psi^\vee (y)} \frac{ e^{- \chi (g) \del_\eta} -1 }{\del_\eta} e^{v \mc{S}(v\hbar \del_{\eta}) \tilde{\psi}^\vee(\eta) } dg
      \nonumber
			\\
			&\quad
			+
    	e^{-vx} e^{- \chi (g) \del_\eta}
			\bigg( 
    	e^{-\psi (s) \del_\gamma}e^{v\gamma} \frac{e^{ \mc{S}( \hbar \del_\gamma \del_\eta) \tilde{\chi} (\gamma) \del_\eta} - 1}{\del_\eta} e^{v \mc{S}(v\hbar \del_{\eta}) \tilde{\psi}^\vee(\eta) } d\zeta
    	\label{SecondStepForn=1}
    	\\
    	&\qquad
			- 
			\Big(
      \frac{e^{-\psi (s) \del_\gamma} - 1}{\del_\gamma} e^{v\gamma} 
      + 
    	e^{-\psi (s) \del_\gamma} \frac{e^{ \mc{S}(\hbar \del_s \del_\gamma) \tilde{\psi}(s) \del_\gamma} - 1}{\del_\gamma} e^{v\gamma} e^{ \mc{S}(\hbar \del_\gamma \del_\eta) \tilde{\chi}(\gamma) \del_\eta }
      \Big) e^{v \mc{S}(v\hbar \del_{\eta}) \tilde{\psi}^\vee(\eta) }
			ds \bigg)
      \Bigg) .
      \nonumber
	\end{align}
	We now use \cref{IntByPartsx} with $ F = \frac{e^{-\chi (g) \del_y} - 1}{v \del_y} (e^{v \mc{S}(v \hbar \del_y) \tilde{\psi}^\vee (y)} - e^{v \psi^\vee (y)})$ to get
	\begin{equation}
    \begin{split}
        0
        &=
        \sum_{j=0}^\infty ( - d \frac{1}{dx} )^j [v^j] e^{-v \psi^\vee (y)} \Big(
        \frac{e^{-\chi (g) \del_y} - 1}{\del_y} (e^{v \mc{S}(v \hbar \del_y) \tilde{\psi}^\vee (y)} - e^{v \psi^\vee (y)}) dg \\
        &\qquad \qquad
        -
        \frac{e^{v \mc{S}(v \hbar \del_y) \tilde{\psi}^\vee (y)} - e^{v \psi^\vee (y)}}{v} dy
        +
        e^{-\chi (g) \del_y} \frac{e^{v \mc{S}(v \hbar \del_y) \tilde{\psi}^\vee (y)} - e^{v \psi^\vee (y)}}{v} ds
        \Big)
    \end{split}
	\end{equation}
	and we subtract this from \cref{SecondStepForn=1} to get
	\begin{equation}
		\begin{split}
			\bar{\omega}_1 (z)
            &=
            \sum_{j=0}^\infty (-d \frac{1}{dx})^j [v^j ]
			\Bigg(
            e^{-v \psi^\vee (y)} \frac{ e^{- \chi (g) \del_y} -1 }{\del_y} e^{v \psi^\vee(y) } dg
			\\
			&\quad
			+
    	    e^{-vx} \Big\lfloor_{\substack{\gamma \to g \\ \eta \to y}} e^{- \chi (g) \del_\eta}
			\bigg( 
    	    e^{-\psi (s) \del_\gamma}e^{v\gamma} \frac{e^{ \mc{S}( \hbar \del_\gamma \del_\eta) \tilde{\chi} (\gamma) \del_\eta} - 1}{\del_\eta} e^{v \mc{S}(v\hbar \del_{\eta}) \tilde{\psi}^\vee(\eta) } d\zeta
    	    \\
    	    &\qquad
    	    +
            \frac{ e^{\psi^\vee (\eta) \del_\gamma} -e^{-\psi (s) \del_\gamma} e^{ \mc{S}(\hbar \del_{\eta} \del_\gamma) \tilde{\psi}^\vee(\eta) \del_\gamma} }{\del_\gamma} e^{v \gamma} ds
            \\
            &\qquad
            - 
    	    e^{-\psi (s) \del_\gamma} \frac{e^{ \mc{S}(\hbar \del_s \del_\gamma) \tilde{\psi}(s) \del_\gamma} - 1}{\del_\gamma} e^{v\gamma} e^{ \mc{S}(\hbar \del_\gamma \del_\eta) \tilde{\chi}(\gamma) \del_\eta }
            e^{v \mc{S}(v\hbar \del_{\eta}) \tilde{\psi}^\vee(\eta) }
			ds \bigg)
            \Bigg) .
		\end{split}
	\end{equation}

	Using that $ \psi^\vee = -\psi$, so that
	\begin{align*}
		\Big\lfloor_{\eta \to y} e^{-\chi(g) \del_\eta} e^{ \psi^\vee (\eta) \del_\gamma} f
		&=
		\Big\lfloor_{\eta \to y} e^{\psi^\vee (\eta- \chi(g)) \del_\gamma} e^{-\chi(g)\del_\eta} f
		=
		\Big\lfloor_{\eta \to y} e^{- \psi (s) \del_\gamma} e^{-\chi(g)\del_\eta} f
		=
		\Big\lfloor_{\eta \to y} e^{-\chi(g)\del_\eta} e^{- \psi (s) \del_\gamma} f,
		\\
		\Big\lfloor_{\substack{\gamma \to g\\ \eta \to y}} e^{-\chi(g)\del_\eta} e^{- \psi (s) \del_\gamma} f
		&=
		\Big\lfloor_{\substack{\gamma \to g - \psi(s) \\ \eta \to y - \chi(g)}} f
		=
		\Big\lfloor_{\substack{\gamma \to \zeta \\ \eta \to s}} f,
	\end{align*}
	we get, also using $ \tilde{\psi}^\vee = - \tilde{\psi}$ in the second equality
	\begin{align*}
		\bar{\omega}_1 (z)
        &=
        \sum_{j=0}^\infty (-d \frac{1}{dx})^j [v^j ] 
		\Bigg( 
        e^{-v \psi^\vee (y)}  \frac{e^{-\chi (g) \del_y} - 1}{\del_y} e^{v \psi^\vee (y )} dg 
        + 
        e^{-vx}
        \Big\lfloor_{\substack{\gamma \to \zeta \\ \eta \to s}} \bigg( 
        e^{v\gamma} \frac{e^{ \mc{S}( \hbar \del_\gamma \del_\eta) \tilde{\chi} (\gamma) \del_\eta} - 1}{\del_\eta} e^{v \mc{S}(v\hbar \del_{\eta}) \tilde{\psi}^\vee(\eta) } d\zeta
        \\
		& \qquad
		+
		\Big(
		\frac{ 1 - e^{ \mc{S}(\hbar \del_{\eta} \del_\gamma) \tilde{\psi}^\vee(\eta) \del_\gamma} }{\del_\gamma} e^{v\gamma}
        - 
        \frac{e^{ \mc{S}(\hbar \del_s \del_\gamma) \tilde{\psi}(s) \del_\gamma} - 1}{\del_\gamma} e^{v\gamma} e^{ \mc{S}(\hbar \del_\gamma \del_\eta) \tilde{\chi}(\gamma) \del_\eta }
		e^{v \mc{S}(v\hbar \del_{\eta}) \tilde{\psi}^\vee(\eta) }
        \Big)
 		ds \bigg) \Bigg) 
 		\tag{\stepcounter{equation}\theequation}
        \\
        &=
        \sum_{j=0}^\infty (-d \frac{1}{dx})^j [v^j ] 
		\Bigg( 
        e^{-v \psi^\vee (y)}  \frac{e^{-\chi (g) \del_y} - 1}{\del_y} e^{v \psi^\vee (y )} dg 
        \\
        &\quad
        + 
        e^{-vx}
        \bigg( 
        \int_{\eta = 0}^s  e^{v\zeta} (e^{ \mc{S}(\hbar \del_\zeta \del_\eta) \tilde{\chi}(\zeta) \del_\eta } - 1)
		e^{v \mc{S}(v\hbar \del_{\eta}) \tilde{\psi}^\vee(\eta) } d\eta \wedge d\zeta
        \\
		& \qquad
		+
		\Big\lfloor_{\eta \to s}
        \frac{e^{ \mc{S}(\hbar \del_s \del_\zeta  ) \tilde{\psi}(s) \del_\zeta } - 1}{\del_\zeta } e^{v\zeta} (1 - e^{ \mc{S}(\hbar \del_\zeta \del_\eta) \tilde{\chi}(\zeta) \del_\eta } )
		e^{v \mc{S}(v\hbar \del_{\eta}) \tilde{\psi}^\vee(\eta) }
 		ds \bigg) \Bigg) .
	\end{align*}
	Now, we use that $\sum_{j=0}^\infty (-d \frac{1}{dx})^j [v^j ] e^{-vx} dI = 0 $ for any $ I$, for the function
	\begin{equation}
		I
		=
		\int_{\sigma=0}^{s} \Big\lfloor_{\eta \to \sigma} \frac{e^{ \mc{S}(\hbar \del_\sigma \del_\zeta  ) \tilde{\psi}(\sigma) \del_\zeta} - 1}{\del_\zeta } e^{v \zeta} (1 - e^{ \mc{S}(\hbar \del_\zeta \del_\eta) \tilde{\chi}(\zeta) \del_\eta } )
			e^{v \mc{S}(v\hbar \del_{\eta}) \tilde{\psi}^\vee(\eta) }
 			d\sigma
	\end{equation}
	whose derivative is
	\begin{equation*}
		\begin{split}
			dI
			&=
			\int_{\sigma = 0}^s \Big\lfloor_{\eta \to \sigma} (e^{ \mc{S}(\hbar \del_\sigma \del_\zeta) \tilde{\psi}(\sigma) \del_\zeta} - 1) e^{v\zeta} (1 - e^{ \mc{S}(\hbar \del_\zeta \del_\eta) \tilde{\chi}(\zeta) \del_\eta } )
			e^{v \mc{S}(v\hbar \del_{\eta}) \tilde{\psi}^\vee(\eta) } d\sigma \wedge d\zeta
			\\
			&\quad
			+
			\Big\lfloor_{\eta \to s} \frac{e^{ \mc{S}(\hbar \del_s \del_\zeta) \tilde{\psi}(s) \del_\zeta} - 1}{\del_\zeta} e^{v\zeta} (1 - e^{ \mc{S}(\hbar \del_\zeta \del_\eta) \tilde{\chi}(\zeta) \del_\eta } )
			e^{v \mc{S}(v\hbar \del_{\eta}) \tilde{\psi}^\vee(\eta) } ds .
		\end{split}
	\end{equation*}
	This yields
	\begin{align*}
			\bar{\omega}_1 (z)
      &=
      \sum_{j=0}^\infty (-d \frac{1}{dx})^j [v^j ] 
			\Bigg( 
      e^{-v \psi^\vee (y)}  \frac{e^{-\chi (g) \del_y} - 1}{\del_y} e^{v \psi^\vee (y )} dg 
      \\
      &\quad
      + 
      e^{-vx}
      \int_{\sigma = 0}^s \Big\lfloor_{\eta \to \sigma} e^{ \mc{S}(\hbar \del_\sigma \del_\zeta) \tilde{\psi}(\sigma) \del_\zeta} e^{v\zeta} (e^{ \mc{S}(\hbar \del_\zeta \del_\eta) \tilde{\chi}(\zeta) \del_\eta } -1  )
			e^{v \mc{S}(v\hbar \del_{\eta}) \tilde{\psi}^\vee(\eta) } d\sigma \wedge d\zeta
      \Bigg) 
      \\
      &=
      \sum_{j=0}^\infty (-d \frac{1}{dx})^j [v^j ] 
			\Bigg( 
      e^{-v \psi^\vee (y)}  \frac{e^{-\chi (g) \del_y} - 1}{\del_y} e^{v \psi^\vee (y )} dg
      - 
      e^{-vx}
      \int_{\sigma = 0}^s e^{v\zeta} d\sigma \wedge d\zeta
      \\
      &\quad
      + 
      e^{-vx}
      \int_{\sigma = 0}^s \Big\lfloor_{\eta \to \sigma} e^{ \mc{S}(\hbar \del_\sigma \del_\zeta) \tilde{\psi}(\sigma) \del_\zeta} e^{v\zeta} e^{ \mc{S}(\hbar \del_\zeta \del_\eta) \tilde{\chi}(\zeta) \del_\eta } 
			e^{v \mc{S}(v\hbar \del_{\eta}) \tilde{\psi}^\vee(\eta) } d\sigma \wedge d\zeta
      \Bigg) . \qedhere
	\end{align*}
\end{proof}

\subsection{The final proof}

\begin{proof}[Proof of \cref{LeakyTRThm}]
	We want to prove that the $ \omega_{g,n}$ of \cref{LeakyMultidifferentials} satisfy topological recursion on the spectral curve \eqref{SpectralCurveParam}. We know that the $ \tilde{\omega}_{g,n}$ of \cref{DualitiesPreserveTR,ExtraTermForn=1} do, so we want to prove that these two collections of multidifferentials are equal. We already know that to prove $ \omega_{g,n} = \tilde{\omega}_{g,n}$, it is sufficient to prove that $ T = \tilde{T}$, and that $ \bar{\omega}_1(z) $ equals the extra term of \cref{FormulaForn=1}. For the equality $ T = \tilde{T}$, by \cref{VrIsMultiplication,SubstitutionOfDerivatives}, we see that it suffices to prove that for any $ t \geq 0$

	\begin{equation}\label{LAsExponentials}
		\begin{split}
			&e^{v (\zeta - x)}
			\del_s^t \Big\lfloor_{\bar{\sigma} \to s}
			e^{ (\del_\zeta +v) \mc{S}( \hbar (\del_\zeta + v) \del_{\bar{\sigma}} )\tilde{\psi} (\bar{\sigma}) }
			e^{ \del_s \mc{S}( \hbar \del_\zeta \del_s ) \tilde{\chi} (\zeta) }
			e^{v \mc{S}(v \hbar \del_s) \tilde{\psi}^\vee (s)}
			\\
			&\qquad =
			L_t(-v, z, s, \hbar)
			=
			\bar{\phi}_{-v}(z,s)^{-1} \del_s^t \phi_{-v}(z, s, 1)
			=
			\Big( \frac{z}{e^x} \Big)^{v} \del_s^t \phi_{-v}(z, s, 1) ,
		\end{split}
	\end{equation}
	where $s$ is still considered as formally independent of $ z$.\par
	Because $ e^{v(\zeta - x)} = \big( \frac{z}{e^x} \big)^v$, it suffices to prove
	\begin{equation}\label{phiAsExponentials}
		\phi_{-v}(z, s, 1)
		=
		\Big\lfloor_{\bar{\sigma} \to s}
		e^{ (\del_\zeta +v) \mc{S}( \hbar (\del_\zeta + v) \del_{\bar{\sigma}} )\tilde{\psi} (\bar{\sigma}) }
		e^{ \del_s \mc{S}( \hbar \del_\zeta \del_s ) \tilde{\chi} (\zeta) }
		e^{v \mc{S}(v \hbar \del_s) \tilde{\psi}^\vee (s)} .
	\end{equation}

	The right side is given by \cref{ExplicitDualityChain}, while the left is given by \cref{phi_m,phi_mn}:

	\begin{equation*}
		\begin{split}
			\phi_{-v} ( z, s, 1) 
			&=
			\sum_{n=0}^\infty z^{kn} \sum_{i=0}^n \frac{(-1)^{n-i}}{i!(n-i)!} \prod_{j=0}^{i-1} w \big( \frac{s}{\hbar} + \frac{-v + (2j+1-n)k}{2} \big) \prod_{j=i}^{n-1} w \big( \frac{s}{\hbar} + \frac{v + (2j+1-n)k}{2} \big) .
		\end{split}
	\end{equation*}

	Using that $ w(\frac{s}{\hbar}) = \frac{1}{\hbar}\bar{w}(s)$ and $ z = e^\zeta $, we get the desired equality.\par
	For the extra term for $ n =1$, we consider \cref{LAsExponentials} for $ t = 0$:
	\begin{equation}
		\begin{split}
			L_0(-v, z, s, \hbar)
			&=
			e^{v(\zeta - x)} \Big\lfloor_{\bar{\sigma} \to s}
			e^{ (\del_\zeta +v) \mc{S}( \hbar (\del_\zeta + v) \del_{\bar{\sigma}} )\tilde{\psi} (\bar{\sigma}) }
			e^{ \del_s \mc{S}( \hbar \del_\zeta \del_s ) \tilde{\chi} (\zeta) }
			e^{v \mc{S}(v \hbar \del_s) \tilde{\psi}^\vee (s)} 
			\\
			&=
			e^{-v x}
			\Big\lfloor_{\substack{\gamma \to \zeta \\ \sigma \to s}} \Big\lfloor_{\eta \to \sigma}
			e^{ \del_\gamma \mc{S}( \hbar \del_\gamma \del_\sigma )\tilde{\psi} (\sigma ) }
			e^{v \gamma} e^{ \del_\eta \mc{S}( \hbar \del_\gamma \del_\eta)\tilde{\chi} (\gamma) }
			e^{v \mc{S}(v \hbar \del_\eta) \tilde{\psi}^\vee (\eta)}
		\end{split}
	\end{equation}
	so that indeed
	\begin{equation}
		\int_{s = 0}^{S(z)} \!\! L_0 (-v, z, s, \hbar ) ds \wedge d\zeta
		=
		e^{-xv} \int_{\sigma=0}^{s} \Big\lfloor_{\eta \to \sigma} \!\!
			e^{ \del_\zeta \mc{S}( \hbar \del_\zeta \del_\sigma )\tilde{\psi} (\sigma ) }
			e^{v \zeta} e^{ \del_\eta \mc{S}( \hbar \del_\zeta \del_\eta)\tilde{\chi} (\zeta) }
			e^{v \mc{S}(v \hbar \del_\eta) \tilde{\psi}^\vee (\eta)} d\sigma \wedge d\zeta .
	\end{equation}
\end{proof}

\section{Quasi--polynomiality of completed cycles leaky Hurwitz numbers}
\label{QuasiPoly}
Since the first version appeared on the arXiv, Wang and Yang studied the structural behaviour of leaky Hurwitz numbers in \cite{wang2026quasi}. In this section, we demonstrate how one of their results, namely the quasipolynomiality of single completed cycles leaky Hurwitz numbers, follows from our results.  In fact, our
approach extends to arbitrary orbifold ramification. More precisely, we prove the following:

\begin{theorem}
\label{thm-quasipoly}
    The connected $(k,r,q)$--leaky Hurwitz numbers $h_{g,\mu}^{k,r,q,\circ}$ are quasipolynomial in the entries of $\mu=(\mu_1,\dots,\mu_n)$ when $2g-2+n>0$. More precisely, recall the notation $B=q(r-1)+k$ and let $s=\frac{q(2g-2+n)+|\mu|}{B}$. We write
    \begin{equation}
        \mu_i=Bm_i+\ell_i\quad\textrm{for}\quad 0\le \ell_i<B.
    \end{equation}
    Then, for each fixed residue vector $\underline{\ell}=(\ell_1,\dots,\ell_n)$, there exists a polynomial $P_{g,n,\underline{\ell}}^{k,r,q}\in\mathbb{Q}[\mu_1,\dots,\mu_n]$ of degree at most $3g-3+n$, such that
    \begin{equation}
        h_{g,\mu}^{k,r,q,\circ}=s!\prod_{i=1}^n\frac{(r\mu_i-k)_{m_i,-k}}{Bm_i!r^{m_i}}P_{g,n,\underline{\ell}}^{k,r,q}(\mu_1,\dots,\mu_n).
    \end{equation}
    Here $(a)_{m,k}$ denotes the Pochhammer $k$-symbol.  
\end{theorem}

\begin{remark}
    The statement in \cref{thm-quasipoly} for $q=1$ is proved as a consequence of \cite[Theorem 1.1]{wang2026quasi}. In \emph{loc.cit.} a stronger statement is derived, namely for $n=1$ it is proved that $P_{g,1,\underline{\ell}}^{k,r,q}(\mu_1,\dots,\mu_n)$ satisfies an additional factorisation that may be absorbed to a simplified pre-factor.
\end{remark}

\begin{proof}
    We proved in \cref{LeakyTRLimit} that connected $(k,r,q)$--leaky Hurwitz numbers satisfy topological recursion. Recall that their spectral curve is given by
    \begin{equation}
        X(z)=z(1+\frac{k}{r}z^{B})^{-\frac{r}{k}},\quad y(z)=z^q(1+\frac{k}{r}z^{B}
    \end{equation}
    and
    \begin{equation}
        dx=\frac{1-(B-\frac{k}{r})z^B}{z(1+\frac{k}{r}z^B)}dz.
    \end{equation}

    Let $a$ be a zero of $dx(z)$, i.e. $1-(B-\frac{k}{r})a^B=0$ and thus $R=\{a\mid a^B=\frac{1}{1-B-\frac{k}{r}}\}$. We have 
    \begin{equation}
        1-(\frac{z}{a})^B=(1-\frac{z}{a})(\sum_{j=0}^{B-1} (\frac{z}{a})^j),
    \end{equation}
    therefore
    \begin{equation}
        (z-a)=\frac{a(1-\frac{z}{a})^B}{\sum_{j=0}^{B-1} (\frac{z}{a})^j}
    \end{equation}
    and finally
    \begin{equation}
        \frac{1}{z-a}=-\frac{\sum_{j=0}^{B-1} (\frac{z}{a})^j}{a(1-\frac{z}{a})^B}=-\sum_{j=1}^{B-1}\frac{a^{-j-1}z^j}{1-(B-\frac{k}{r})z^B}.
    \end{equation}

    We define $\alpha_j=\frac{z^j}{1-(B-\frac{k}{r})z^B}$ and compute for $\mu=Bm+j$ that
    \begin{equation}
        [X^\mu]\alpha_j=\mathrm{Res}_{X=0}\alpha_jX^{-(\mu+1)}dX=\mathrm{Res}_{z=0}z^{j-\mu-1}(1+\frac{k}{r}z^B)^{\mu\frac{r}{k}-1}dz=\frac{(r\mu-k)_{m,-k}}{m!r^m}.
    \end{equation}

    The same computation yields $0$ if $\mu\not\equiv j\,\mathrm{mod}\, \mu$

    Together with \cref{cor-polyfromTR} this finishes the proof.
\end{proof}


\printbibliography

\end{document}